\documentclass[12pt, a4paper,reqno]{amsart}
\usepackage{amscd,amssymb,amsmath, array,latexsym,amsbsy,amsrefs, hyperref,graphicx}
\usepackage[shortlabels]{enumitem}
\usepackage[top=4cm,bottom=4cm,outer=2.3cm, inner=2.3cm, marginparwidth=1.8cm, marginparsep=0.2cm]{geometry}
\usepackage[active]{srcltx}
\usepackage{color}
\usepackage{microtype}
\usepackage{verbatim}
\allowdisplaybreaks[2]

\def\Prim{\operatorname{Prim}}

\def\N{\mathcal{N}} 
\def\C{\mathbb{C}}
\def\M{\mathbb{M}}

\def\R{\mathbb{R}}

\def\NN{\mathbb{N}}

\newtheorem{thm}{Theorem}[section]
\newtheorem{cor}[thm]{Corollary}
\newtheorem{prop}[thm]{Proposition}
\newtheorem{lemma}[thm]{Lemma}
\theoremstyle{definition}
\newtheorem{defn}[thm]{Definition}
\newtheorem{remark}[thm]{Remark}
\newtheorem{question}[thm]{Question}
\newtheorem{example}[thm]{Example}
\numberwithin{equation}{section}

\newcounter{desccount}
\newcommand{\descitem}[1]{
	\item[#1] \refstepcounter{desccount}\label{#1}
}
\newcommand{\descref}[1]{\hyperref[#1]{#1}}

\title[The Dixmier property and tracial states]{The Dixmier property and tracial states for $C^*$-algebras.}

\thanks{A.T.\ was partially supported by an NSERC Postdoctoral Fellowship and through the EPSRC grant EP/N00874X/1.}

\author
[Archbold]{Robert Archbold}
\address{Robert Archbold
\\Institute of Mathematics
\\University of Aberdeen
\\King's College
\\Aberdeen AB24 3UE
\\Scotland
\\United Kingdom
} \email{r.archbold@abdn.ac.uk}

\author
[Robert]{Leonel Robert}
\address{Leonel Robert
\\Department of Mathematics
\\University of Louisiana at Lafayette
\\Lafayette, 70504-3568
\\USA
} \email{lrobert@louisiana.edu}

\author
[Tikuisis]{Aaron Tikuisis}
\address{Aaron Tikuisis
\\Institute of Mathematics
\\University of Aberdeen
\\King's College
\\Aberdeen AB24 3UE
\\Scotland
\\United Kingdom
} \email{a.tikuisis@abdn.ac.uk}


\thanks{}

\begin{document}

\begin{abstract}
It is shown that a unital $C^*$-algebra $A$ has the Dixmier property if and only if it is weakly central and satisfies certain tracial conditions. This generalises the Haagerup--Zsid\'o theorem for simple $C^*$-algebras.
We also study a uniform version of the Dixmier property, as satisfied for example by von Neumann algebras and the reduced $C^*$-algebras of Powers groups, but not by all $C^*$-algebras with the Dixmier property, and we obtain necessary and sufficient conditions for a simple unital $C^*$-algebra with unique tracial state to have this uniform property.
We give further examples of $C^*$-algebras with the uniform Dixmier property, namely all $C^*$-algebras with the Dixmier property and finite radius of comparison-by-traces. Finally, we determine the distance between two Dixmier sets, in an arbitrary unital $C^*$-algebra, by a formula involving tracial data and algebraic numerical ranges.
\end{abstract}

\maketitle

\thanks{}

\bigskip

\section{Introduction}

\bigskip

Let $A$ be a unital $C^*$-algebra with unitary group ${\mathcal U}(A)$ and centre $Z(A)$. For $a\in A$, the \emph{Dixmier set}
$D_A(a)$ is the norm-closed convex hull of the set $\{uau^* : u\in {\mathcal U}(A)\}$. Then, acting by conjugation, ${\mathcal U}(A)$ induces a group of isometric affine transformations of the convex set $D_A(a)$, and this group of transformations has a common fixed point if and only if
$D_A(a) \cap Z(A)$ is non-empty. The $C^*$-algebra $A$ is said to have the \emph{Dixmier property} if $D_A(a) \cap Z(A)$ is non-empty for all $a\in A$, and $A$ is said to have the \emph{singleton Dixmier property} if $D_A(a) \cap Z(A)$ is a singleton set for all $a\in A$.

In \cite{Dix1949}, it was shown that every von Neumann algebra has the Dixmier property and an example was given of a unital $C^*$-algebra for which the Dixmier property does not hold. Since then, there has been an extensive literature, studying variants of the averaging process and the form of the subsets of $Z(A)$ obtained, and also giving several applications to a number of topics including centre-valued traces, commutators, derivations, $C^*$-simplicity, relative commutants, commutation in tensor products, and the study of masas and subalgebras of finite index in von Neumann algebras.  See 
\cite{AkJo}--\cite{RJA}, \cite{BedosJOT}, \cite{BedosCAMB}, \cite{Choi}, \cite{Conway}--\cite{delaHarpeSkandalis}, \cite{Dix1949}, \cite{DixvN}, \cite{Goldman}, \cite{Green}, \cite{Haagerup}--\cite{Halpern86}, \cite{BEJ}--\cite{KadRing}, \cite{Kennedy}, \cite{Magajna}, \cite{Misonou2}, \cite{NRS}--\cite{Ozawa}, \cite{Popa_vN}--\cite{RingPLMS84}, \cite{Schwartz}--\cite{Skoufranis}, \cite{StratilaZsido}, \cite{Zsido}
and the references cited therein.

In \cite{HZ}, Haagerup and Zsid\'o established a definitive result about the Dixmier property for simple $C^*$-algebras: a simple unital $C^*$-algebra has the Dixmier property if and only if it has at most one tracial state.
For non-simple $C^*$-algebras, the Dixmier property imposes serious restrictions on the ideal structure: if a $C^*$-algebra has the Dixmier property, then it is weakly central (\cite[p.\ 275]{RJA}), see Definition \ref{weakly-central-def}.
One of our main results is a complete generalisation of Haagerup and Zsid\'o's, showing that the Dixmier property is equivalent to this ideal space restriction together with tracial conditions:

\begin{thm} [Theorem \ref{DP-characterisation}]
\label{MainThmIntro}
Let $A$ be a unital $C^*$-algebra.
Then $A$ has the Dixmier property if and only if all of the following hold.
\begin{enumerate}[label=\emph{(\roman*)}]
\item $A$ is weakly central,
\item every simple quotient of $A$ has at most one tracial state, and
\item every extreme tracial state of $A$ factors through some simple quotient.
\end{enumerate}
\end{thm}

A characterisation of the singleton Dixmier property is an immediate consequence of this result (Corollary \ref{cor:SDP1}): it corresponds to the case that in (ii), every simple quotient has exactly one tracial state.
We also take the opportunity to remove a separability condition from a result in \cite{Arch-thesis}: a postliminal $C^*$-algebra $A$ has the (singleton) Dixmier property if and only if $Z(A/J) = (Z(A)+J)/J$ for every proper closed ideal $J$ of $A$ (Theorem~\ref{thesis theorem}).

The case of trivial centre in Theorem~\ref{MainThmIntro} is already an interesting generalisation of the Haagerup--Zsid\'o theorem:  a unital $C^*$-algebra $A$ has the Dixmier property with centre $Z(A)=\C1$ if and only if $A$ has a unique maximal ideal $J$, $A$ has at most one tracial state and $J$ has no tracial states (Corollary~\ref{genHZ1}). This result has a crucial application in Section \ref{sec:UDP} (see below).

\medskip

In Section \ref{sec:UDP}, we consider a strengthening of the Dixmier property, called the \emph{uniform Dixmier property}, in which the number of unitaries used to approximately average an element depends only on the tolerance (and not the particular element).
This is closely related to the uniform strong Dixmier property studied in \cite[Section 7.2]{FHLRTVW}, as well as the uniform averaging properties recently considered in \cite[Section 5]{NRS} and \cite[Section 6]{NS}.
Many of the classical examples of $C^*$-algebras with the Dixmier property turn out to have the uniform Dixmier property, including von Neumann algebras and $C^*_r(\mathbb F_2)$ (see Remark \ref{rmk:UDPexamples}).
Adding to this, we show that any $C^*$-algebra with the Dixmier property and with finite radius of comparison-by-traces has the uniform Dixmier property (Corollary \ref{cor exact UDP}).
We use Corollary~\ref{genHZ1} to characterise, in terms of two distinct uniformity conditions, when a tracial unital $C^*$-algebra with the Dixmier property and trivial centre has the uniform Dixmier property (Theorem~\ref{thm:SimpleUDPEquiv}).
Finally, following a suggestion by the referee, we find explicit constants for the uniform Dixmier property in a number of examples in Section \ref{sec:ExplicitConstants}.
\medskip

The starting point for our results is the following recent theorem of Ng, LR, and Skoufranis (\cite[Theorem 4.7]{NRS}), generalising a version by Ozawa (\cite[Theorem 1]{Ozawa}) in which all quotients have a tracial state:

\begin{thm} \cite{NRS} \label{NRS theorem}
Let $A$ be a unital $C^*$-algebra. Let  $a$  be a self-adjoint element in $A$. Then $0\in D_A(a)$
if and only if
\begin{enumerate}[(a)]
\item
$\tau(a)=0$ for all tracial states $\tau$ on $A$, and
\item
in no nonzero quotient of $A$ can the image of $a$ be either invertible and positive or invertible and negative.
\end{enumerate}
Furthermore, if $A$ has no tracial states then condition (a) is vacuously satisfied.
\end{thm}

Note that, in order to verify condition (b) in Theorem~\ref{NRS theorem}, it suffices to check simple quotients (that is, quotients of $A$ by maximal ideals).
Theorem \ref{MainThmIntro} is proven using the Kat\v{e}tov--Tong insertion theorem (see Theorem \ref{Katetov-Tong} below) to produce candidate central elements corresponding to any given self-adjoint element $a\in A$, and then using Theorem~\ref{NRS theorem} to verify that these candidates are indeed in the respective Dixmier set.

\medskip

In Section \ref{sec:Distance}, motivated by Theorem~\ref{NRS theorem}, we show that for elements $a$ and $b$ in an arbitrary unital $C^*$-algebra $A$, the distance between the Dixmier sets $D_A(a)$ and $D_A(b)$ can be read off from tracial data and the algebraic numerical ranges of $a$ and $b$ in quotients of $A$ (Theorem~\ref{DaDb}).
This result extends  Theorem~\ref{NRS theorem} in several ways: first by considering the Dixmier sets  of a pair of elements  $a$ and $b$ (rather than one of them being zero), second  by providing a distance formula between these sets (rather than focusing on the case that this distance is zero), and third by allowing the elements $a$ and $b$ to be non-self-adjoint.
We also show that, in certain cases, the distance between $D_A(a)$ and $D_A(b)$ is attained (Proposition~\ref{DaDbDP}).
In this section, we obtain elements in $Z(A)$ by using Michael's selection theorem, rather than the Kat\v{e}tov--Tong  theorem (cf. \cite{Magajna}, \cite{DSprox})

\bigskip

\subsection*{Acknowledgments}
We are grateful to Luis Santiago for helpful discussions at an early stage of this investigation.
We would also like to thank the referee for providing helpful comments, which have led to a number of improvements.

\bigskip

\subsection{Preliminaries and notation}

For a $C^*$-algebra $A$, we use the standard notation $S(A)$, $P(A)$ and $T(A)$ for the set of states, pure states and tracial states, respectively; the weak$^*$-topology is the natural topology used on these sets.
The set $T(A)$ is convex, and we use $\partial_e T(A)$ to denote its extreme boundary.
If $\tau\in T(A)$ then the left kernel
$$\{a\in A : \tau(a^*a)=0\}$$
is a closed (two-sided) ideal of $A$ and is easily seen to coincide with the kernel of the Gelfand--Naimark--Segal (GNS) representation $\pi_{\tau}$ (and with the right kernel). We shall refer to this ideal as the \emph{trace-kernel ideal} for $\tau$.
When $C$ is a commutative $C^*$-algebra (generally, arising as the centre of another $C^*$-algebra $A$) and $N \subseteq C$ is a maximal ideal, define $\phi_N \in P(C)$ to be the (unique) pure state satisfying
\begin{equation}
\label{phiN-def}
\phi_N(N) = \{0\}.
\end{equation}
For any proper closed ideal $J$ of $A$,
\[ q_J: A\to A/J \]
will denote the canonical quotient map.
For a subset $S$ of a $C^*$-algebra (or of $\mathbb R$), we write $\mathrm{co}(S)$ for the convex hull of $S$.

Let $A$ be a unital $C^*$-algebra with centre $Z(A)$ and let ${\rm Max}(A)$ be the subspace of $\Prim(A)$ (with the hull-kernel topology) consisting of all the maximal ideals of $A$.
It is well known and easy to see that there is a continuous surjection $\Psi: {\rm Max}(A) \to {\rm Max}(Z(A))$ given by $\Psi(M):=M\cap Z(A)$ for every maximal ideal $M$ of $A$.

\begin{defn} (\cite{Misonou1, Misonou2})
\label{weakly-central-def}
A $C^*$-algebra $A$ is said to be \emph{weakly central} if $\Psi$ (as just described) is injective.
\end{defn}

When $A$ is weakly central, $\Psi$ is a homeomorphism since its domain is compact and its range is Hausdorff.
Misonou used the Dixmier property to show that every von Neumann algebra is weakly central (\cite[Theorem 3]{Misonou2}).
As observed in \cite[p. 275]{RJA}, the same method shows that every unital $C^*$-algebra with the Dixmier property is weakly central.
Although weak centrality does not imply the Dixmier property (consider any unital simple $C^*$-algebra with more than one tracial state), Magajna has given a characterisation of weak centrality in terms of a more general kind of averaging involving elementary completely positive mappings (\cite{Magajna}).

A \emph{Glimm ideal} of a unital $C^*$-algebra $A$ is an ideal $NA$ ($=ANA$) generated by a maximal ideal $N$ of $Z(A)$ (see \cite[Section 4]{GlimmSW}); note that $NA$ is already closed by the Banach module factorisation theorem (a fact that does not require $N$ to be maximal).

Let $A$ be a $C^*$-algebra with centre $Z(A)$. A \emph{centre-valued trace on $A$} is a positive, linear contraction $R: A\to Z(A)$ such that $R(z)=z$ ($z\in Z(A)$) and $R(ab)=R(ba)$ ($a,b\in A$).
The equivalence of (i) and (ii) in the next result, together with the description of the centre-valued trace $R$, is essentially well-known and easy to see.
It underlies Dixmier's approach to the trace in a finite von Neumann algebra \cite{Dix1949, DixvN, KadRing}. A detailed proof is given in \cite[5.1.3]{Arch-thesis} using the same methods as in the case of a von Neumann algebra (see, for example, \cite[Corollaire III.8.4]{DixvN}).
(The equivalence with (iii) is probably also well-known, although we were unable to find a reference.)

\begin{prop} \label{SDPandCVT}
Let $A$ be a unital $C^*$-algebra with the Dixmier property. The following conditions are equivalent.
\begin{enumerate}[label=\emph{(\roman*)}]
\item
$A$ has the singleton Dixmier property.
\item
There exists a centre-valued trace on $A$.
\item
For every $M\in{\rm Max}(A)$, $T(A/M)$ is non-empty.
\end{enumerate}

\noindent When these equivalent conditions hold, the centre-valued trace $R$ is unique,
$$ \{R(a)\} = D_A(a)\cap Z(A) \qquad (a \in A),$$
and, for every $M\in{\rm Max}(A)$, $T(A/M)$ is a singleton.
\end{prop}

\begin{proof} It remains to establish the equivalence of (iii), and also the last part of the final sentence. Suppose that $A$ has the singleton Dixmier property and that $R:A \to Z(A)$ is the associated centre-valued trace on $A$.  Let $M\in {\rm Max}(A)$ and observe that, since $A/M$ is simple, $Z(A/M) = \C1_{A/M} = (Z(A)+M)/M$.
Since $R(a) \in D_A(a)$ ($a\in A$), it follows that $R(M) \subseteq M$ and hence it is easily seen that $R$ induces a centre-valued trace $R_M: A/M \to \C1_{A/M}$ (cf.\ the proof of \cite[Proposition 5.1.11]{Arch-thesis}). In particular, $A/M$ has a tracial state $\tau_M$ such that $\tau_M(q_M(a))1_{A/M} = R_M(q_M(a))$ ($a\in A$). Thus $T(A/M)$ is non-empty. In fact $T(A/M)= \{\tau_M\}$ since $A/M$ has the Dixmier property
\cite[p. 544]{Arch-PLMS} and trivial centre.

   Conversely, suppose that (iii) holds, that $a\in A$ and that $z_1,z_2\in D_A(a)\cap Z(A)$. Let $\phi\in P(Z(A))$,
\[ N:=\{a \in Z(A): \phi(a^*a)=0\} \in {\rm Max}(Z(A)), \]
and $M:=\Psi^{-1}(N)\in {\rm Max}(A)$. Let $\tau\in T(A/M)$. Then $\tau\circ q_M\in T(A)$, $(\tau\circ q_M)|_{Z(A)}=\phi$ and $\tau\circ q_M$ is constant on $D_A(a)$. Hence
   $$ \phi(z_1)= (\tau\circ q_M)(a)= \phi(z_2).$$
   Since this holds for all $\phi\in P(Z(A))$, we obtain that $z_1=z_2$ as required for (i).
\end{proof}

Since the Dixmier property passes to quotients (\cite[p. 544]{Arch-PLMS}), it is immediate from Proposition~\ref{SDPandCVT} (iii) that the singleton Dixmier property passes to quotients of unital $C^*$-algebras. More generally, the singleton Dixmier property passes to ideals and quotients of arbitrary $C^*$-algebras \cite[Proposition 5.1.11]{Arch-thesis}.

\medskip

The next theorem will not be applied until Section~\ref{sec:UDP}, but we include it here as it may be of independent interest (cf.\ \cite[Theorem 4.3]{APT}).
 In \cite[Lemma 2.1 (i)]{NgRobert2}, it is shown that a limit of sums of self-adjoint commutators in a quotient can be lifted (as below), but at a cost of $\epsilon$ in the norm. Theorem~\ref{thm:Lifting} shows that this $\epsilon$ cost can be avoided.
The proof uses a technique from Loring and Shulman's \cite{lor-shul}; the result almost follows from \cite[Theorem 3.2]{lor-shul}, except that they work with polynomials (in non-commuting variables) whereas we need to work with a  series (of commutators).
Here $[A,A]$ means the span of commutators in $A$, i.e., the span of elements of the form $[a,b]=ab-ba$, where $a,b \in A$.
Recall also that a \emph{quasicentral approximate unit} of an ideal $J$ of $A$ is an approximate unit $(u_\lambda)$ for $J$ which is approximately central in $A$.

We will need the following in the proof of this theorem.

\begin{lemma}\label{acudelta}
	Let $A$ be a $C^*$-algebra, $J$ a closed ideal of $A$
	and $(u_\lambda)$ a quasicentral approximate unit of $J$. Let $0<\delta<1$
	and $a\in A$. Then
	\[
	\limsup_\lambda \|a(1-\delta u_\lambda)\|\leq \max(\|q_J(a)\|, (1-\delta)\|a\|).
	\]
\end{lemma}	
\begin{proof}
	This is a special case of \cite[Theorem 2.3]{lor-shul}.
\end{proof}

\begin{thm}\label{thm:Lifting}
	Let $A$ be a $C^*$-algebra, let $J$ be a closed ideal of $A$ and let  $\bar a\in A/J$ be a self-adjoint element
	in $\overline{[A/J,A/J]}$. Then there exists a self-adjoint lift $a\in \overline{[A,A]}$
	of $\bar a$ such that $\|a\|=\|\bar a\|$.
\end{thm}	

\begin{proof}
	We may assume without loss of generality that $\|\bar a\|=1$.
	The strategy of the proof is as follows: We will construct a sequence $(a^{(n)})_{n=1}^\infty$
	of self-adjoints lifts of $\bar a$ such that
		$a^{(n)}\in \overline{[A,A]}$ for all $n$, $\|a^{(n)}\|\to 1$, and
		the sequence $(a^{(n)})_{n=1}^\infty$ is Cauchy.
	This is sufficient to prove the theorem, for then  $\lim_n a^{(n)}$ is the desired lift.

Pick any decreasing sequence $0<\delta_n<2/3$ such that $\sum_{n=1}^\infty \delta_n<\infty$.
Define  $\epsilon_n$ such that $(1+2\epsilon_{n})(1-\delta_n)=1$ for all $n\geq 1$. Notice that $\epsilon_{n}$
is also a decreasing sequence, $\epsilon_n<1$, and $\epsilon_n\to 0$.

We shall iteratively produce $a^{(n)}$ with the following properties:
\begin{itemize}
\item it has the form
\begin{equation}\label{aneq}
a^{(n)}=\sum_{i=1}^\infty [(x_i^{(n)})^*,x_i^{(n)}]
\end{equation}
for some $x_1^{(n)},x_2^{(n)},\dots \in A$;
\item $a^{(n)}$ is a lift of $\bar a$;
\item $\|a^{(n)}\| \leq 1+\epsilon_n$; and
\item $\|a^{(n)}-a^{(n-1)}\| < 4\delta_{n-1}$, for $n\geq 2$.
\end{itemize}
Since $\sum_{n=1}^\infty \delta_n < \infty$, the final item ensures that the sequence is Cauchy, and so upon finding such $a^{(n)}$, we are done.
	
	Let us start with a self-adjoint lift $a^{(1)}\in \overline{[A,A]}$ of $\bar a$ such that $\|a^{(1)}\|< 1+\epsilon_1$. This can be done by    \cite[Lemma 2.1 (i)]{NgRobert2}.
	By \cite[Theorem 2.6]{Cuntz-Ped}, we have
	\[
	a^{(1)}=\sum_{i=1}^\infty [(x_i^{(1)})^*,x_i^{(1)}]
	\]
	for some $x_i^{(1)}\in A$, where the series is norm convergent.
Now fix $n \geq 1$, and suppose that we have defined a self-adjoint $a^{(n)}$ that is a lift of $\bar a$, such that $\|a^{(n)}\|< 1+\epsilon_n$, and such that $a^{(n)}$ has the form
\[
a^{(n)}=\sum_{i=1}^\infty [(x_i^{(n)})^*,x_i^{(n)}].
\]
Find $k_n\in \NN$
such that
\begin{equation}\label{eq:smalltail}
\|\sum_{i>k_n} [(x_i^{(n)})^*,x_i^{(n)}]\|<\frac {\epsilon_{n+1}}{3}.
\end{equation}
Let $(u_\lambda)$ be a quasicentral approximate unit of $J$, and define
\[
x_{i}^{(n+1)}:=
\begin{cases}
x_i^{(n)}, & \hbox{ if }i>k_n;\\
x_i^{(n)}(1-\delta_{n} u_\lambda)^{\frac12}, & \hbox{ if }i\leq k_n.
\end{cases}
\]
Define  $a^{(n+1)}$ as in \eqref{aneq} using the new elements $x_i^{(n+1)}$ (the new series also converges since only finitely many terms were changed). It is clear that $a^{(n+1)}$ is a self-adjoint lift of $\bar a$  and that
$a^{(n+1)}\in \overline{[A,A]}$.
Presently, the element $a^{(n+1)}$ depends on $\lambda$.
We will choose $\lambda$ suitably. We have
\[
\|a^{(n+1)}\|< \Big\|\sum_{i=1}^{k_n}[(x_i^{(n+1)})^*,x_i^{(n+1)}]\Big\|+\frac{\epsilon_{n+1}}{3}.
\]
Exploiting the approximate centrality of $(u_\lambda)$ (using \cite[Proposition 1.8]{Winter:dr} to get that $(1-\delta_n u_\lambda)^{\frac12}$ is approximately central), we can choose $\lambda$ large enough such that
\[
\|a^{(n+1)}\|< \Big\|\Big(\sum_{i=1}^{k_n}[(x_i^{(n)})^*,x_i^{(n)}]\Big)(1-\delta_{n}u_\lambda)\Big\|+\frac{\epsilon_{n+1}}{3}
+\frac{\epsilon_{n+1}}{3}.
\]
We have
\[
\Big\|\sum_{i=1}^{k_n}[(x_i^{(n)})^*,x_i^{(n)}]\Big \|\leq 1+\epsilon_n+\frac{\epsilon_{n+1}}{3}<1+2\epsilon_n.
\]
Using \eqref{eq:smalltail}, we find that the norm of the image of $\sum_{i=1}^{k_n}[(x_i^{(n)})^*,x_i^{(n)}]$ in the quotient
$A/J$ is less than  $\|\bar a\|+\epsilon_{n+1}/3=1+\epsilon_{n+1}/3$.
So, by Lemma \ref{acudelta}, we can choose $\lambda$
large enough such that
\begin{align*}
\|(\sum_{i=1}^{k_n}[(x_i^{(n)})^*,x_i^{(n)}])(1-\delta_{n}u_\lambda)\| &< \max( 1+\frac{\epsilon_{n+1}}{3},(1-\delta_n)(1+2\epsilon_n)) \\
&=\max( 1+\frac{\epsilon_{n+1}}{3},1) \\
&=1+\frac{\epsilon_{n+1}}{3}.
\end{align*}
Then, for such choices of $\lambda$ we have $\|a^{(n+1)}\|< 1+\epsilon_{n+1}$.

Now consider $a^{(n+1)}-a^{(n)}$:
\[
a^{(n+1)}-a^{(n)}=\sum_{i=1}^{k_n}[(x_i^{(n+1)})^*,x_i^{(n+1)}]-[(x_i^{(n)})^*,x_i^{(n)}].
\]
Again using approximate centrality of $(u_\lambda)$, we may possibly increase $\lambda$ to get
\begin{align*}
\|\sum_{i=1}^{k_n}[(x_i^{(n+1)})^*,x_i^{(n+1)}]-[(x_i^{(n)})^*,x_i^{(n)}]\| &\leq \delta_{n}+
\|\sum_{i=1}^{k_n}[(x_i^{(n)})^*,x_i^{(n)}]((1-\delta_{n}u_\lambda)-1)\|\\
&\leq \delta_{n}+(1+2\epsilon_n)\delta_{n}\leq 4\delta_{n}.
\end{align*}
Thus, $\|a^{(n+1)}-a^{(n)}\|\leq 4\delta_n$, as required.
\end{proof}	

We recall from \cite[Proposition 2.7]{Cuntz-Ped} that, for $a\in A$,
$a\in \overline{[A,A]}$ if and only if $\tau(a)=0$ for all tracial states of $A$. Thus Theorem 1.6 can clearly be rephrased in terms of tracial states of $A/J$ and of $A$ instead of commutators.

\bigskip

\section{Tracial characterisations of the Dixmier property and the singleton Dixmier property}
\label{DPandSDP}

We begin with a few straightforward (and probably well-known) facts.

\begin{lemma} \label{lemma on unique max ideal}
Suppose that $A$ is a unital $C^*$-algebra containing a unique maximal ideal $J$. Then
$Z(A)=\C1$ and $Z(J)=\{0\}$.
\end{lemma}

\begin{proof}
Since the map $\Psi: {\rm Max}(A) \to {\rm Max}(Z(A))$ is surjective, $Z(A)$ has only one maximal ideal and this must therefore be the zero ideal. Thus $Z(A)=\C1$ and hence
$Z(J)=Z(A)\cap J=\{0\}$.
\end{proof}

The next result can be proved by using a quasicentral approximate unit or the GNS representation, or the invariance of the extension under unitary conjugation. A proof using an arbitrary approximate unit is given in \cite[Lemma 3.1]{Scarparo}.

\begin{lemma}\label{traces extend}
Let $J$ be a nonzero closed ideal of a $C^*$-algebra $A$ and let $\tau \in T(J)$.
Then the unique extension of $\tau$ to a state of $A$ (see \cite[3.1.6]{Ped}) is a tracial state.
\end{lemma}

\begin{lemma}\label{extreme traces preserved by quotients}
Let $J$ be a proper closed ideal of a unital $C^*$-algebra $A$.
Then for any $\tau \in \partial_e T(A/J)$, $\tau \circ q_J \in \partial_e T(A)$.
\end{lemma}

\begin{lemma}\label{lemma: extreme trace and Glimm ideal}
Let $A$ be a unital $C^*$-algebra and suppose that $\tau\in \partial_eT(A)$.
Then $\tau|_{Z(A)}$ is a pure state on $Z(A)$.
\end{lemma}

\begin{proof}
Let $z \in Z(A)$ be a positive contraction.
Then the function $\tau_z:A \to \mathbb C$ given by $\tau_z(a):=\tau(za)$ is a tracial functional on $A$ which clearly satisfies $\tau_z\leq \tau$.
Since $\tau$ is an extreme tracial state, it follows that $\tau_z$ is a scalar multiple of $\tau$, and so
\[ \tau(za)=\tau_z(1)\tau(a) = \tau(z)\tau(a). \]
In particular, this shows that $\tau|_{Z(A)}$ is multiplicative, and therefore a pure state.
\end{proof}

We will also need the following.

\begin{thm}[Kat\v{e}tov--Tong insertion theorem] \label{Katetov-Tong}
Let $X$ be a normal space and $Y$ a closed subspace.
Let $f:X \to \R$ be upper semicontinuous, $g:Y \to \R$ be continuous, and $h:X \to \R$ be lower semicontinuous, satisfying
\[ f(x) \leq h(x) \quad (x \in X) \quad \text{and}\quad f(x)\leq g(x) \leq h(x)\quad (x \in Y). \]
Then there exists $\tilde g:X \to \R$ continuous such that $\tilde g|_Y= g$ and
\begin{equation}
\label{KT conclusion}
 f(x)\leq \tilde g(x)\leq h(x) \quad (x \in X).
\end{equation}
\end{thm}

\begin{proof}
We reduce this to the standard form of the Kat\v{e}tov--Tong insertion theorem (see \cite{Katetov} or \cite{Tong}), which is the case that $Y=\emptyset$.
Define $f_1,h_1:X \to \R$ by
\begin{align*}
f_1(x) &:= \begin{cases} g(x),\quad &x \in Y, \\ f(x),\quad &x\not\in Y, \end{cases} \quad \text{and} \\
h_1(x) &:= \begin{cases} g(x),\quad &x \in Y, \\ h(x),\quad &x\not\in Y. \end{cases}
\end{align*}
Using that $f$ is upper semicontinuous, that $Y$ is closed, and that $f\leq g$ on $Y$, it follows that $f_1$ is upper semicontinuous.
Likewise, $h_1$ is lower semicontinuous.
It is also clear that $f_1 \leq h_1$.
Therefore by the standard form of the Kat\v{e}tov--Tong insertion theorem, there exists a continuous function $\tilde g:X \to \R$ such that
\[ f_1 \leq \tilde g \leq h_1. \]
The definitions of $f_1$ and $h_1$ ensure that \eqref{KT conclusion} holds.
\end{proof}

Here is our first main theorem, characterising the Dixmier property in terms of other conditions that are more readily verified, namely weak centrality and tracial conditions.

\begin{thm} \label{DP-characterisation}
Let $A$ be a unital $C^*$-algebra.
The following are equivalent.
\begin{enumerate}[label=\emph{(\roman*)}]
\item $A$ has the Dixmier property
\item $A$ is weakly central and, for every $M\in {\rm Max}(A)$,
\begin{enumerate}
\item $A/(M \cap Z(A))A$ has at most one tracial state, and
\item if $A/(M \cap Z(A))A$ has a tracial state $\tau$, then $\tau(M/(M \cap Z(A))A)=\{0\}$.
\end{enumerate}
\item $A$ is weakly central and
\begin{enumerate}
\item for every $M\in {\rm Max}(A)$, $A/M$ has at most one tracial state, and
\item every extreme tracial state of $A$ factors through $A/M$ for some maximal ideal $M$.
\end{enumerate}
\end{enumerate}
When $A$ has the Dixmier property, $\partial_e T(A)$ is homeomorphic to the set
\begin{equation}
\label{Y-def} Y:=\{M \in \mathrm{Max}(A): A \text{ has a (unique) tracial state } \tau_M \text{ that annihilates } M\}
\end{equation}
via the assignment $M\mapsto \tau_M$, the set $Y$ is closed in $\mathrm{Max}(A)$, and $T(A)$ is a Bauer simplex (possibly empty).
\end{thm}

\begin{proof}
(i)$\Rightarrow$(ii):
Suppose that $A$ has the Dixmier property and hence is weakly central.
Let $M\in{\rm Max}(A)$ and set $N:=M \cap Z(A)$, a maximal ideal of $Z(A)$.
By weak centrality,  $M/NA$ is the unique maximal ideal of $A/NA$. Hence $Z(A/NA)=\C(1+NA)$ and $Z(M/NA)=\{0\}$ by Lemma~\ref{lemma on unique max ideal}.
By \cite[p. 544]{Arch-PLMS}, the $C^*$-algebra $A/NA$ has the Dixmier property.
 Since tracial states are constant on Dixmier sets, we conclude that if $A/NA$ has a tracial state then it is unique and it annihilates $M/NA$.

(ii)(a)$\Rightarrow$(iii)(a):
For $M\in {\rm Max}(A)$,  $(M\cap Z(A))A \subseteq M$ and so $A/M$ is a quotient of $A/(M \cap Z(A))A$.
Hence (ii)(a) implies (iii)(a).

(ii)(b)$\Rightarrow$(iii)(b):
Let $\tau$ be an extreme tracial state of $A$. By Lemma~\ref{lemma: extreme trace and Glimm ideal}, there exists a maximal ideal $N$ of $Z(A)$ such that $\tau(N)=\{0\}$. Hence $\tau(NA)=\{0\}$ by the Cauchy--Schwartz inequality for states.
Let $M \in \mathrm{Max}(A)$ be such that $M \cap Z(A)=N$; then since $\tau$ induces a tracial state on $A/NA=A/(M\cap Z(A))A$, it follows from (ii)(b) that this tracial state annihilates $M/(M \cap Z(A))A$, i.e., $\tau(M)=\{0\}$, as required.

(iii)$\Rightarrow$(ii)(b):
To prove (ii)(b), it suffices by the Krein--Milman theorem to show that if $\tau$ is an extreme tracial state on $A/(M\cap Z(A))A$ then $\tau(M/(M\cap Z(A))A)=\{0\}$.
By Lemma \ref{extreme traces preserved by quotients}, the induced tracial state $\tilde\tau$ on $A$ is also extreme, and by (iii) it factors through $A/M'$ for some $M'\in{\rm Max}(A)$.
Then (using $\phi_{M \cap Z(A)}$ as defined by \eqref{phiN-def}),
$$\phi_{M'\cap Z(A)} = \tilde\tau|_{Z(A)} =\phi_{M\cap Z(A)},$$
and so $M'\cap Z(A)=M\cap Z(A)$.
By weak centrality, we conclude that $M'=M$ and therefore $\tau(M/(M\cap Z(A))A)=0$.

(iii) and (ii)(b)$\Rightarrow$(i):
Assume that (iii) and (ii)(b) hold.
Define $X:=\mathrm{Max}(A) \cong \mathrm{Max}(Z(A))$ (thus a compact Hausdorff space) and $Y:=\{M \in X: A/M \text{ has a tracial state}\}$. By the Krein--Milman theorem and (iii)(b), $Y$ is non-empty if and only if $T(A)$ is non-empty. By (iii)(a), for each $M\in Y$, there is a unique tracial state $\tau_M$ of $A$ that vanishes on $M$. It follows from Lemma~\ref{extreme traces preserved by quotients} that $\tau_M\in\partial_eT(A)$. We define $G: Y\to \partial_eT(A)$ by $G(M):=\tau_M$ ($M\in Y$). If $M_1,M_2\in Y$ and
$G(M_1)=G(M_2)$ then the state $\tau_{M_1}$ vanishes on $M_1+M_2$ and so $M_1=M_2$.  Thus $G$ is injective, and it is surjective by (iii)(b).
We will show that $Y$ is closed in ${\rm Max}(A)$ (and hence compact) and that the bijection $G$ is continuous for the weak$^*$-topology on the Hausdorff space $\partial_eT(A)$ (and hence $G$ is a homeomorphism).

Let $M$ belong to the closure of $Y$ in ${\rm Max}(A)$ and let $(M_i)$ be an arbitrary net in $Y$ that is convergent to $M$. Since $T(A)$ is weak$^*$-compact, there exist $\tau\in T(A)$ and a subnet $(M_{i_j})$ such that $\tau_{M_{i_j}}\to_j \tau$. Then
$$\tau|_{Z(A)} = \lim_j  \phi_{M_{i_j}\cap Z(A)} = \phi_{M\cap Z(A)}.$$
It follows from the Cauchy--Schwartz inequality for states that $\tau$ annihilates the Glimm ideal $(M\cap Z(A))A$ and hence $\tau(M)=\{0\}$ by (ii)(b). Thus $M\in Y$ and $\tau=\tau_M$. Since $(M_i)$ is an arbitrary net in $Y$ convergent to $M$ and
$\tau_{M_{i_j}}\to_j \tau_M$, $G$ is continuous at $M$ and therefore continuous on $Y$.

Now let $a \in A$ be self-adjoint.
We show that $D_A(a) \cap Z(A) \neq \emptyset$.
Our strategy is to define a candidate $z \in Z(A)$ and then use Theorem~\ref{NRS theorem} to show that $z \in D_A(a)$.
Define functions $f,h:X \to \mathbb R$ by
\[ f(M) := \min \mathrm{sp}(q_M(a)), \quad h(M) := \max \mathrm{sp}(q_M(a)) \quad (M \in \mathrm{Max}(A)) \]
(cf.\ \cite[p.279]{RJA}).
One can rewrite these as
\[ f(M)= \|a\|-\big\|q_M(\|a\|1-a)\big\| \quad \text{and} \quad h(M)= \big\|q_M(\|a\|1+a)\big\|-\|a\|; \]
\cite[Proposition 4.4.4]{Ped} tells us that the functions $M \mapsto \big\|q_M(\|a\|1\pm a)\big\|$ are lower semicontinuous, and therefore, $h$ is lower semicontinuous and  $f$ is upper semicontinuous.

Finally define $g:Y \to \mathbb R$ by $g(M):=G(M)(a)=\tau_M(a)$. Since $G$ is continuous on $Y$, so is $g$.
Evidently,
\[ f(M) \leq h(M) \quad (M \in X). \]
For all $M\in Y$,
\[ f(M)1_{A/M} \leq q_M(a) \leq h(M)1_{A/M} \]
and hence, by the positivity of the tracial state induced by $\tau_M$ on $A/M$,
\[ f(M) \leq g(M)  \leq h(M). \]

By the Kat\v{e}tov--Tong insertion theorem (Theorem \ref{Katetov-Tong}), there exists a function $\tilde g \in C(X)$ such that $\tilde g|_Y = g$ and
\[ f(M) \leq \tilde g(M) \leq h(M) \quad (M \in X). \]
Since $\tilde g\circ\Psi^{-1}\in C({\rm Max}(Z(A)))$, Gelfand theory for the commutative $C^*$-algebra $Z(A)$ yields a self-adjoint element $z\in Z(A)$ such that
$$ q_M(z) = \tilde g(M)1_{A/M} \in A/M \qquad (M\in {\rm Max}(A)).$$
Then $\tau_M(a-z)=0$ for all $M \in Y$. Since $G$ is surjective,  the Krein--Milman theorem yields
\[ \tau(a-z)=0 \quad (\tau \in T(A)), \]
verifying (a) of Theorem~\ref{NRS theorem}.
For every maximal ideal $M$ of $A$, $0$ is in the convex hull of the spectrum of $q_M(a-z)$; this is because the spectrum of this element is the translation of the spectrum of $q_M(a)$ by $\tilde g(M)$, and $\tilde g(M)$ is chosen to be between the minimum and the maximum of the spectrum of $g_M(a)$.
Therefore $0$ is in the convex hull of the spectrum of the image of $a-z$ in any quotient of $A$.
This shows that (b) of Theorem~\ref{NRS theorem} holds.
Hence by Theorem~\ref{NRS theorem} , $0 \in D_A(a-z)$ and so $z \in D_A(a)$ as required.

Now, for $a\in A$ (not necessarily self-adjoint) we may write $a=b+ic$, where $b$ and $c$ are self-adjoint elements of $A$, and a standard argument of successive averaging (cf.\ the proof of \cite[Lemma 8.3.3]{KadRing}) shows that $d(D_A(a), Z(A))=0$. By \cite[Lemma 2.8]{Arch-PLMS}, $A$ has the Dixmier property.

Finally, we have seen above that when $A$ has the Dixmier property, $\partial_eT(A)$ is homeomorphic to the compact set $Y$ and so the Choquet simplex $T(A)$ is a Bauer simplex (possibly empty).
\end{proof}

Suppose that $A$ is a unital $C^*$-algebra with the Dixmier property and that $\theta: Z(A) \to C({\rm Max}(A))$ is the canonical $^*$-isomorphism induced by the Gelfand transform for $Z(A)$ and the homeomorphism $\Psi: {\rm Max}(A) \to {\rm Max}(Z(A))$. Let $a=a^*\in A$, let $f$ and $h$ be the associated spectral functions on ${\rm Max}(A)$ and let $g$ be the associated function on the closed subset $Y$ of ${\rm Max}(A)$ (see the proof of Theorem~\ref{DP-characterisation}). Then it follows from Theorem~\ref{NRS theorem} that
$$D_A(a) \cap Z(A) = \{z\in Z(A): z=z^*,\,\, f\leq \theta(z)\leq h \text{ and } \theta(z)|_Y=g\}.$$
Thus $D_A(a) \cap Z(A)$ is closed under the operations of max and min (regarding self-adjoint elements of $Z(A)$ as continuous functions on ${\rm Max}(A)$). Furthermore, if $z_1,z_2,z_3$ are self-adjoint elements of $Z(A)$ such that $z_1\leq z_2\leq z_3$ and $z_1,z_3\in D_A(a)$ then $z_3\in D_A(a)$.

In the case where $A$ is a properly infinite von Neumann algebra (and hence for a general von Neumann algebra), Ringrose has shown that $D_A(a) \cap Z(A)$ is an order interval in the self-adjoint part of $Z(A)$ and has given a formula for the end-points in terms of spectral theory (see \cite[Corollary 2.3, Theorem 3.3 and Remark 3.5]{RingPLMS84}). The next result gives a different spectral description for the end-points.

\begin{cor} \label{cor vNalg}
Let $A$ be a properly infinite von Neumann algebra and let $a=a^*\in A$. Then, with the notation above, the spectral functions $f$ and $h$ are continuous on ${\rm Max}(A)$,
$\theta^{-1}(f), \theta^{-1}(h) \in D_A(a)\cap Z(A)$ and
$$D_A(a)\cap Z(A) = \{z\in Z(A): z=z^* \text{ and } \theta^{-1}(f)\leq z\leq \theta^{-1}(h)\}.$$
\end{cor}

\begin{proof}
For $b\in A$, the function $M\to \Vert q_M(b)\Vert$ is continuous on ${\rm Max}(A)$ by \cite[Proposition 1]{Halpern1970}. It follows that the functions $f$ and $h$ are continuous on ${\rm Max}(A)$. Since $A$ has no tracial states, the subset $Y$ of ${\rm Max}(A)$ is empty. The result now follows from the discussion above.
\end{proof}

We now show how Theorem~\ref{DP-characterisation} leads to necessary and sufficient conditions for the singleton Dixmier property.

\begin{cor} \label{cor:SDP1}
Let $A$ be a unital $C^*$-algebra. The following are equivalent.
\begin{enumerate}[label=\emph{(\roman*)}]
\item
$A$ has the singleton Dixmier property.
\item
$A$ is weakly central and, for every $M\in{\rm Max}(A)$, $A/(M\cap Z(A))A$ has a unique tracial state and this state annihilates $M/(M\cap Z(A))A$.
\item
$A$ is weakly central and
\begin{enumerate}
\item
for every $M\in{\rm Max}(A)$, $A/M$ has a unique tracial state, and
\item
every extreme tracial state of $A$ factors through $A/M$ for some $M\in{\rm Max}(A)$.
\end{enumerate}
\item
\begin{enumerate}
\item
for every $M\in{\rm Max}(A)$, $A/M$ has a unique tracial state, and
\item
the restriction map $r:T(A)\to S(Z(A))$ is a homeomorphism for the weak$^*$-topologies.
\end{enumerate}
\item
\begin{enumerate}
\item
for every $M\in{\rm Max}(A)$, $T(A/M)$ is non-empty, and
\item
the restriction map $r_e:\partial_eT(A)\to P(Z(A))$ is injective.
\end{enumerate}
\end{enumerate}
\end{cor}

\begin{proof}
The equivalence of (i), (ii) and (iii) follows from Theorem~\ref{DP-characterisation} and Proposition~\ref{SDPandCVT}. It is also clear that (iv) implies (v) (note that $r_e$ maps extreme tracial states into $P(Z(A))$ by Lemma~\ref{lemma: extreme trace and Glimm ideal}).

(i)$\Rightarrow$(iv):
Suppose that $A$ has the singleton Dixmier property. Then (iv)(a) holds by Proposition~\ref{SDPandCVT}.
For (iv)(b), we proceed as in the well-known case of a finite von Neumann algebra (cf.\ \cite[Proposition III.5.3]{DixvN}). For the surjectivity of $r$, we observe that if $\phi\in S(Z(A))$ then $\phi\circ R\in T(A)$, where $R:A\to Z(A)$ is the unique centre-valued trace of $A$, and $(\phi\circ R)|_{Z(A)}=\phi$. The injectivity of $r$ follows from the
 facts that $A$ has the Dixmier property and tracial states are constant on Dixmier sets. Since $r$ is a weak$^*$-continuous bijection from the compact space $T(A)$ to the Hausdorff space $S(Z(A))$, it is a homeomorphism.

(v)$\Rightarrow$(iii):
 Suppose that $A$ satisfies (v) and let $M\in{\rm Max}(A)$. By (v)(a) and the Krein--Milman theorem, there exists $\tau_M\in \partial_eT(A/M)$. Then $\tau_M\circ q_M\in \partial_eT(A)$ (Lemma \ref{extreme traces preserved by quotients}) and $(\tau_M\circ q_M)|_{Z(A)} = \phi_{M\cap Z(A)}$. Since $r_e$ is injective, $\tau_M$ is unique and hence $T(A/M)=\{\tau_M\}$ by the Krein--Milman theorem. This establishes (iii)(a).

For weak centrality, suppose that $M_1, M_2\in {\rm Max}(A)$ and $M_1\cap Z(A)=M_2\cap Z(A)$. Then
$$r_e(\tau_{M_1}\circ q_{M_1}) = \phi_{M_1\cap Z(A)} =\phi_{M_2\cap Z(A)} =r_e(\tau_{M_2}\circ q_{M_2}).$$
Since $r_e$ is injective, $\tau_{M_1}\circ q_{M_1}=\tau_{M_2}\circ q_{M_2}$, which is a state annihilating $M_1+M_2$. Hence $M_1=M_2$.

For (iii)(b), let $\tau\in \partial_eT(A)$. By Lemma~\ref{lemma: extreme trace and Glimm ideal}, there exists $N \in \mathrm{Max}(Z(A))$ such that
\[ \tau|_{Z(A)} = \phi_N. \]
Let $M\in \mathrm{Max}(A)$ satisfy $M \cap Z(A) = N$, so that
$$\tau|_{Z(A)} = \phi_{M\cap Z(A)} = (\tau_M\circ q_M)|_{Z(A)}.$$
Since $r_e$ is injective, $\tau=\tau_M\circ q_M$.
\end{proof}

\begin{cor} \label{cor:SDP2}
Let $A$ be a unital $C^*$-algebra with the Dixmier property and suppose that $T(A)$ is non-empty.
Then there exists a unique proper closed ideal $J$ of $A$ with the following property: for every proper closed ideal $K$ of $A$, $A/K$ has the singleton Dixmier property if and only if $K\supseteq J$.
\end{cor}

\begin{proof}
From Theorem~\ref{DP-characterisation}, we have that
$$Y:=\{M\in{\rm Max}(A) : T(A/M) \text{ is non-empty }\}$$
is a non-empty closed subset of ${\rm Max}(A)$. Let $N:=\bigcap_{M\in Y}(M\cap Z(A))$ and $J:=NA$. Since $Y$ is non-empty, $J$ is a proper ideal of $A$.

Let $K$ be a proper closed ideal of $A$ and suppose that $A/K$ has the singleton Dixmier property. Let $P$ be a primitive ideal of $A$ containing $K$ and let $M$ be a maximal ideal of $A$ containing $P$. Since $A/K$ has the singleton Dixmier property, it follows from Proposition~\ref{SDPandCVT} that $T((A/K)/(M/K))$ is non-empty and hence $M\in Y$. On the other hand, $P\cap Z(A)$ is a prime ideal of $Z(A)$ and hence
$$ P\cap Z(A) = M\cap Z(A) \supseteq N.$$
It follows that $P\supseteq NA =J$. Since this holds for all $P\in \Prim(A/K)$, we obtain that $K\supseteq J$.

Conversely, suppose that $K\supseteq J$. Since $A$ has the Dixmier property, so does $A/K$. Let $M$ be a maximal ideal of $A$ that contains $K$. Since $M\cap Z(A) \supseteq J\cap Z(A) \supseteq N$ and $\Psi(Y)$ is closed in ${\rm Max}(Z(A))$, we obtain that $M\in Y$.
Thus $T((A/K)/(M/K))$ is non-empty and so $A/K$ has the singleton Dixmier property by Proposition~\ref{SDPandCVT}.

The uniqueness of $J$ is immediate from its stated property.
\end{proof}

We highlight the special case of Theorem~\ref{DP-characterisation} in which $Z(A)$ is trivial, which generalises results from  \cite{HZ}.
This case plays a crucial role in our investigation of the \emph{uniform} Dixmier property for $C^*$-algebras with trivial centre, in Section \ref{sec:SimpleUltrapower}.

\begin{cor}
\label{genHZ1}
Suppose that $A$ is a unital $C^*$-algebra. The following conditions are equivalent.
\begin{enumerate}[label=\emph{(\roman*)}]
\item
$Z(A)=\C1$ and $A$ has the Dixmier property.
\item
$A$ has a unique maximal ideal $J$, $A$ has at most one tracial state and $J$ has no tracial states.
\item
$A$ has a unique maximal ideal $J$, $A/J$ has at most one tracial state  and $J$ has no tracial states.
\end{enumerate}

When these hold, $A$ has the singleton Dixmier property exactly when it has a tracial state $\tau$, and in this case,
\[ J = \{x \in A : \tau(x^*x)=0\}, \]
the trace-kernel ideal for $\tau$.

If $A$ has the Dixmier property and no tracial states then
$$  D_A(a) \cap \C1 = \{t1: t\in \mathrm{co}(\mathrm{sp}(q_J(a))\}.$$
\end{cor}

\begin{proof}
Suppose that (i) holds. By Theorem~\ref{DP-characterisation} ((i) $\Rightarrow$ (ii)), $A$ is weakly central and hence, since $Z(A)=\C1$, $A$ has a unique maximal ideal $J$. Since $J\cap Z(A)=\{0\}$, $A$ has at most one tracial state by Theorem~\ref{DP-characterisation}(ii)(a) and if $A$ does have a tracial state then it annihilates $J$ by Theorem~\ref{DP-characterisation}(ii)(b). By Lemma~\ref{traces extend}, $J$ has no tracial states. Thus (ii) holds.

Conversely, suppose that (ii) holds. Then $Z(A)=\C1$ (by Lemma~\ref{lemma on unique max ideal}) and $A$ is weakly central. If $A$ has a tracial state then it must annihilate $J$ since $J$ has no tracial states. Thus (i) holds by Theorem~\ref{DP-characterisation}((ii) $\Rightarrow$ (i)).

(ii)$\Leftrightarrow$(iii) is immediate.

The statement concerning the singleton Dixmier property follows from Corollary \ref{cor:SDP1} (i)$\Leftrightarrow$(ii), and the final statement follows from Theorem~\ref{NRS theorem}.
\end{proof}

An example of a non-simple $C^*$-algebra with a unique maximal ideal, with the Dixmier property but not the singleton Dixmier property is the ``Cuntz--Toeplitz algebra'' $A:=C^*(S_1, \ldots, S_n)$ where $2\leq n<\infty$ and $S_1, \ldots, S_n$ are isometries on an infinite dimensional Hilbert space with mutually orthogonal range projections having sum less than $1$ (cf.\ \cite[Theorem 11]{Arch-JLMS}).

Corollary~\ref{genHZ1} above motivates the following question.
Is there an example of a unital $C^*$-algebra $A$ containing a unique maximal ideal $J$ such that $A$ has a unique tracial state and $A/J$ has no tracial states?
A non-separable example is the multiplier algebra $M(J)$ where $J$ is a non-unital hereditary subalgebra of a UHF algebra; here, $J$ is simple and has a unique trace, and by \cite[Theorem 3.1 and its proof]{Elliott2}, $M(J)/J$ is simple and infinite.
Thus $J$ is the unique maximal ideal of $M(J)$, and the extension of the trace on $J$ is the unique trace on $M(J)$.

For a separable nuclear example, one may utilise a construction of Kirchberg \cite{Kirch} as pointed out by Ozawa at the end of \cite[Section 3]{Ozawa}. Thus $J$ and $A$ are $C^*$-subalgebras of the CAR algebra $\M_{2^{\infty}}$ such that $J$ is hereditary in $\M_{2^{\infty}}$ and is an ideal in $A$ such that $A/J\cong \mathcal{O}_{\infty}$. Since $\M_{2^{\infty}}$ is simple and has a faithful, unique tracial state, $J$ also has both of these properties (note that any tracial state of $J$ can be extended to a bounded tracial functional on $\M_{2^{\infty}}$). Suppose that $A$ has a maximal ideal $M$ distinct from $J$. Then $M\cap J=\{0\}$ and so
$$\mathcal{O}_{\infty} \cong (M+J)/J \cong M/ M\cap J = M,$$
contradicting the fact that $A$ has a faithful tracial state induced from $\M_{2^{\infty}}$.
It follows from a theorem of Cuntz and Pedersen \cite[Theorem 2.9]{Cuntz-Ped}, as in the proof of \cite[Theorem 14]{Murphy}, that $A$ has a faithful, unique tracial state.
Even though $A/J$ satisfies a strong form of the Dixmier property \cite[Theorem 8]{Arch-JLMS}, $A$ itself does not have the Dixmier property because its tracial state does not vanish on $J$. 

This example also shows that, in Corollary~\ref{cor:SDP1}, the condition (v)(a) does not follow from condition (v)(b). On the other hand, to see that condition (v)(a) does not imply condition (v)(b) in Corollary~\ref{cor:SDP1}, consider any simple unital $C^*$-algebra with more than one tracial state.

The following concerns the Dixmier property for non-unital $C^*$-algebras; a non-unital $C^*$-algebra $A$ is said to have the \emph{(singleton) Dixmier property} if the unitisation $A+\C1$ has the same property.

\begin{cor} \label{third cor}
Let $A$ be a $C^*$-algebra with no tracial states.
Then the following conditions are equivalent.
\begin{enumerate}[label=\emph{(\roman*)}]
\item
$A$ has the Dixmier property and $Z(A)=0$.
\item
$A$ has the singleton Dixmier property and $Z(A)=0$.
\item
$A$ is the unique maximal ideal of the unitisation $A+\C1$.
\end{enumerate}
\end{cor}

\begin{proof}
(i)$\Leftrightarrow$(iii) is Corollary~\ref{genHZ1} (i)$\Leftrightarrow$(iii) applied to $A+\C1$, while the singleton Dixmier property in (ii) is the final sentence of Corollary~\ref{genHZ1}.
\end{proof}

A $C^*$-algebra $A$ (with or without an identity) is said to have the \emph{centre-quotient property} if $Z(A/J) = (Z(A)+J)/J$ for every proper closed ideal $J$ of $A$. Vesterstr\o m  showed that, for unital $A$, the centre-quotient property is equivalent to weak centrality (\cite[Theorems 1 and 2]{Vest}). Dixmier  observed that the centre-quotient property is a simple consequence of the Dixmier property in a von Neumann algebra (\cite[p.259, Ex. 7]{DixvN}).
Similarly, it is easily seen that if a $C^*$-algebra has the Dixmier property then it also has the centre-quotient property (\cite[2.2.2]{Arch-thesis}).
The next result was obtained in \cite[4.3.1, 5.1.9]{Arch-thesis} under the additional assumption that either $A$ is separable or there is a finite bound on the covering dimension of compact Hausdorff subsets of the spectrum ${\hat A}$. The method was very different from that used below.

\begin{thm} \label{thesis theorem} Let $A$ be a postliminal $C^*$-algebra. The following conditions are equivalent.
\begin{enumerate}[label=\emph{(\roman*)}]
\item
$A$ has the centre-quotient property.
\item
$A$ has the singleton Dixmier property.
\item
$A$ has the Dixmier property.
\end{enumerate}
\end{thm}

\begin{proof}
(i)$\Rightarrow$(ii):
 Suppose first of all that $A$ is a unital postliminal $C^*$-algebra with the centre-quotient property. Then $A$ is weakly central (\cite{Vest}). Furthermore, $A$ automatically satisfies conditions (iii)(a) and (iii)(b) of Corollary~\ref{cor:SDP1}. For (iii)(a), recall that a simple, unital $C^*$-algebra of type I is $^*$-isomorphic to $M_n$ for some $n\in\NN$.  For (iii)(b), note that if $\tau\in\partial_e T(A)$ then $\pi_{\tau}(A)''$ is a finite factor of type I (see \cite[6.8.7 and 6.8.6]{Dix}) and so $\ker\pi_{\tau}$ is maximal. By Corollary~\ref{cor:SDP1}, $A$ has the singleton Dixmier property.

Secondly, suppose that $A$ is a non-unital postliminal $C^*$-algebra with the centre-quotient property. Then it is easily seen that $A+\C1$ has the centre-quotient property (note that if $J$ is a closed ideal of $A+\C1$ then either $J\subseteq A$ or else
$(A+\C1)/J$ is canonically $^*$-isomorphic to $A/(A\cap J$). Thus $A+\C1$ is a unital postliminal $C^*$-algebra with the centre-quotient property and so has the singleton Dixmier property by the first part of the proof.

(iii)$\Rightarrow$(i): For the convenience of the reader, we give the details in the case where $A$ is a non-unital $C^*$-algebra with the Dixmier property. The unital case is even easier (and could alternatively be obtained via weak centrality and \cite{Vest}).  Let $J$ be a closed ideal of A and let $q:A+\C1\to (A+\C1)/J$ be the canonical quotient map. Suppose that $a\in A$ and that $q(a)\in Z(A/J)\subseteq Z((A+\C1)/J)$. Since $D_{A+\C1}(a)\subset A$ and $Z(A+\C1)\cap A = Z(A)$, there exists $z\in D_{A+\C1}(a)\cap Z(A)$. Then $q(z) \in D_{q(A+\C1)}(q(a)) =\{q(a)\}$ and so $q(a)\in (Z(A)+J)/J$, as required.
\end{proof}

\begin{cor}
Let $A$ be a postliminal $C^*$-algebra such that every irreducible representation of $A$ is infinite dimensional. Then $A$ has the singleton Dixmier property.
\end{cor}

\begin{proof} As in the proof of \cite[4.3.2]{Arch-thesis}, the use of a composition series with liminal quotients shows easily that the centre of $A$ is $\{0\}$.
Since the same applies to any nonzero quotient of $A$, it follows that $A$ has the centre-quotient property and hence the singleton Dixmier property.
\end{proof}

\bigskip

\section{The uniform Dixmier property}
\label{sec:UDP}

In this section, we introduce and study the following uniform version of the Dixmier property (cf.\ \cite[Section 6]{NS} and \cite[Section 5]{NRS}).

\begin{defn}
A unital $C^*$-algebra $A$ has the \emph{uniform Dixmier property} if for every $\epsilon>0$ there exists $n\in\mathbb N$ such that for all $a \in A$, there exist unitaries $u_1,\dots,u_n \in \mathcal U(A)$ such that
\[ d\Big(\sum_{i=1}^n \frac1n u_iau_i^*, Z(A)\Big) \leq \epsilon\|a\|. \]
\end{defn}

\begin{thm}\label{newUDP-equiv}
Let $A$ be a unital $C^*$-algebra. The following are equivalent:
\begin{enumerate}[label=\emph{(\roman*)}]
\item
$A$ has the uniform Dixmier property.

\item
There exist $m\in \NN$ and $0<\gamma<1$ such that for every self-adjoint $a\in A$ we have that
\begin{equation*}
\Big\|\sum_{i=1}^m \frac1m u_iau_i^*-z\Big\| \leq \gamma\|a\|,
\end{equation*}
for some $z\in Z(A)$ and  $u_1,\dots,u_m \in \mathcal U(A)$.

\item
There exist $m\in \NN$ and $0<\gamma<1$ such that for every self-adjoint  $a\in A$ we have that
\begin{equation}\label{tgamma}
\Big\|\sum_{i=1}^m t_iu_iau_i^*-z\Big\| \leq \gamma\|a\|,
\end{equation}
for some $z\in Z(A)$, some  $u_1,\dots,u_m \in \mathcal U(A)$, and some $t_1,\ldots,t_m\in [0,1]$ such that $\sum_{i=1}^m t_i=1$.
\item
There exists a function $\Phi: A\to Z(A)$ such that for every $\epsilon>0$ there exists
$n\in \NN$ such that for all $a\in A$ we have that
\[
\Big\|\sum_{i=1}^n \frac1n u_iau_i^*-\Phi(a)\Big\| \leq \epsilon\|a\|,
\]
for some  unitaries $u_1,\dots,u_n \in \mathcal U(A)$.
\end{enumerate}
\end{thm}
\begin{proof}
This proof uses known ideas from the theory of the Dixmier property and of sequence algebras, and is included for completeness.

The implications (i)$\Rightarrow$(ii)$\Rightarrow$(iii)  and (iv)$\Rightarrow$(i) are clear.

Let us prove that (iii)$\Rightarrow$(iv).
Given an arbitrary element $a \in A$, we can decompose $a$ as $b+ic$ where $b,c$ are self-adjoint and $\|b\|,\|c\| \leq \|a\|$. By a standard argument of successive averaging (cf.\ the proofs of \cite[Lemmas 8.3.2 and 8.3.3]{KadRing}), we deduce from \eqref{tgamma}  the existence of $m^{2k}$ unitaries $v_1,\ldots,v_{m^{2k}}$ such that
\begin{equation*}
\Big\|\sum_{i=1}^{m^{2k}} t_iv_iav_i^*-z\Big\| \leq 2\gamma^{k} \|a\|,
\end{equation*}
for some $z\in Z(A)$ and some scalars $t_i\in [0,1]$ such that $\sum_{i=1}^{m^{2k}} t_i=1$. In this way, we
extend \eqref{tgamma} to all $a\in A$ at the expense of changing $(m,\gamma)$ for $(m^{2k},2\gamma^k)$ (where $k$ is chosen so that $2\gamma^k<1$).
Henceforth, let us instead assume, without loss of generality,  that the constants $(m,\gamma)$ are such that \eqref{tgamma} is valid for all $a\in A$.

Let $a\in A$.
Then there exists $z_1\in Z(A)$ such that
\[
\Big\|\sum_{i=1}^mt_iu_iau_i^*-z_1\Big\|\leq \gamma\|a\|,
\]
for some unitaries $u_1,\ldots,u_m\in A$ and scalars $t_1,\ldots,t_m\in [0,1]$
such that $\sum_{i=1}^m t_i=1$. Set $a_1:=\sum_{i=1}^mt_iu_iau_i^*$ so that
$\|a_1-z_1\|\leq \gamma \|a\|$.
 Applying the same argument to
$a_1-z_1$ we find $z_2\in Z(A)$, and  a convex combination of  $m$ unitary conjugates of $a_1-z_1$, call it $b_2$, such that
\[
\|b_2-z_2\|\leq \gamma \|a_1-z_1\|\leq \gamma^2\|a\|.
\]
Notice that $b_2=a_2-z_1$, where $a_2$ is a convex combination of $m$ unitary conjugates of $a_1$ (whence, also a convex combination of $m^2$ unitary conjugates of $a$). Then
\[
\|a_2-z_1-z_2\|\leq \gamma^2\|a\|.\]
Continuing this process ad infinitum  we
find  $a_k\in D(a)$ and $z_k\in Z(A)$ for $k=1,2,\ldots$ such that
$a_k$ is a convex combination of $m$ unitary conjugates of $a_{k-1}$ and
\[
\Big\|a_k-\sum_{i=1}^k z_i\Big\|\leq \gamma^k\|a\|
\]
for all $k\geq 1$. For each $k\geq 1$ we have that
\[
\|z_k\| \leq \gamma^k\|a\|+\Big\|a_k-\sum_{i=1}^{k-1}z_i\Big\|,
\]
and since $a_k-\sum_{i=1}^{k-1}z_i$ is a convex combination of unitary conjugates of $a_{k-1}-\sum_{i=1}^{k-1}z_i$,
\begin{align*}
\|z_k\|&\leq  \gamma^k\|a\|+\Big\|a_{k-1}-\sum_{i=1}^{k-1}z_i\Big\|\\
&\leq \gamma^k\|a\|+\gamma^{k-1}\|a\|.
\end{align*}
It follows that $\sum_{i=1}^\infty z_i$ is a convergent series. Define  $\Phi(a):=\sum_{i=1}^\infty z_i$. Let us show that $\Phi$ is as desired. We have that
\[
\|a_k-\Phi(a)\|\leq \Big\|a_k-\sum_{i=1}^k z_i\Big\|+\sum_{i>k} \|z_i\|\leq
2\|a\|\frac{\gamma^k}{1-\gamma}.
\]
Recall  that $a_k$ is a convex combination of $m^k$ unitary conjugates of $a$.
Notice also that the rightmost side  tends to 0 as $k\to \infty$. This shows that for each $\epsilon>0$ there exists $n$ such that $\|a'-\Phi(a)\|\leq \epsilon \|a\|$ for all $a\in A$, where $a'$
is a convex combination of $n$ unitary conjugates of $a$. It remains to show that this convex combination may be chosen to be an average (for a larger $n$).
Let $\epsilon>0$.
Pick $n\in \NN$ such that for any $a\in A$ we have
\[
\Big\|\sum_{i=1}^n \lambda_i v_iav_i^*-\Phi(a)\Big\| \leq \frac{\epsilon}2 \|a\|,
\]
for some  $v_1,\dots,v_n \in \mathcal U(A)$ and $\lambda_1,\dots,\lambda_n \geq 0$ such that $\sum_{i=1}^n \lambda_i = 1$. Now let
$N \geq 2n/\epsilon$.
We can find non-negative rational numbers of the form $\mu_i = p_i/N$ for $i=1,\dots,n$, such that
\[ \sum_{i=1}^n \mu_i = 1 \quad \text{and}\quad |\mu_i-\lambda_i| < \frac1N, \quad i=1,\dots,n. \]
(To find such $\mu_i$, first set $p_1$ to be the greatest integer such that $\frac{p_1}N \leq \lambda_1$; then having picked $p_1,\dots,p_{i-1}$, pick $p_i$ to be the greatest integer such that $\frac{p_1+\cdots+p_i}N \leq \lambda_1 + \cdots + \lambda_i$.)
Let $u_1,\dots,u_N$ be given by listing each unitary $v_i$ a total of $p_i$ times, so that
\[ \sum_{i=1}^N \frac1N u_iau_i^* = \sum_{i=1}^n \mu_i v_iav_i^*. \]
Then
\begin{align*}
	\Big\|\sum_{i=1}^N \frac1N u_iau_i^*-\Phi(a)\Big\|
	&= \Big\|\sum_{i=1}^n \mu_i v_iav_i^*-\Phi(a)\Big\| \\
	&\leq \Big\|\sum_{i=1}^n (\mu_i-\lambda_i)v_iav_i^*\Big\|
	+ \Big\|\sum_{i=1}^n \lambda_i v_iav_i^*-\Phi(a)\Big\| \\
	&\leq \frac nN\|a\| + \frac\epsilon2\|a\| \\
	&\leq \epsilon\|a\|.
\end{align*}
Thus, $N$ is as desired.
\end{proof}
\begin{remark}\label{PhiR}
The map  $\Phi$ in Theorem \ref{newUDP-equiv} (iv) clearly satisfies that $\Phi(a)\in D_A(a)\cap Z(A)$ for all $a\in A$. Hence, the uniform Dixmier property implies  the Dixmier property. Moreover,
if   $A$ has the singleton Dixmier property, then  $\Phi$   must be the centre-valued trace.  That  Theorem \ref{newUDP-equiv} (ii) implies the Dixmier property has been used many times before to show the Dixmier property (e.g.,  \cite{DixvN,Powers}).
\end{remark}

We will find it useful to keep track of the constants $(m,\gamma)$ such that Theorem \ref{newUDP-equiv} (ii) is satisfied.
If there exist $m\in \NN$ and $0<\gamma<1$ such that for every self-adjoint $a\in A$ we have that
\begin{equation}\label{mgamma}
	\Big\|\sum_{i=1}^m \frac1m u_iau_i^*-z\Big\| \leq \gamma\|a\|,
\end{equation}
for some $z\in Z(A)$ and some $u_1,\dots,u_m \in \mathcal U(A)$, then
we say  that $A$ has the uniform Dixmier property \emph{with constants} $(m,\gamma)$.

\begin{remark}
\label{rmk:UDPexamples}
Some examples of $C^*$-algebras with the uniform Dixmier property are all von Neumann algebras and $C^*_r(\mathbb F_2)$.
Von Neumann algebras have the uniform Dixmier property
since condition (ii) of Theorem \ref{newUDP-equiv} follows from \cite[Lemma 1 of \S III.5.1]{DixvN}. In this case we have constants $(m,\gamma)=(2,3/4)$ (this can be somewhat improved; see Theorem \ref{Wconstants} below).
In particular, all finite dimensional $C^*$-algebras have the uniform Dixmier property with constants $m=2$ and $\gamma=3/4$.
For $C^*_r(\mathbb F_2)$, Powers's original argument that $C^*_r(\mathbb F_2)$ is simple can be used to derive explicit constants, and more generally, $C^*_r(G)$ has the uniform Dixmier property for any Powers group $G$ as defined in \cite[p.\ 244]{delaHarpeSkandalis}.
See Example \ref{eg:PowersGroups}.
\end{remark}

There have been significant recent advances in the understanding of when $C^*_r(G)$ has the properties of simplicity and of unique trace (for a discrete group $G$) \cite{BKKO,Haagerup,KalKen,Kennedy,LeBoudec}; in particular, if $C^*_r(G)$ is simple, then it also has a unique trace.
Therefore, simplicity and the Dixmier property coincide for $C^*_r(G)$; it turns out that, in fact, the Dixmier property is witnessed using only group unitaries to do the averaging (\cite[Theorem 4.5]{Haagerup} or \cite[Theorem 5.3]{Kennedy}).
However, it is not clear when $C^*_r(G)$ has the \emph{uniform} Dixmier property.

\begin{question}
Is there a discrete group $G$ for which $C^*_r(G)$ has the Dixmier property (i.e., is simple), but not the uniform Dixmier property?
Is the uniform Dixmier property for $C^*_r(G)$ the same as being able to average uniformly using group unitaries?
\end{question}

In Corollary \ref{cor:AF} below we show that all AF $C^*$-algebras with the Dixmier property have the uniform Dixmier property.
In Section \ref{sec:SimpleUltrapower}  we show that  all $C^*$-algebras with the Dixmier property and finite radius of comparison-by-traces have the uniform Dixmier property.
More examples, and explicit constants, are discussed in Section \ref{sec:ExplicitConstants}.

\begin{thm}\label{quotprod}
	Let $m\in \NN$ and $0<\gamma<1$.
\begin{enumerate}[label=\emph{(\roman*)}]
	\item	
If $A$ is a unital $C^*$-algebra  with the uniform Dixmier property  with constants $(m,\gamma)$, then    all of the quotients of $A$ have the uniform Dixmier property, also with constants $(m,\gamma)$.

\item
If $A_1,A_2,\ldots$ are unital $C^*$-algebras with the uniform Dixmier property with constants $(m,\gamma)$, then $\prod_{n=1}^\infty A_n$ has the uniform Dixmier property, also with constants $(m,\gamma)$.
\end{enumerate}
\end{thm}

\begin{proof}
This is straightforward.

(i): For every self-adjoint $a\in A/I$ we can find a self-adjoint
lift $\tilde a\in A$ with the same norm. Then there exist unitaries $u_1,\ldots,u_m\in \mathcal U(A)$
such that \eqref{mgamma} holds for $\tilde a$. Passing to the quotient $A/I$ we get the same for $a$.

(ii):  Let $a=(a_n)_n \in \prod_n A_n$ be self-adjoint.
For each $n$ we may find $m$ unitaries $u_{1,n},\dots,u_{m,n} \in \mathcal U(A_n)$  and $z_n\in Z(A)$ such that
\[
\Big\|\sum_{i=1}^m \frac1m u_{i,n}a_nu_{i,n}^* - z_n\Big\| \leq \gamma\|a_n\|.
\]
Let $u_i=(u_{i,n})_n$ for $i=1,\ldots,m$  and define $z:=(z_n)_{n} \in \prod_{n=1}^\infty Z(A_n)$
(note that the sequence $(z_n)_n$ is bounded since $\Vert z_n\Vert \leq(1+\gamma)\Vert a\Vert$ for all $n$).
Then
\[ \Big\|\sum_{i=1}^m \frac1m u_iau_i^* - z\Big\|  \leq \gamma\|a\|, \]
as desired.
\end{proof}

It will be convenient in the proof of Proposition \ref{indlimits} below to use the following notation from \cite{Arch-PLMS}: For a unital $C^*$-algebra $A$ and a subgroup $\mathcal V$ of $\mathcal U(A)$, ${\rm Av}(A,\mathcal V)$ is the set of all mappings (called \emph{averaging operators}) $\alpha: A\to A$ which can be defined by an equation of the form
$$\alpha(a) = \sum^n_{j=1}\lambda_ju_jau^*_j  \qquad (a\in A),$$
where $n\in \NN$, $\lambda_j>0$, $u_j\in \mathcal V$ ($1\leq j\leq n$), and $\sum^n_{j=1}\lambda_j=1$. Elementary properties of such mappings $\alpha$ are described in \cite[2.2]{Arch-PLMS}.

\begin{prop}\label{indlimits}
Let $(A_k)_{k=1}^\infty$  be an increasing sequence of $C^*$-subalgebras of $A$ whose union is dense in $A$, all containing the unit. Suppose that
$A_k$ has the singleton Dixmier property for all $k$. The following are equivalent:
\begin{enumerate}[label=\emph{(\roman*)}]
\item
$A$ has the Dixmier property.
\item
The limit $\lim_{k\to \infty} R_k(a)$ exists for all $a\in \bigcup_{k=1}^\infty A_k$,
where $R_k$ denotes the centre-valued trace on $A_k$ for all $k$.
\item
$A$ has the singleton Dixmier property and
\[
R(a)=\lim_{k\to \infty} R_k(a)\hbox{ for all  $a\in \bigcup_{k=1}^\infty A_k$},
\]
where $R$ denotes the centre-valued trace on $A$.
\end{enumerate}

\end{prop}
Note that an inductive limit of $C^*$-algebras with the singleton Dixmier property need not have the Dixmier property (e.g., there exist   simple, unital AF algebras with more than one tracial state).

\begin{proof}
Glimm's argument for UHF algebras \cite[Lemma 3.1]{Glimm} shows that $\bigcup_{k\geq1}\mathcal U(A_k)$ is norm-dense in $\mathcal U(A)$ (in brief, if $a_n\to u$ then
$a_n(a^*_na_n)^{-1/2}\to u$). Since multiplication is jointly continuous for the norm-topology on $A$, it follows that, for all $a\in A$,
\begin{equation}
	\label{eq: unitary-dense}
	D_A(a) = \overline{ \bigcup_{k\geq1}\big\{\alpha(a) : \alpha\in {\rm Av}\big(A,\mathcal U(A_k)\big)\,\big\}}.
\end{equation}
We shall use this repeatedly.

(i)$\Rightarrow$(iii): Let us first show that $A$ has the singleton Dixmier property.
Suppose that    $z_1,z_2\in D_A(a)\cap Z(A)$ for some $a\in A$. Let $\epsilon>0$. By \eqref{eq: unitary-dense}, there exists $n\in\NN$ and $\alpha,\beta\in {\rm Av}(A, \mathcal{U}(A_n))$ such that
$$ \Vert z_1-\alpha(a)\Vert < \frac{\epsilon}{4} \text{ and } \Vert z_2-\beta(a)\Vert < \frac{\epsilon}{4}.$$
Enlarging $n$ if necessary, we can find   $b\in A_n$ such that $\Vert a-b\Vert< \epsilon/4$. Notice then that $\Vert z_1-\alpha(b)\Vert <\epsilon/2$.
Since $z_1$ is invariant under conjugation by unitary elements of $A$, $\Vert z_1-R_n(\alpha(b))\Vert \leq\epsilon/2$.
But $R_n$ is constant on Dixmier sets in $A_n$ and so $R_n(\alpha(b))=R_n(b)$. Thus
$$
\Vert z_1 - R_n(b)\Vert \leq \frac{\epsilon}{2} \text{  and similarly  } \Vert z_2 - R_n(b)\Vert \leq\frac{\epsilon}{2}.
$$
It follows that $\Vert z_1-z_2\Vert\leq \epsilon$ and hence that $z_1=z_2$, as required.

Let $R:A\to Z(A)$ be the unique centre-valued trace on $A$. Let $k\geq1$, $a\in A_k$ and $\epsilon>0$. By \eqref{eq: unitary-dense}, there exists $M\geq k$ and $\alpha\in {\rm Av}(A, \mathcal{U}(A_M))$
such that $\Vert R(a) - \alpha(a)\Vert <\epsilon/2$. For each $n\geq M$, there exists $\beta_n\in {\rm Av}(A_n, \mathcal{U}(A_n))$ such that
$$\Vert R_n(\alpha(a))-\beta_n(\alpha(a))\Vert < \frac{\epsilon}{2}.$$
Since $R_n$ is constant on Dixmier sets in $A_n$, $R_n(\alpha(a))= R_n(a)$, and since $R(a)\in Z(A)$, $\Vert R(a)-\beta_n(\alpha(a))\Vert <\epsilon/2$. Hence
$$\Vert R(a)-R_n(a)\Vert
 \leq \Vert R(a)-\beta_n(\alpha(a))\Vert +\Vert\beta_n(\alpha(a))-R_n(\alpha(a)) \Vert
 < \epsilon.$$
Thus $R_n(a)\to R(a)$ as $n\to\infty$.

(iii)$\Rightarrow$(ii) is obvious.

(ii)$\Rightarrow$(i): Let $k\geq 1$ and $a\in A_k$. Then (ii) yields $z\in A$ such that, for $n\geq k$,
$R_n(a)\to z$ as $n\to \infty$.
Since $R_n(a)\in Z(A_n)$, $z$ belongs to the relative commutant of $\cup_{j\geq k}A_j$ in $A$ and hence $z\in Z(A)$. Since $R_n(a)\in D_{A_n}(a)\subseteq D_A(a)$ ($n\geq k$), $z\in D_A(a)$.
Thus, by \cite[Lemma 2.8]{Arch-PLMS},   $A$ has the Dixmier property.
\end{proof}

Suppose that $A$ has the singleton Dixmier property. Let $R: A\to Z(A)$ denote its centre-valued trace.
If $A$ also has the uniform Dixmier property   then by Theorem \ref{newUDP-equiv} (iv) (applied to $a-R(a)$) and Remark \ref{PhiR}, there exist $M\in \NN$ and $0<\Upsilon<1$ such that for every self-adjoint $a\in A$ we have that
\begin{equation}\label{MUpsilon}
\Big\|\sum_{i=1}^M\frac1Mu_iau_i^*-R(a)\Big\|\leq \Upsilon\|a-R(a)\|
\end{equation}
for some $u_1,\ldots,u_M\in \mathcal U(A)$.
We will find it necessary to keep track of these constants in the theorem below, so we will say in this case that   $A$ has the uniform singleton Dixmier property with constants $(M,\Upsilon)$.

\begin{example}\label{eg:PowersGroups}
By \cite[Lemma 1 and Proposition 3]{delaHarpeSkandalis},  $C^*_r(G)$ has the uniform single Dixmier property with constants $(M,\Upsilon)=(3,0.991)$ for any Powers group $G$ as defined in \cite[p.\ 244]{delaHarpeSkandalis}.
\end{example}

Note that if $A$ has the singleton Dixmier property, then $A$ has the uniform singleton Dixmier property if and only if \eqref{MUpsilon} holds for every self-adjoint $a\in A$ such that $R(a)=0$. But, since tracial states are constant on Dixmier sets, $T(A)=\{\phi\circ R:\phi\in S(Z(A))\}$ and hence $R(a)=0$ if and only if $\tau(a)=0$ for all $\tau\in T(A)$.
In turn, \cite[Proposition 2.7]{Cuntz-Ped} tells us that $\tau(a)=0$ for all $\tau\in T(A)$ if and only if $a \in \overline{[A,A]}$.
Thus, if $A$ has the singleton Dixmier property, then $A$ has the uniform singleton Dixmier property if and only if \eqref{MUpsilon} holds for every self-adjoint $a \in \overline{[A,A]}$.

As with the uniform Dixmier property constants, if $A$ has the uniform singleton Dixmier property with constants $(M,\Upsilon)$, then it  also has the uniform singleton Dixmier property with constants $(M^k,\Upsilon^k)$ ($k =2,3,\ldots)$.
The constants $(m,\gamma)$ for which we have \eqref{mgamma} may not satisfy \eqref{MUpsilon}, nor vice versa. However, we do have the following.

\begin{lemma} \label{lemma UDP to USDP}
Let $A$ be a unital $C^*$-algebra with the singleton Dixmier property.
\begin{enumerate}[label=\emph{(\roman*)}]
\item
If $A$ has the uniform Dixmier property with constants $(m,\gamma)$ then $A$ has the uniform singleton Dixmier property with constants $M=m^k$ and $\Upsilon=2\gamma^k$ for all natural numbers $k$ such that $2\gamma^k<1$.
\item
If $A$ has the uniform singleton Dixmier property with constants $(M,\Upsilon)$ then $A$ has the uniform Dixmier property with constants $m=M^k$ and $\gamma=2\Upsilon^k$ for all natural numbers $k$ such that $2\Upsilon^k<1$.
\end{enumerate}
\end{lemma}

\begin{proof}
(i):
Since $A$ has the uniform Dixmier property with constants $(m^k,\gamma^k)$ for all $k\in \NN$, it suffices to show that if $\gamma<1/2$ then $A$ has  the uniform singleton Dixmier property with constants $M=m$ and $\Upsilon=2\gamma$. Let us prove this.
Let $h=h^*\in A$. Then $h-R(h)$ is self-adjoint (where $R$ is the centre-valued trace). Hence there exist $z\in Z(A)$ and $u_1,\ldots,u_M\in \mathcal U(A)$ such that
$$ \Big\|\sum_{i=1}^M\frac1Mu_ihu_i^* -R(h) -z\Big\| =\Big\|\sum_{i=1}^M\frac1Mu_i(h-R(h))u_i^*-z\Big\|\leq \gamma\|h-R(h)\|.$$
Since $R$ is contractive, tracial and fixes elements of $Z(A)$, $\|z\|\leq \gamma\|h-R(h)\|$.
Hence
\[
\Big\|\sum_{i=1}^M\frac1Mu_ihu_i^*-R(h)\Big\|\leq 2\gamma\|h-R(h)\|.
\]

(ii):
This is immediate since we always have $\|a-R(a)\|\leq 2\|a\|$.
\end{proof}

\begin{thm}
Let $A_1,A_2,\ldots$ be unital $C^*$-algebras with the uniform  singleton  Dixmier property, all of them  satisfying \eqref{MUpsilon} for some constants $(M,\Upsilon)$.  Let $A=\lim A_i$ be a unital inductive limit $C^*$-algebra.
If $A$ has the Dixmier property, then it   has the uniform singleton  Dixmier property with constants $(M,\Upsilon')$ for any $\Upsilon<\Upsilon'<1$.
\end{thm}

\begin{proof}
The uniform singleton Dixmier property, and indeed the constants $(M,\Upsilon)$, pass to quotients (by the same proof as for Theorem \ref{quotprod} (i), using Theorem \ref{thm:Lifting} in place of lifting self-adjoint elements to self-adjoint elements); thus, we may reduce to the case that the connecting maps of the inductive limit are inclusions. So let us assume that the $C^*$-algebras  $(A_k)_{k=1}^\infty$  form an increasing sequence of subalgebras of $A$ whose union is dense in $A$. We denote the centre-valued trace on $A_k$
by $R_k$. By Proposition \ref{indlimits}, $A$ has the singleton Dixmier property. We denote its centre-valued trace by $R$.

Let $a\in A$ be a self-adjoint contraction with $R(a)=0$. Let $\epsilon>0$.
Find a self-adjoint contraction $b\in A_k$, for $k$ large enough, such that $\|a-b\|<\epsilon$.
 Find $n>k$ such that $\|R_n(b)-R(b)\|<\epsilon$
(its existence is guaranteed by Proposition \ref{indlimits}).
Thus,
\[ \|R_n(b)\| \leq \|R_n(b)-R(b)\|+\|R(b-a)\| < 2\epsilon. \]
Since $A_n$ has the uniform singleton Dixmier property
with constants $(M,\Upsilon)$, we have that
\[
\Big\|\sum_{i=1}^M\frac1M u_ibu_i^*-R_n(b)\Big\|\leq \Upsilon\|b-R_n(b)\|
\]
for some unitaries $u_1,\ldots,u_M\in \mathcal U(A_n)$.
Hence,
\begin{align*}
\Big\|\sum_{i=1}^M\frac1M u_iau_i^*\Big\|
&\leq
\|a-b\|+\Big\|\sum_{i=1}^M\frac1M u_ibu_i^*-R_n(b)\Big\|+\|R_n(b)\| \\
&\leq \epsilon+\Upsilon\|b-R_n(b)\|+2\epsilon \\
&\leq \Upsilon(1+2\epsilon)+3\epsilon.
\end{align*}
Thus, $A$ has the uniform singleton Dixmier property with constants $(M,\Upsilon(1+2\epsilon)+3\epsilon)$ for any sufficiently small $\epsilon>0$.
\end{proof}

\begin{cor}\label{cor:AF}
All unital AF $C^*$-algebras with the Dixmier property have the uniform singleton Dixmier property with constants $M=4$ and  $1/2<\Upsilon<1$ (i.e., satisfy \eqref{MUpsilon} for $M=4$ and any $1/2<\Upsilon<1$).
\end{cor}	
\begin{proof}
Finite dimensional $C^*$-algebras have the uniform singleton Dixmier property with constants $M=4$ and $\Upsilon=1/2$ by Proposition \ref{prop:M_nExplicitConstants} below.
\end{proof}	

Necessary and sufficient conditions for a unital AF $C^*$-algebra to have the Dixmier property have been given in \cite[Theorem 6.6]{RJA}. The example in \cite[Example 6.7]{RJA} shows how these conditions can be verified by using a Bratteli diagram.

In the following, we let $\omega$ be a free ultrafilter on $\NN$ and denote by $A_\omega$ the ultrapower of $A$ under $\omega$.
Generally, many of the arguments used with sequence algebras $\prod_n A_n/\bigoplus_n A_n$ also work with $A_\omega$ (and more generally, ultrapowers $\prod_\omega A_n$); for example we could have used ultraproducts in Theorem \ref{quotprod} instead of sequence algebras.
However, $A_\omega$ has some advantages in terms of its size.
For example, if $A$ is simple and purely infinite then $A_\omega$ is simple \cite[Proposition 6.2.6]{Rordam:Book}, whereas $\prod_n A/\bigoplus_n A$ has a maximal ideal corresponding to each free ultrafilter.
Likewise, if $A$ has a unique trace then $A_\omega$ has a unique distinguished trace (which is potentially unique -- see Theorem \ref{thm:SimpleUDPEquiv}), whereas $\prod_n A/\bigoplus_n A$ has a (distinguished) trace corresponding to each free ultrafilter.
For more about ultrapowers, see \cite[Section 3]{KirRor}.
The following theorem is a standard  application of ultraproducts.

\begin{thm}\label{UDP-ultra-alg}
Let $A$ be a unital $C^*$-algebra.
The following are equivalent.
\begin{enumerate}[label=\emph{(\roman*)}]
\item
$A$ has the uniform Dixmier property.
\item
$A_\omega$ has the Dixmier property and $Z(A_\omega) = Z(A)_\omega$.
\end{enumerate}
\end{thm}

\begin{proof}
(i)$\Rightarrow$(ii): Let $m\in \NN$ and $0<\gamma<1$ be such that $A$ has the uniform Dixmier property.
By Theorem \ref{quotprod} (i), $\ell^\infty(A)$ has the uniform Dixmier property (with the same constants), and then by Theorem \ref{quotprod} (ii), so does the quotient $A_\omega$.
Moreover, since
$Z(\ell^\infty(A))=\ell^\infty(Z(A)) $, and $\ell^\infty(A)$
has the centre-quotient property (since it has the Dixmier property),  $\ell^\infty(Z(A))$
is mapped onto the centre of $A_\omega$ by the quotient map. Thus,
$Z(A)_\omega=Z(A_\omega)$.\footnote{In fact, one only needs the (not necessarily uniform) Dixmier property to get $Z(A_\omega)=Z(A)_\omega$, by Proposition \ref{InnerDerivProp} below and the fact that $K(A)\leq 1$ when $A$ has the Dixmier property.}

(ii)$\Rightarrow$(i):
Suppose that (ii) holds and, for a contradiction, that (i) does not.
Using Theorem \ref{newUDP-equiv} (iii)$\Rightarrow$(i), we have that condition (iii) of Theorem \ref{newUDP-equiv} does not hold, and in particular it does not hold for $\gamma=1/2$.
Thus, for each $n\geq1$ there exists $a_n \in A$ such that $\|a_n\|=1$ and for all $u_1,\dots,u_n \in \mathcal U(A)$ and $t_1,\dots,t_n \in [0,1]$ with $\sum_{i=1}^n t_i=1$,
\[ d\Big(\sum_{i=1}^n t_i u_ia_nu_i^*,Z(A)\Big) \geq \frac12. \]
Let $a \in A_\omega$ be the element represented by the sequence $(a_n)_n$. Since $A_\omega$ has the Dixmier property, there exist $u_1,\dots,u_k\in \mathcal U(A_\omega)$, $t_1,\dots,t_k\in [0,1]$ with $\sum_{i=1}^k t_i=1$, and $z \in Z(A_\omega)$ such that
\[
\Big\|\sum_{i=1}^k t_i u_iau_i^* - z\Big\| < \frac{1}{2}.
\]
Since $Z(A)_\omega=Z(A_\omega)$, we can lift $z$ to a bounded sequence $(z_n)_n$ from $Z(A)$.
We may also lift each $u_i$ to a sequence $(u_{i,n})_n$ from $\mathcal U(A)$ (either by using \cite[Proposition 2.5]{Arch-PLMS} in the initial choice of the elements $u_i$ or the fact that unitaries from $A_\omega$ always lift to a sequence of unitaries).
We have
\[ \lim_{n \to \omega} \Big\|\sum_{i=1}^k t_i u_{i,n}a_nu_{i,n}^* - z_n\Big\| < \frac{1}{2}. \]
In particular, for some $n \geq k$ we must have
\[ \Big\|\sum_{i=1}^k t_i u_{i,n}a_nu_{i,n}^* - z_n\Big\| < \frac{1}{2}, \]
which gives a contradiction.
\end{proof}

\begin{remark}
\label{rmk:UDP-ultra-class}
The argument in the previous proof shows, more generally, that for a class $\mathcal C$ of $C^*$-algebras, the following are equivalent:
\begin{enumerate}[label=(\roman*)]
\item
There exist constants $(m,\gamma)$ such that every algebra $A$ in $\mathcal C$ has the uniform Dixmier property with constants $(m,\gamma)$
\item
For every sequence $(A_n)_{n=1}^\infty$ from $\mathcal C$, $\prod_\omega A_n$ has the Dixmier property and $Z(\prod_\omega A_n) = \prod_\omega Z(A_n)$.
\end{enumerate}
\end{remark}

For a $C^*$-algebra $A$, the condition $Z(A_\omega) = Z(A)_\omega$ is related to norms of inner derivations, as follows.
Firstly, recall that the triangle inequality shows that $\Vert {\rm ad}(a)\Vert \leq 2d(a,Z(A))$, where ${\rm ad}(a)$ is the inner derivation of $A$ induced by $a\in A$ (that is, ${\rm ad}(a)(x) := xa-ax$).
In the reverse direction, $K(A)$ is defined to be the smallest number in $[0,\infty]$ such that
$d(a,Z(a)) \leq K(A)\Vert {\rm ad}(a)\Vert$ for all $a\in A$  (\cite{RJA}).  It was shown in the proof of \cite[Theorem 5.3]{KLR} that $K(A)<\infty$ if and only if the set of inner derivations of $A$ is norm-closed in the set of all derivations of $A$. If $A$ is non-commutative (as we shall assume from now on in this summary) then $K(A)\geq \frac {1} {2}$. If $A$ is a von Neumann algebra (or, more generally, an $AW^*$-algebra) or a unital primitive $C^*$-algebra (in particular, a unital simple $C^*$-algebra) then $K(A)=\frac {1} {2}$ (\cite{Elliott, Gaj, BEJ, Stampfli, Zsido-norm}).
These and other such cases are covered by Somerset's characterisation for unital $A$:  $K(A)=\frac {1} {2}$ if and only if the ideal $P\cap Q\cap R$ is primal whenever $P$, $Q$ and $R$ are primitive ideals of $A$ such that $P\cap Z(A)=Q\cap Z(A) = R\cap Z(A)$ (\cite{Som}). If a unital $C^*$-algebra $A$ has the Dixmier property then $K(A) \leq 1$ (see \cite[Section 2]{RingPLMS78} and \cite[Proposition 2.4]{RJA}) (this holds more generally if $A$ is weakly central, see \cite{RJA, Som}).
On the other hand, in \cite[6.2]{KLR}, an example is given where $K(A)=\infty$.
By \cite[Corollary 4.6]{Som:JOT}, finiteness of $K(A)$ depends only on the topological space $\mathrm{Prim}(A)$.
Further information on possible values of $K(A)$ 
may be found in \cite{AKS, AS} and the references cited therein.

\begin{prop}
\label{InnerDerivProp}
Let $A$ be a $C^*$-algebra.
The following are equivalent:
\begin{enumerate}[label=\emph{(\roman*)}]
\item  $Z(A_\omega) = Z(A)_\omega$.
\item  $K(A) < \infty$.
\end{enumerate}
\end{prop}

\begin{proof}
(i)$\Rightarrow$(ii):
 Suppose that $K(A)=\infty$. For each $n\geq1$, there exists $b_n\in A$ such that
 $$ 0 < n\|{\rm ad}(b_n)\| < d(b_n,Z(A)).$$
 By scaling, we may assume that $d(b_n,Z(A)) =1$ for all $n\geq1$. Then, for each $n\geq1$, there exists $z_n\in Z(A)$ such that $\|b_n-z_n\|<2$.
 Let $c_n:=b_n-z_n$ ($n\geq1$) and let $c\in A_{\omega}$ correspond to the bounded sequence $(c_n)_n$.
 Note that
 $$ d(c_n,Z(A))= d(b_n,Z(A)) =1 \qquad (n\geq1) $$
 and $\|{\rm ad}(c_n)\| = \|{\rm ad}(b_n)\| \to 0$ as $n\to\infty$.
 For any bounded sequence $(a_n)_n$ in $A$, $\lim_{n\to\omega}\|a_nc_n-c_na_n\|=0$ and so $c\in Z(A_{\omega})$.
 On the other hand, for any bounded sequence $(y_n)_n$ in $Z(A)$, $\lim_{n\to\omega}\|c_n-y_n\|\geq1$ and so $c\notin Z(A)_{\omega}$.

\bigskip

(ii)$\Rightarrow$(i):
The containment $Z(A)_\omega \subseteq Z(A_\omega)$ is clear.
For the other way, let $b \in Z(A_\omega)$ be represented by a bounded sequence $(b_n)_n$ in $A$.
For each $n\geq1$, there exists $z_n\in Z(A)$ such that
$$ \|b_n-z_n\| \leq d(b_n,Z(A))+\frac{1}{2n} \leq K(A)\|\mathrm{ad}(b_n)\| +\frac{1}{2n}$$
and there exists $a_n\in A$ such that $\|a_n\|\leq 1$ and
$$K(A)\|{\rm ad}(b_n)\|\leq K(A)\|b_na_n-a_nb_n\| +\frac{1}{2n}.$$
Then, for all $n\geq1$,
$$\|b_n-z_n\| \leq K(A)\|b_na_n-a_nb_n\| + \frac{1}{n}.$$
Recalling that $b \in Z(A_\omega)$, we obtain that $\lim_{n\to\omega}\|b_na_n-a_nb_n\|=0$ and hence that $\lim_{n\to\omega}\|b_n-z_n\|=0$.
Since $\|z_n\|\leq 2\|b_n\|+\frac{1}{2n}$, $(z_n)_n$ is a bounded sequence and so $b\in Z(A)_{\omega}$.
\end{proof}

It is easily seen that the method of proof of Proposition~\ref{InnerDerivProp} also shows that $K(A)<\infty$ if and only if the centre of
$\ell^{\infty}(A)/c_0(A)$ is the canonical image of $\ell^{\infty}(Z(A))/c_0(Z(A))$.

\bigskip

Bearing in mind that the Dixmier property is a necessary condition for the uniform Dixmier property, we record the following simple corollary of the results in this section.

\begin{cor} \label{cor DPvUDP}
Suppose that $A$ is a unital $C^*$-algebra with the Dixmier property. The following conditions are equivalent.
\begin{enumerate}[label=\emph{(\roman*)}]
\item $A$ has the uniform Dixmier property.
\item $A_{\omega}$ has the Dixmier property.
\end{enumerate}
\end{cor}

\begin{proof}  Since $A$ has the Dixmier property, $K(A)\leq1$ (see \cite[Section 2]{RingPLMS78} or \cite[Proposition 2.4]{RJA}) and so $Z(A_\omega) = Z(A)_\omega$ by Proposition~\ref{InnerDerivProp}. The result now follows from Theorem~\ref{UDP-ultra-alg}.
\end{proof}

\begin{question}
If $A_\omega$ has the Dixmier property, does it follow that $A$ has the Dixmier property?
(In other words, by Theorem \ref{UDP-ultra-alg}, if $A_\omega$ has the Dixmier property, is $Z(A)_\omega=Z(A_\omega)$?)
\end{question}
\bigskip

\subsection{Radius of comparison-by-traces}
Let  $A$  be unital with the Dixmier property. If $A$ has strict comparison of positive elements by traces (see Remark \ref{rmk:StrictComp}), then if follows from  \cite[Theorem 1.2]{NRS} that $A$ has the uniform Dixmier property.
We now show that this holds more generally  when strict comparison by traces is replaced by finite radius of comparison-by-traces.

Let $A$ be a unital  $C^*$-algebra.
For each tracial state $\tau$  define $d_\tau:\bigcup_{n=1}^\infty M_n(A)_+ \to [0,\infty)$ by
\[ d_\tau(a) := \lim_{n\in\NN} \tau(a^{1/n}). \]
This is the dimension function associated to $\tau$ (\cite{black-hand}).

\begin{defn}
Let $r\in [0,\infty)$. Let $A$ be a unital $C^*$-algebra. Let us say that $A$ has \emph{radius of comparison-by-traces at most $r$} if
for all positive elements $a,b\in \bigcup_{k=1}^\infty M_k(A)$, with $b$ a full element, if
\begin{equation}\label{trc}
	d_\tau(a)+r'\leq d_\tau(b)
\end{equation}
for all $\tau\in T(A)$ and some $r'>r$, then $a$ is Cuntz below $b$. (Recall that $a$ is said to be Cuntz below $b$ if $d_nbd_n^*\to a$ for some sequence $(d_n)$ in $\bigcup_{k=1}^\infty M_k(A)$.) The radius of comparison-by-traces of $A$ is the minimum $r$ such that $A$ has radius of comparison-by-traces at most $r$. If no such $r$ exists then
we say that $A$ has infinite radius of comparison-by-traces.
\end{defn}

In  \cite{BRTTW} the radius of comparison of $A$ is defined as above, except that  in \eqref{trc} $\tau$
ranges through all 2-quasitraces of $A$ normalised at the unit.  We use the name ``radius of comparison-by-traces"
to emphasise that the comparison of $a$ and $b$ in \eqref{trc} is done only on tracial states. Clearly, the radius of comparison-by-traces dominates the radius of comparison. If the  $C^*$-algebra $A$ is exact, then by \cite{Haagerup:qtraces} its bounded 2-quasitraces are traces so the two numbers agree.

\begin{remark}
\label{rmk:StrictComp}
For simple $C^*$-algebras, strict comparison of positive elements by traces is the same as having radius of comparison-by-traces $0$ (by the same argument as in \cite[Proposition 6.4]{Toms:Flat}, cf.\ \cite[Proposition 3.2.4]{BRTTW}).
All $\mathcal Z$-stable $C^*$-algebras have strict comparison of positive elements (and therefore strict comparison of positive elements by traces) by \cite[Proposition 3.2 and Theorem 4.5]{Rordam:Z}.
\end{remark}

The seminal examples of simple nuclear $C^*$-algebras constructed by Villadsen in \cite{Villadsen1} and \cite{Villadsen} have nonzero radius of comparison-by-traces; variations on the first of these examples can be arranged to achieve any possible value of radius of comparison (\cite[Theorem 5.11]{Toms:CMP}).
Of particular interest here are Villadsen's second examples, which have stable rank in $\{2,3,\dots\}$; they have nonzero finite radius of comparison while being simple and having unique trace, see Remark \ref{rmk:VilladsenRadius}.

\begin{thm}\label{rcprod}
	Let $A_1,A_2,\ldots$ be unital $C^*$-algebras with radius of comparison-by-traces at most $r$ and let $A:=\prod_{i=1}^\infty A_i$. The following are true:
	\begin{enumerate}[label=\emph{(\roman*)}]
		\item	
		$A$ has radius of comparison-by-traces at most $r$.
		\item
		The convex hull of $\bigcup_{i=1}^\infty T(A_i)$ is dense in $T(A)$ in the weak$^*$-topology. (We regard
		$T(A_i)$ as a subset of $T(A)$ via the embedding induced by the quotient map $A\to A_i$.)
	\end{enumerate}
\end{thm}
\begin{proof}
	(i):	Let $K$ be the weak$^*$-closure in $T(A)$ of the convex hull of $\bigcup_{i=1}^\infty T(A_i)$.
	Let $a,b\in M_k(A)$ be positive elements, with $b$ full. Suppose that $a$ and $b$ satisfy \eqref{trc}
	for all tracial states $\tau\in K$ and some $r'>r$. We will prove that $a$ is Cuntz below $b$ (which clearly shows that $A$ has radius of comparison-by-traces at most $r$).
	Let $\epsilon>0$ and choose $r<r''<r'$. We claim that  there exists $\delta>0$
	such that
	\begin{equation}\label{trcK}
		d_\tau((a-\epsilon)_+)+r''\leq d_\tau((b-\delta)_+)\hbox{ for all $\tau\in K$.}
	\end{equation}
	Indeed, let $g_\epsilon\in C_0((0,\|a\|])_+$ be such that $g_\epsilon(t)=1$ for $t\geq \epsilon$. Then
	\[
	d_\tau((a-\epsilon)_+)\leq \tau(g_\epsilon(a))\leq d_\tau(a)\hbox{ for all $\tau\in T(A)$.}
	\]
	The function $\tau\mapsto \tau(g_\epsilon(a))+r''$ is continuous
	on $T(A)$ while $\tau\mapsto d_\tau((b-\frac 1 n)_+)$ is lower semicontinuous for all  $n$. Since
	\[
	\sup_n d_\tau\left(\left(b-\frac 1 n\right)_+\right)=d_\tau(b)>\tau(g_\epsilon(a))+r''
	\]
	for all $\tau\in K$ and $K$ is compact,   there exists $n$ such that \[
	d_\tau\left(\left(b-\frac 1  n\right)_+\right)>\tau(g_\epsilon(a))+r''
	\] for all $\tau\in K$, thus yielding the desired $\delta$. Decreasing
	$\delta$ if necessary, let us also  assume that
	$(b-\delta)_+$ is full.
	Letting $\tau$ range through $T(A_i)\subseteq K$ in \eqref{trcK}, and using that $A_i$ has radius of comparison-by-traces at most $r$, we obtain that
	$(a_i-\epsilon)_+$ is Cuntz below $(b_i-\delta)_+$ for all $i$. Hence, using \cite[Proposition 2.4]{Rordam:JFA2},
we obtain $x_i\in M_k(A_i)$ such that
	$(a_i-2\epsilon)_+=x_i^*x_i$ and $x_ix_i^*\leq Mb_i$ for all $i$, where $M>0$ is a scalar independent of $i$.
	Then $(a-2\epsilon)_+=x^*x$ and $xx^*\leq Mb$, where $x:=(x_i)_i\in \prod_{i=1}^\infty M_k(A_i)\cong M_k(A)$. Since $\epsilon$ is arbitrary, we get that
	$a$ is Cuntz below $b$, as desired.
	
	(ii): Here we follow closely arguments from \cite{NgRobert}. We first establish two claims.
	
	\emph{Claim 1}: If $a,b\in M_k(A)$ are positive elements, with $b$ full, such that
	$d_\tau(a)\leq d_\tau(b)$ for all $\tau\in K$ then $d_\tau(a)\leq d_\tau(b)$ for all $\tau\in T(A)$.
	Let us prove this.  Choose a natural number $r'> r$.
	Then
	\[
	d_\tau(a^{\oplus n})+r'\leq d_\tau(b^{\oplus n}\oplus 1_{r'})
	\]
	for all $n=1,2,\ldots$ and all $\tau\in K$.
	By the proof of (i),  $a^{\oplus n}$ is Cuntz below $b^{\oplus n}\oplus 1_{r'}$ for all $n\in \NN$. Now let $\tau\in T(A)$.  Then
	$nd_\tau(a)\leq nd_\tau(b)+r'$. Letting $n\to \infty$ we get that $d_\tau(a)\leq d_\tau(b)$, proving our claim.
	
	\emph{Claim 2}: If $a,b\in A_+$, with $b$ full, are such that $\tau(a)\leq \tau(b)$ for all $\tau\in K$ then $\tau(a)\leq \tau(b)$ for all $\tau\in T(A)$. Let us prove this. Let $\epsilon>0$.
	Since
	\[ \sigma(c) = \int_0^{\|c\|} d_\sigma((c-t)_+)\,dt, \]
	for all positive elements $c\in A$ and all $\sigma\in T(A)$ (see for example \cite[Proposition 4.2]{ERS}), one can construct positive elements $a_n,b_n$ (in matrix algebras over $A$), and find natural numbers $r_n,s_n$ such that
	\[ \lim_{n\to\infty} \frac1{r_n}d_\sigma(a_n) = \sigma((1-\epsilon)a) \]
	and
	\[ \lim_{n\to\infty} \frac1{s_n}d_\sigma(b_n) = \sigma(b) \]
	for all tracial states $\sigma$, with both sequences increasing.  Since $b$ is full, we have that
	\[
	\tau((1-\epsilon)a)\leq \tau((1-\epsilon)b)<\tau(b)\] for all $\tau\in K$.
	Using lower semi-continuity and the compactness of $K$, as in part (i), we obtain   $n\in \NN$ such that $\tau((1-\epsilon)a)\leq \frac1{s_n}d_{\tau}(b_n)$ for all $\tau\in K$. Hence,
	\[
	\frac1{r_m}d_{\tau}(a_m) \leq   \frac1{s_n}d_{\tau}(b_n), \hbox{ for all $\tau\in K$},
	\]
	for all $m$ and for all sufficiently large $n$.
	By the first claim applied to the positive elements $a_m^{\oplus s_n}$ and $b_n^{\oplus r_m}$,
	\[ \frac1{r_m}d_{\sigma}(a_m) \leq  \frac1{s_n}d_{\sigma}(b_n) \]
	for any  $\sigma\in T(A)$.
	Taking the limit as $m\to \infty$, we obtain $\sigma((1-\epsilon)a) \leq  \sigma(b)$.
	Letting $\epsilon\to 0$ proves the claim.
	
	Let us now show that $K=T(A)$.
	By the Hahn--Banach theorem, it suffices to show that for all self-adjoint $a\in A$, if $\tau(a)=0$ for all $\tau\in K$ then
	$\tau(a)=0$ for all $\tau\in T(A)$. If $a$ is a self-adjoint  such that $\tau(a)=0$ for all $\tau\in K$ then $\tau(a+t1)=\tau(t1)$ for all $\tau\in K$. Moreover, for $t>\|a\|$
	both $a+t1$ and $t1$ are positive and full. It follows by Claim 2 that
	$\tau(a+t1)=\tau(t1)$ for all $\tau\in T(A)$ and $t>\|a\|$, which yields the desired result.
\end{proof}

In the proof of the next result  we make use of Theorem \ref{0inDa}, proven in Section \ref{sec:Distance} below, and whose proof is independent from the results of this section.
Theorem \ref{0inDa} is an extension of Theorem \ref{NRS theorem} from the introduction to non-self-adjoint elements.

\begin{thm}\label{rcuzm}
	Let $r\in [0,\infty)$. There exists $M\in \NN$ such that if $A$ is a unital $C^*$-algebra with radius of comparison-by-traces at most $r$
	and $a\in A$ is  such that $0\in D_A(a)$, then
	\[
	\Big\|\frac 1 M\sum_{i=1}^M u_iau_i^*\Big\|\leq \frac 1 2\|a\|
	\]
	for some unitaries $u_1,\ldots,u_M\in A$.
\end{thm}	

\begin{proof}
	Suppose, for the sake of contradiction, that there exist unital  $C^*$-algebras $A_1,A_2,\ldots$ with radius of comparison-by-traces
	at most $r$, and  contractions  $a_n\in A_n$ such that $0\in D_A(a_n)$ for all $n$,  but  any average of
	$n$ unitary conjugates of $a_n$ has norm greater than $1/2$. Let $A:=\prod_{n=1}^\infty A_n$ and $a:=(a_n)_n\in A$. We will show that $0\in D_A(a)$ relying on Theorem \ref{0inDa} from Section \ref{sec:Distance}. To show that $0\in D_A(a)$, it suffices to check  conditions (a) and (b) of Theorem \ref{0inDa}.
	Notice that $\tau(a)=\tau(a_n)=0$ for all $\tau\in T(A_n)$ and all $n$. It follows by Theorem \ref{rcprod} (ii) that $\tau(a)=0$
	for all $\tau\in T(A)$, i.e., condition (a)   holds.
	In order to show   that $a$  satisfies condition (b), we prove that it satisfies the equivalent form (b''), stated right before the proof of Theorem \ref{0inDa}.
	Let $t',t>0$  be such $t'>t$ and let $w\in \C$.
	Since $0\in D_{A_n}(a_n)$, we have, by   condition (b'') applied to $a_n$,  that  $(\mathrm{Re}(wa_n)+t)_+$ is a full element of $A_n$ (i.e., it generates $A_n$ as a closed two-sided ideal).
	For all $\tau\in T(A_n)$ we have
	\begin{align*}
	d_\tau((\mathrm{Re}(wa_n)+t)_+)&\geq \frac{1}{|w|+t}\tau((\mathrm{Re}(wa_n)+t)_+)\\
	&\geq \frac{1}{|w|+t}\tau(\mathrm{Re}(wa_n)+t)=\frac{t}{|w|+t},
	\end{align*}
where we have used that   $d_\tau(c)\geq \tau(c)/\|c\|$ for any $c\geq 0$ in the first inequality. Choose $N\geq (2+r)(|w|+t)/t$. Then
	\[
	d_\tau((\mathrm{Re}(wa_n)+t)_+^{\oplus N})\geq 2+r.
	\]
	Since $A_n$ has radius of comparison-by-traces at most $r$, the above (including fullness of $(\mathrm{Re}(wa_n)+t)_+$) implies that $1\in A_n$ is Cuntz below $(\mathrm{Re}(wa_n)+t)_+^{\oplus N}$. Thus, there exists a partial isometry $v_n\in M_N(A_n)$ such that $1=v_n^*v_n$ and
	\[
	v_nv_n^*\leq C\cdot (\mathrm{Re}(wa_n)+t')_+^{\oplus N},
	\] where
	$C>0$ depends on $t$  and $t'$ but not on $n$. Then, setting $v:=(v_n)_n\in M_N(A)$, we get $1=v^*v$ and $vv^*\leq C\cdot (\mathrm{Re}(wa_n)+t')_+^{\oplus N}$. Hence,
	$(\mathrm{Re}(wa)+t')_+$ is full for all $t'>0$ and $w\in \mathbb C$. This proves condition (b'').  It follows  that $0\in D_A(a)$.
	Thus, there is a finite convex combination of unitary conjugates of $a$ whose  norm is less than $\frac 1 2$.
	Enlarging the number of terms if necessary,  we may assume that this convex combination is an average (see the proof of Theorem \ref{newUDP-equiv} (iii)$\Rightarrow$(iv)). So, there exist  $M\in \NN$
	and unitaries $u_1,\ldots,u_M\in A$ such that
	\[
	\Big\|\frac 1 M\sum_{i=1}^M u_iau_i^*\Big\|\leq \frac 1 2.
	\]
	We arrive at a contradiction by projecting onto $A_M$.
\end{proof}	

\begin{remark}
For the case of self-adjoint elements (which is all that is needed in the next corollary), Theorem \ref{rcuzm} can be proven using Theorem \ref{NRS theorem} in place of Theorem \ref{0inDa}.
\end{remark}

\begin{cor} \label{cor exact UDP}
	Let $r\in [0,\infty)$.	Then there exist constants $(m,\gamma)$ such that every unital $C^*$-algebra with the Dixmier property and with radius of comparison-by-traces at most $r$ has the uniform Dixmier property with constants $(m,\gamma)$. In particular, every simple unital $C^*$-algebra with at most one tracial state and radius of comparison-by-traces at most $r$ has the uniform Dixmier property with constants $(m,\gamma)$.
\end{cor}	

\begin{proof}
	Let $M\in \NN$ be as in Theorem \ref{rcuzm}.
	Suppose that $A$ is a unital $C^*$-algebra with the Dixmier property and radius of comparison-by-traces at most $r$.
	Now let $a\in A$ be a self-adjoint element and choose $z\in D_A(a)\cap Z(A)$. Then $0\in D_A(a-z)$, and so
	\[
	\Big\|\frac1M\sum_{i=1}^Mu_iau_i^*-z\Big\|\leq \frac12\|a\|,
	\]
	for some unitaries $u_1,\ldots,u_M$.
	Hence, $A$ has the uniform Dixmier property with constants $(M,1/2)$.
\end{proof}

\begin{remark}
\label{rmk:VilladsenRadius}
In \cite{Villadsen} Villadsen obtains  examples of finite, simple, unital  $C^*$-algebras with  stable rank in $\{2,3,\ldots,\infty\}$. These $C^*$-algebras are nuclear
and  have a unique tracial state (\cite[Section 6]{Villadsen}). It can be shown that the examples constructed by Villadsen with finite stable rank  have  finite  non-zero radius of comparison-by-traces.
Thus, these $C^*$-algebras have the  uniform Dixmier property, although  they fail to have   strict comparison of positive elements by traces.

Let us explain why these examples have finite radius of comparison-by-traces (which is the same as finite radius of comparison, since they are exact).
For $n \in \mathbb N$, Villadsen's algebra $A$ with stable rank $n+1$ is constructed, according to \cite[Section 3]{Villadsen}, as $A = \varinjlim A_i$ where $A_i = p_i(C(X_i)\otimes \mathcal K)p_i$, with $X_i$ a certain space of dimension $n(1+2\cdot1!+4\cdot2!+\cdots+2i\cdot i!)$ and $p_i$ a certain projection of constant rank $(i+1)!$.
We compute
\begin{align*}
\frac{\dim(X_i)-1}{2\,\mathrm{rank}(p_i)}
&\leq \frac{2n(1!+2!+\cdots+i!)}{2i!} \\
&\leq \frac{n(i-1)(i-1)!}{i!} + \frac{ni!}{i!} \\
&\leq 2n
\end{align*}
By \cite[Theorem 5.1]{Toms:CMP}, it follows that $A_i$ has radius of comparison at most $2n$.
Hence by \cite[Proposition 3.2.4]{BRTTW}, the radius of comparison of $A$ is at most $2n$.
\end{remark}
\bigskip

\subsection{$C^*$-algebras with trivial centre}
\label{sec:SimpleUltrapower}
If $A$ is a unital $C^*$-algebra with trivial centre,  then by Corollary \ref{genHZ1} $A$ has the Dixmier property if and only if we have one of the following four cases:
\begin{enumerate}[(1)]
\item
$A$ is simple and has no tracial states,
\item
$A$ is simple and has a unique tracial state,
\item
$A$ has no tracial states and a unique non-zero maximal ideal,
\item
$A$ has a unique tracial state and its trace-kernel ideal is the unique nonzero maximal ideal of $A$.
\end{enumerate}
Cases (2) and (4) have the singleton Dixmier property while cases (1) and (3) do not.
Now, since $A$ is unital and has the Dixmier property, $K(A) \leq 1 < \infty$   and so $Z(A_{\omega})=Z(A)_{\omega}=\C1$ by Proposition~\ref{InnerDerivProp}.
(That $\C_\omega = \C$ is because every bounded sequence of complex numbers has a unique limit under $\omega$, i.e., the map taking $(x_n)_{n=1}^\infty \in \prod_n \C$ to $\lim_{n\to\omega} x_n$ induces an isomorphism $\C_\omega \to \C$.)
Thus, by  Corollary~\ref{cor DPvUDP}, in order for $A$ to have the uniform Dixmier property $A_\omega$ must also fall in one of the four cases above.
In Theorem \ref{thm:SimpleUDPEquiv} below we take this analysis further  to obtain explicit  conditions for
having the uniform Dixmier property when $A$ falls in cases (2) and (4) above.

Suppose that $A$ is in either case (2) or (4).
Let $\tau$ denote the unique tracial state
of $A$.  Then $\tau$ induces a canonical tracial state $\tau_{\omega}$ on $A_\omega$, by
\[ \tau_\omega(a) := \lim_{n\to\omega} \tau(a_n), \]
whenever $a$ is represented by the sequence $(a_n)_n$.
Let
\[ J:= \{a \in A_\omega : \tau_\omega(a^*a) = 0\}, \]
the trace-kernel ideal for $\tau_{\omega}$.
Using the Kaplansky density theorem, one can see that $A_\omega/J$ is isomorphic to the tracial von Neumann ultrapower of $\pi_\tau(A)''$, where $\pi_\tau$ is the GNS representation associated to $\tau$ (\cite[Theorem 3.3]{KirRor}).
In particular, this quotient is a finite factor, and is therefore simple, so that $J$ is a maximal ideal.

In the next result, conditions (i) and (iii) are both expressed purely in terms of the $C^*$-algebra $A$. However, in order to show that these conditions are equivalent, we introduce $A_{\omega}$ so that we can apply Corollary~\ref{genHZ1} and Corollary~\ref{cor DPvUDP}.

\begin{thm}\label{thm:SimpleUDPEquiv}
Let $A$ be a $C^*$-algebra with the Dixmier property, trivial centre, and unique tracial state $\tau$.
The following are equivalent:
\begin{enumerate}[label=\emph{(\roman*)}]
\item
$A$ has the uniform Dixmier property.
\item
$\tau_\omega$ is the unique tracial state on $A_\omega$ and the trace-kernel ideal $J$ is the unique maximal ideal of $A_\omega$.
\item
Both of the following hold:
\begin{enumerate}
\item
there exists  $m\in\NN$   such that  if $a \in A$ is a self-adjoint contraction satisfying $\tau(a) = 0$ then there exist contractions $x_1,\dots,x_m \in A$ such that
\[ \Big\|a - \sum_{i=1}^m [x_i,x_i^*]\Big\| \leq (1-1/m)\|a\|, \quad \text{and} \]
\item
for every $\epsilon>0$ there exists $n\in\NN$ such that, if $a \in A_+$ is a positive contraction and $\tau(a) > \epsilon$ then there exist contractions $x_1,\dots,x_n \in A$ such that
\[ \sum_{i=1}^n x_iax_i^* = 1. \]
\end{enumerate}
\end{enumerate}
\end{thm}

\begin{proof}
Recall that since $A$ is unital and has the Dixmier property, $K(A) \leq 1 < \infty$   and so $Z(A_{\omega})=Z(A)_{\omega}=\C1$ by Proposition~\ref{InnerDerivProp}.

(i)$\Leftrightarrow$(ii):
By Corollary~\ref{cor DPvUDP}, (i) is equivalent to $A_\omega$ having the Dixmier property.
Thus, (i)$\Leftrightarrow$(ii) follows from Corollary \ref{genHZ1}.

(ii)$\Leftrightarrow$(iii):
We will first show that (a) is equivalent to $\tau_\omega$ being the unique tracial state on $A_\omega$, then that (b) is equivalent to $J$ being the unique maximal ideal of $A_\omega$.

For a unital $C^*$-algebra $B$, set $B_0$ equal to the norm-closure of the $\R$-span of the set of self-commutators $[x,x^*]$.
For a tracial state $\tau_B$ on $B$, by \cite[Theorem 2.6 and Proposition 2.7]{Cuntz-Ped}, $\tau_B$ is the unique tracial state of $B$ if and only if
\[ B_0 = \{b \in B : b \text{ is self-adjoint and }\tau_B(b)=0\}. \]

Suppose that $\tau_\omega$ is the unique tracial state on $A_\omega$ and, for a contradiction, that (a) doesn't hold.
Then for each $n\in\NN$ there exists a self-adjoint contraction $a_n \in A$ such that $\tau(a_n) = 0$ and
\begin{equation}
\label{eq:CPcontradiction}
 \Big\|a_n-\sum_{i=1}^n [x_i,x_i^*]\Big\| \geq (1-1/n)
\end{equation}
for all tuples $(x_1,\dots,x_n)$ of contractions in $A$.

Since the sequence $(a_n)_n$ is bounded, it defines a self-adjoint element $a \in A_\omega$, and this element clearly satisfies $\tau_\omega(a)=0$.
Since $\tau_\omega$ is the unique tracial state, it follows (as mentioned above) that there exist $m\in\NN $ and $y_1,\dots,y_m \in A_\omega$ such that
\[ \Big\|a-\sum_{i=1}^m [y_i,y_i^*]\Big\| < \frac12. \]
By increasing $m$ if necessary, we may assume that all of the elements $y_i$ are contractions.

Lifting each $y_i$ to a sequence $(x_{i,n})_n$ of contractions in $A$, we have for $\omega$-almost all $n \in \mathbb N$,
\[ \Big\|a_n-\sum_{i=1}^m [x_{i,n},x_{i,n}^*]\Big\| < \frac12. \]
In particular, for some $n \geq m$, we obtain a contradiction to \eqref{eq:CPcontradiction}.
This proves that if $A_\omega$ has a unique tracial state then (a) holds.

Now suppose that (a) holds, which provides a number $m$.
If $a \in A_\omega$ is a self-adjoint contraction satisfying $\tau_\omega(a)=0$, then we may lift $a$ to a sequence $(a_n)_{n=1}^\infty$ of self-adjoint elements satisfying $\tau(a_n)=0$ and $\Vert a_n\Vert \leq\Vert a\Vert$ for all $n$.
(To achieve this, we first lift $a$ to any bounded sequence of self-adjoint elements, then correct the tracial state on each element by adding an appropriate scalar, and finally scale to obtain $\Vert a_n\Vert \leq\Vert a\Vert$.)
By applying (a) to each $a_n$, we can arrive at elements $x_1,\dots,x_m \in A_\omega$ such that
\[ \Big\|a - \sum_{i=1}^m [x_i,x_i^*]\Big\| \leq (1-1/m)\Vert a\Vert. \]
In other words, this shows that $A_\omega$ satisfies (a), with $\tau_\omega$ in place of $\tau$.

Next, by iterating, we see that if $a \in A_\omega$ is a self-adjoint contraction and satisfies $\tau_\omega(a)=0$, then for any $k \in \N$, there exist $mk$ contractions $x_1,\dots,x_{mk} \in A_\omega$ such that
\[ \Big\|a - \sum_{i=1}^{mk} [x_i,x_i^*]\Big\| \leq (1-1/m)^k\Vert a\Vert. \]
It follows that $a \in (A_\omega)_0$.
By $\R$-linearity,
\[ (A_{\omega})_0 = \{a \in A_{\omega} : a \text{ is self-adjoint and }\tau_{\omega}(a)=0\} \]
and hence $\tau_\omega$ is the unique tracial state of $A_\omega$.

Now, suppose that $J$ is the unique maximal ideal of $A_\omega$ and let us prove that (b) holds.
Suppose for a contradiction that (b) doesn't hold.
Then there exists $\epsilon>0$ and, for each $n \in \mathbb N$, a contraction $a_n \in A_+$ such that $\tau(a_n) > \epsilon$ yet
\begin{equation}
\label{NotFullEq}
\sum_{i=1}^n x_ia_nx_i^* \neq 1
\end{equation}
for all contractions $x_1,\dots,x_n \in A$.

Define $a \in A_\omega$ by the sequence $(a_n)_n$, so that $\tau_\omega(a) \geq \epsilon$.
Since $J$ is the unique maximal ideal of $A_\omega$, the ideal generated by $a$ is $A_{\omega}$.
Hence, there exists $y_1,\dots,y_m \in A_\omega$ such that
\[ \sum_{i=1}^m y_iay_i^* = 1, \]
and by increasing $m$ if necessary we may assume that all of the elements $y_i$ are contractions.
Lift each $y_i$ to a sequence $(y_{i,k})_k$ of contractions.
Then, for $\omega$-almost all indices $k$, we have
\[ \Big\|\sum_{i=1}^m y_{i,k}a_ky_{i,k}^* - 1\Big\| < \frac12. \]
Pick $k \geq 2m$ such that this holds.
Set
\[ b:=\sum_{i=1}^m y_{i,k}a_ky_{i,k}^*, \]
so that the spectrum of $b$ is contained in $[1/2,3/2]$.
Therefore, $(2b)^{-1/2}y_{i,k}$ is a contraction, and
\[ 1 = 2\sum_{i=1}^m (2b)^{-1/2}y_{i,k}a_ky^*_{i,k}(2b)^{-1/2}, \]
in contradiction to \eqref{NotFullEq}.

Finally assume that (b) holds, and we'll prove that $J$ is the unique maximal ideal of $A_\omega$.
Let $I$ be an ideal of $A_\omega$, such that $I \not\subseteq J$.
Therefore, $I$ contains a positive contraction $a \not\in J$, so that $r:=\tau_\omega(a) > 0$.
Using $\epsilon:=r/2$, we get some $n \in \mathbb N$ from (b).

We may lift $a$ to a sequence $(a_k)_k$ of positive contractions such that $\tau(a_k) > r/2$ for each $k$.
Then for each $k$ there exist $n$ contractions $x_{1,k},\dots,x_{n,k} \in A$ such that
\[ 1 = \sum_{i=1}^n x_{i,k} a_k x_{i,k}^*. \]
Letting $x_i \in A_\omega$ be the element represented by the sequence $(x_{i,k})_k$, we have
\[ 1 = \sum_{i=1}^n x_i a x_i^* \in I, \]
and therefore $I=A_\omega$.
This shows that $J$ is the unique maximal ideal of $A_\omega$.
\end{proof}
Under the hypotheses of Theorem~\ref{thm:SimpleUDPEquiv}, it is unclear whether there is any relation between the conditions that $\tau_\omega$ is the unique tracial state on $A_\omega$ (equivalently, condition (a)) and that $J$ is the unique maximal ideal of $A_\omega$ (equivalently, condition (b)).

\begin{question}
Does condition (a)  in Theorem~\ref{thm:SimpleUDPEquiv} (iii) imply condition (b), or vice versa?
\end{question}

In \cite[Theorem 1.4]{Robert}, LR showed that there is a simple unital (and nuclear, in fact AH) $C^*$-algebra $A$ with unique tracial state, which doesn't satisfy (iii)(a) in Theorem~\ref{thm:SimpleUDPEquiv} (i.e., $A_\omega$ doesn't have a unique tracial state).
Since $A$ has the Dixmier property by \cite{HZ}, this shows that the Dixmier property is strictly weaker than the uniform Dixmier property.
\smallskip

Let us briefly discuss the cases when $A$ is unital, has the Dixmier property,  trivial centre, and no tracial states (i.e., cases (1) and (3) from the beginning of this section).
If $A$ is simple  and purely infinite, then $A_\omega$ is also simple and purely infinite (\cite[Proposition 6.2.6]{Rordam:Book}), whence has the Dixmier property, and so  $A$ has the uniform Dixmier property. In the cases that  $A$ is not simple and purely infinite, we have little to say about whether $A$ has the uniform Dixmier property.
In such cases, $A_\omega$ has no tracial states either, but it is not simple, for if $A_\omega$ is simple and non-elementary then $A$ must be simple and purely infinite (\cite[Remark 2.4]{Kirch:Abel}).
 R\o rdam has constructed examples of simple unital separable (even nuclear) $C^*$-algebras which are not purely infinite, yet have no tracial states (\cite{Rordam:Acta}).

\begin{question}
Are there simple unital $C^*$-algebras with the uniform Dixmier property and without tracial states other than the purely infinite ones?
\end{question}

Let $A$ be a simple unital $C^*$-algebra with no tracial states, which is not purely infinite.
Then there is a bounded sequence of self-adjoint elements $(a_n)_{n=1}^\infty$ with $1 \in D_A(a_n)$ for all $n$, but $1 \not\in D_{A_\omega}((a_n)_n)$.
(However, it is conceivable that $D_{A_\omega}((a_n)_n)$ meets $Z(A_\omega)$ in another point, so this does not show that $A$ does not have the uniform Dixmier property.)
To see this, first, since $A_\omega$ is non-simple, there exists a positive element $a \in A_\omega$ of norm $2$ that is not full.
Lift $a$ to a bounded sequence $(a_n)_n$ of positive elements.
Since $\|a\|=2$ and $a$ is not invertible, for $\omega$-almost all $n$, the convex hull of the spectrum of $a_n$ contains $1$.
Modifying $a_n$ for $n$ in an $\omega$-null set, we can arrange that the convex hull of the spectrum of $a_n$ contains $1$ for all $n$.
Since $A$ is simple, it follows that $1 \in D_{A}(a_n)$ for all $n$.
However, if $1 \in D_{A_\omega}(A)$ then with $I:=\mathrm{Ideal}(a)$, $1 \in D_{A_\omega/I}(q_I(a)) = D_{A_\omega/I}(0)$, which is a contradiction.

\subsection{Explicit constants}	
\label{sec:ExplicitConstants}
Suppose that  $A$ is a unital $C^*$-algebra that has the Dixmier property as well as one of the following properties:
\begin{enumerate}[(1)]
\item
finite nuclear dimension, or
\item
finite radius of comparison by traces.
\end{enumerate}
Then $A$ has the uniform Dixmier property for suitable constants $(m,\gamma)$ (i.e.,  \eqref{mgamma} holds). For finite nuclear dimension, this follows from \cite[Theorem 5.6]{NRS}. For finite radius of comparison, this is Corollary \ref{cor exact UDP} obtained  above.
These results are proven by contradiction, with repeated use of the Hahn--Banach Theorem, thereby not yielding explicit values for the constants $(m,\gamma)$.
In fact, we do not know   explicit values for $(m,\gamma)$ holding globally in either one of these two cases. (On the other hand, explicit constants may be extracted from the methods used in \cite{Skoufranis} and \cite{NS}, for  simple $C^*$-algebras with real rank zero, strict comparison by traces, and a unique tracial state.) 
In this section, prompted by an interesting question posed by the referee, we find explicit values for the constants $(m,\gamma)$ for a variety of $C^*$-algebras with the  uniform Dixmier property. When the $C^*$-algebras have the singleton Dixmier property, we also estimate the constants $(M,\Upsilon)$ (i.e, for which \eqref{MUpsilon} holds).

Let $A$ be a $C^*$-algebra. Let $h\in A$ be a self-adjoint element and let $[l(h),r(h)]$ be the smallest interval containing  the spectrum of $h$, i.e., the  numerical range of $h$. Set $\omega(h):=r(h)-l(h)$ and note that $\omega(h)\leq 2\|h\|$.

We first consider uniform Dixmier property constants for von Neumann algebras (slightly improving the constants that can be extracted from Dixmier's original argument \cite[Lemma 1 of \S III.5.1]{DixvN}).  Let $W$ be a von Neumann algebra and $h$ a self-adjoint element of $W$. Let $e\in W$ be a central projection. In the next lemma   $\omega_e(h)$ denotes $\omega(eh)$ in the von Neumann algebra $eW$.
 \begin{lemma}\label{spectrumhalving}
 Let $W$ be a von Neumann algebra.
 	Let $h\in W$ be a self-adjoint element with finite spectrum.
 	Then there exist central projections $e_1,\ldots,e_n$ adding up to $1$ and a unitary $u\in W$ such that
 	\[
 	\omega_{e_k}\Big(\frac{h+uhu^*}{2}\Big)\leq \frac 1 2\omega_{e_k}(h)\hbox{ for all }k.
 	\]
 \end{lemma}
 \begin{proof}
 	It is shown in \cite[Proposition 3.2]{NRS} that given two self-adjoint elements $h_1,h_2\in W$ with finite spectrum,
 	it is possible to find projections $P_1,\ldots,P_N$ adding up to 1, a unitary $u\in W$,
 	and self-adjoint central elements $\lambda_1,\mu_1,\ldots,\lambda_N,\mu_N\in Z(W)$ with finite spectrum
 	such that
 	\begin{align*}
 	h_1=\sum_{i=1}^N \lambda_iP_i,\quad
 	uh_2u^*=\sum_{i=1}^N \mu_iP_i,
 	\end{align*}
 	and
 	\[
 	\lambda_1\geq \ldots\geq \lambda_N,\quad \mu_1\geq \ldots\geq \mu_N.
 	\]
 	(Note: \cite[Proposition 3.2]{NRS} is stated for positive elements but it is easily extended to self-adjoint elements by adding a scalar.)
 	Let us apply this result to the self-adjoint elements $h$ and $-h$. We then get
 	\[
 	h=\sum_{i=1}^N \lambda_iP_i\, \hbox{ and }\,
 	uhu^*=\sum_{i=1}^N \nu_iP_i,
 	\]
 	where $\lambda_1\geq \ldots \geq \lambda_N$ and $\nu_1\leq \ldots\leq \nu_N$.
 	Since all of the $\lambda_i$ and $\nu_i$ have finite spectrum, there exist central projections $e_1,\ldots,e_n$ with sum 1 such that $e_k\lambda_i$ and $e_k\nu_i$
 	are scalar multiples of $e_k$ for all $i$ and $k$. Let us show that $e_1,\ldots,e_n$ and $u$ are as desired. Let $\tilde h:=(h+uhu^*)/2$.
 	Fix $1\leq k\leq n$.  Let $S:=\{i\in \{1,\ldots,N\} : e_kP_i\neq 0\} $.
 	Denote the scalars  $e_k\lambda_i$ and $e_k\nu_i$ (in $e_kW$)
 	simply as $\lambda_i$ and $\nu_i$. Then the spectrum of $e_kh$ in $e_kW$ is
 	$\{\lambda_i : i\in S\}$ and also (since $e_kh$ is unitarily equivalent to $e_kuhu^*$) $\{\nu_i:i\in S\}$.
 	On the other hand, the spectrum
 	of  $e_k\tilde h$ in $e_kW$ is the set
 	\[
 	\Big\{\frac{\lambda_i+\nu_i}{2} :  i\in S\Big\}.
 	\]
 	Let $i,j\in S$ with $i\leq j$. Then
 	\begin{align*}
 	\Big|\frac{\lambda_{i}+\nu_i}{2}-\frac{\lambda_{j}+\nu_{j}}{2}\Big| &=
 	\Big|\frac{\lambda_i-\lambda_{j}}{2}-\frac{\nu_{j}-\nu_i}{2}\Big|\\
 	&\leq \max\Big(\frac{\lambda_i-\lambda_{j}}{2},\frac{\nu_j-\nu_i}{2}\Big)\leq \frac{\omega_{e_k}(h)}{2}.
 	\end{align*}
 	Thus, $\omega_{e_k}(\tilde h)\leq \omega_{e_k}(h)/2$ for all $k$, as desired.
 \end{proof}

 \begin{thm}\label{Wconstants}
 	Let $W$ be a von Neumann algebra. Then $W$ has the uniform Dixmier property with constants $(m,\gamma)$ for $m=2$ and
 	every $\gamma\in (1/2, 1)$.
  If $W$ is finite, then it has the uniform singleton Dixmier property with constants $(M,\Upsilon)$ for
 	$M=4$ and every $\Upsilon\in (1/2, 1)$.
 \end{thm}	
 \begin{proof} Let $0<\epsilon <1/2$ and let $0\neq g=g^*\in W$. By the spectral theorem, there is a self-adjoint element $h\in W$ with finite spectrum such that $\|g-h\|<\epsilon\|g\|$ and $\|h\|\leq \|g\|$.
  Apply Lemma~\ref{spectrumhalving} to $h$ to obtain a unitary $u\in W$ and central projections $e_1,\dots,e_n$ as in the statement of that lemma, and then define the central element $z:=\sum_{k=1}^n \alpha_ke_k$, where $\alpha_k$ is the midpoint of the spectrum of $e_k(h+uhu^*)/2$ in $e_kW$ (that is, the midpoint of the interval $[l(e_k(h+uhu^*)/2),r(e_k(h+uhu^*)/2)]$).
Then we see that 
\begin{align*}
\Big\|e_k\Big(\frac{h+uhu^*}2-z\Big)\Big\| 
&= \Big\|e_k\frac{h+uhu^*}2-\alpha_ke_k\Big\| \\
&= \frac12\omega_{e_k}\Big(\frac{h+uhu^*}2\Big) \\
&\leq \frac14\omega_{e_k}(h) \\
&\leq \frac12\|e_kh\| \leq \frac12\|h\|.
\end{align*}
Since the $e_k$ are orthogonal central projections, it follows that $\|(h+uhu^*)/2-z\|\leq \frac12\|h\|$.
Then
  \[
  \|(g+ugu^*)/2 -z\| \leq \|h\|/2 +\epsilon\|g\| \leq (1/2+\epsilon)\|g\|.
  \]
 	
 	Suppose now that $W$ is finite and hence has the singleton Dixmier property. For all $\epsilon>0$ such that $(1/2+\epsilon)^2<1/2$, $W$ has the uniform Dixmier property with constants $(2^2, (1/2+\epsilon)^2)$ and hence the uniform singleton Dixmier property with constants $(4,2(1/2+\epsilon)^2)$ (by Lemma~\ref{lemma UDP to USDP}). Since $2(1/2+\epsilon)^2\to 1/2$ as $\epsilon \to 0$, we obtain the required result.
 \end{proof}

\begin{prop}
\label{prop:M_nExplicitConstants}
The $C^*$-algebra  $M_n$ has the uniform Dixmier property with constants $m=2$ and
 	$\gamma=1/2$ and the
uniform singleton Dixmier property with constants $M=4$ and $\Upsilon=1/2$.
\end{prop}	
\begin{proof}
That $M_n$ has the uniform Dixmier property with constants $m=2$ and $\gamma=1/2$ follows at once from Lemma \ref{spectrumhalving} above. The constants  $M=4$ and $\Upsilon=1/2$ are obtained from Lemma~\ref{lemma UDP to USDP}.
\end{proof}
	
\begin{thm}
Let $X$ be a compact Hausdorff space with covering dimension $d<\infty$. Let $n\in \NN$. The following are true:
\begin{enumerate}
\item
The $C^*$-algebra   $C(X,M_n)$
has the uniform Dixmier property with constants $(m,\gamma)$ for $m=d+2$ and every $\gamma\in((d+1)/(d+2),1)$.
It has the uniform singleton Dixmier property  with constants $(M,\Upsilon)$ for $M=3d+4$ and every $\Upsilon\in((3d+2)/(3d+4),1)$.

\item
If $d\leq 2$ and in the \v{C}ech cohomology we have $\check H^2(X)=0$ (e.g., $X=[0,1]$ or $X=[0,1]^2$), then
$C(X,M_n)$ has the uniform Dixmier property with constants $(m,\gamma)$ for $m=2$ and
 	every $\gamma\in(1/2,1)$ and the
uniform singleton Dixmier property with constants $(M,\Upsilon)$ for $M=4$ and every $\Upsilon\in(1/2,1)$.
\end{enumerate}
\end{thm}
\begin{proof} It is well-known that $C(X,M_n)$ has the singleton Dixmier property: for example, the Dixmier property holds by \cite[Proposition 2.10]{Arch-PLMS} and the singleton Dixmier property is then a consequence of the fact that every simple quotient has a tracial state (see Proposition~\ref{SDPandCVT}).
We prove (ii) first, because the argument is more similar to the previous proof.

\noindent (ii):
Let $h\in C(X,M_n)$ be a self-adjoint element. By \cite[Theorem 10]{Thomsen}, $h$ is approximately unitarily equivalent to a diagonal self-adjoint $h'=\mathrm{diag}(\lambda_1,\ldots,\lambda_n)$, where the eigenvalue functions
$\lambda_1,\ldots,\lambda_n\in C(X,\R)$ are arranged in decreasing order: $\lambda_1\geq \lambda_2\geq\cdots\geq\lambda_n$.
Note that a self-adjoint element $a$ in a unital $C^*$-algebra satisfies \eqref{mgamma} if and only if every unitary conjugate of $a$ does so (with the same central element $z$). Hence, by an approximation argument similar to that in the proof of Theorem~\ref{Wconstants}, it suffices to establish \eqref{mgamma} with $m=2$ and $\gamma=1/2$ for diagonal self-adjoint elements of the form above. So assume that $h$
is diagonal with decreasing eigenvalue functions.
Let $u\in M_n$ be the permutation unitary such that $uhu^*=\mathrm{diag}(\lambda_n,\ldots,\lambda_1)$.
Set $\tilde h:=(h+vhv^*)/2$, where $v\in \mathcal U(C(X,M_n))$ is given by $v(x):=u$ ($x\in X$). Then,
\[
\tilde h=\mathrm{diag}\Big(\frac{\lambda_1+\lambda_n}{2},\frac{\lambda_2+\lambda_{n-1}}{2},\ldots,\frac{\lambda_n+\lambda_1}{2}\Big).
\]
The same estimates used in the proof of Lemma \ref{spectrumhalving} show  that
$\omega(\tilde h)\leq \omega(h)/2$. It follows that
\[
\|\tilde h-\frac{\lambda_1+\lambda_n}{2}\cdot 1_n\|\leq \frac 1 2 \|h\|.
\]
As observed above, this shows that $C(X,M_n)$ has the uniform Dixmier property with $m=2$ and every $\gamma\in (1/2,1)$.
The constants $M=4$ and $\Upsilon\in (1/2,1)$ are then derived from the constants $(m,\gamma)$ as in the proof of Theorem \ref{Wconstants}.
\smallskip

(i):
Let $\epsilon>0$ be given.
Let $f \in C(X,M_n)$ be a self-adjoint contraction.
For each $x \in X$, by Proposition \ref{prop:M_nExplicitConstants}, we may find $\lambda_x \in \R$ and a unitary $u_x \in M_n$ such that
\[
\left\|\frac12 \left(f(x) + u_xf(x)u_x^*\right) - \lambda_x1\right\| \leq \frac12.
\]
Evidently, we may assume $\lambda_x \in [-1,1]$.
By continuity, we may then find a neighbourhood $W_x$ of $x$ such that
\[
\left\|\frac12 \left(f(y) + u_xf(y)u_x^*\right) - \lambda_x1\right\| <\frac12+\epsilon \quad \text{for all $y \in W_x$}. \]
From the open cover $\{W_x:x\in X\}$ of $X$, using compactness and the fact that $X$ has dimension $d$, we may find a finite refinement of the form $\{W^{(i)}_j\}_{i=0,\dots,d;\, j=1,\dots,r}$ which covers $X$, and such that
\[ \overline{W^{(i)}_j} \cap \overline{W^{(i)}_{j'}} = \emptyset \]
for all $j \neq j'$.
Denote $u^{(i)}_j$ the unitary corresponding to the open set $W^{(i)}_j$ and $\lambda^{(i)}_j$ the scalar, i.e., such that
\[ \left\|\frac12 \left(f(y) + u^{(i)}_jf(y)(u^{(i)}_j)^*\right) - \lambda^{(i)}_j1\right\| <\frac12+\epsilon \]
for all $y \in W^{(i)}_j$.
For $i\in\{0,\dots,d\}$, since all unitaries in $M_n$ are homotopic to the identity, we may produce a unitary $u^{(i)} \in C(X,M_n)$ such that
\[ u^{(i)}(y) = u^{(i)}_j \quad \text{whenever $y \in W^{(i)}_j$}, \]
as follows.
We may find disjoint open sets $V^{(i)}_1,\dots,V^{(i)}_r$ containing $\overline{W^{(i)}_1},\dots,\overline{W^{(i)}_r}$ respectively, and then we may use a homotopy of unitaries to get a unitary in $C(\overline{V^{(i)}_j},M_n)$ which is identically $u^{(i)}_j$ on $W^{(i)}_j$ and identically $1$ on $\partial V^{(i)}_j$.
We may then define the continuous unitary $u^{(i)} \in C(X,M_n)$ so that it restricts to the unitary just defined on each $\overline{V^{(i)}_j}$ and is identically $1$ outside of $V^{(i)}_1 \cup \cdots \cup V^{(i)}_r$.

We claim that $\tilde{f}:=\frac1{d+2}(f+u^{(0)}f(u^{(0)})^*+\cdots+u^{(d)}f(u^{(d)})^*)$ has distance at most $(d+1)/(d+2)+\epsilon$ to the centre.
Note that, by a partition-of-unity argument, the distance from $\tilde{f}$ to the centre is equal to the supremum over all $x \in X$ of the distance from $\tilde{f}(x)$ to $Z(M_n)=\C1_n$ (see \cite[Theorem 2.3]{Som} for a more general result).
For $x \in X$, pick $i_0,j$ such that $x \in W^{(i_0)}_j$.
Without loss of generality, $i_0=0$.
Then
\begin{align*}
\left\|\tilde{f}(x)-\frac2{d+2}\lambda^{(0)}_j1\right\|
&\leq \frac2{d+2}\left\|\frac12(f(x)+u^{(0)}(x)f(x)(u^{(0)}(x))^*)-\lambda^{(0)}_j1\right\| \\
&\qquad +\frac1{d+2}\sum_{i=1}^d \|u^{(i)}(x)f(x)(u^{(i)}(x))^*\| \\
&= \frac2{d+2}\left\|\frac12(f(x)+u^{(0)}_jf(x)(u^{(0)}_j)^*)-\lambda^{(0)}_j1\right\| \\
&\qquad +\frac1{d+2}\sum_{i=1}^d \|u^{(i)}(x)f(x)(u^{(i)}(x))^*\| \\
&< \frac 2{d+2}\left(\frac12+\epsilon\right) + \frac d{d+2} \\
&\leq  \frac{d+1}{d+2} +\epsilon
\end{align*}
as required.

A similar argument is used to get uniform singleton Dixmier property constants.
Here we may replace $f$ with $f-R(f)$, so that $f(x)$ has trace $0$ for all $x \in X$.
Then we use the same argument as above, with the uniform singleton Dixmier property constants $(4,\frac12)$ from Proposition \ref{prop:M_nExplicitConstants}, and with $\lambda_x=0$ for all $x$ (and thereby $\lambda^{(i)}_j=0$ for all $i,j$), to get $M=3(d+1)+1 = 3d+4$ and (for any sufficiently small $\epsilon>0$)
\[
\Upsilon=\frac12\frac4{3d+4} + \frac{3d}{3d+4} + \epsilon = \frac{3d+2}{3d+4} + \epsilon.\qedhere
\]
\end{proof}

Consider the following property of unital $C^*$-algebras $A$:

\begin{description}[leftmargin=1cm,font=\normalfont]
\descitem{(P)}	
There exist $M\in \NN$ and $0<\Upsilon<1$ such that if $h\in A$ is a self-adjoint such that $\tau(h)=0$ for all $\tau\in T(A)$ then	
\begin{equation}\label{UZM}
\Big\|\frac  1  M\sum_{i=1}^Mu_ihu_i^*\Big\|\leq \Upsilon\|h\|
\end{equation}
for some unitaries $u_1,\ldots,u_M$.
\end{description}

Note that if $A$ has the property \descref{(P)} for some $(M,\Upsilon)$ then it also has \descref{(P)} for $(M^k,\Upsilon^k)$ ($k =2,3,\ldots)$.

Suppose that $A$ has the Dixmier property and has the property \descref{(P)} for some $(M,\Upsilon)$.
 For $h=h^*\in A$ and $z_1,z_2\in D_A(h)\cap Z(A)$, we have $\tau(z_1-z_2)=0$ for all $\tau\in T(A)$ and hence
 $0\in D_A(z_1-z_2) = \{z_1-z_2\}$. Thus $z_1=z_2$. An elementary argument with real and imaginary parts shows that $A$ has the singleton Dixmier property. It then follows from \descref{(P)} that $A$ has the uniform singleton Dixmier property with the same constants $(M,\Upsilon)$ (as  introduced in \eqref{MUpsilon}).

 Conversely, suppose that $A$ has the uniform singleton Dixmier property with constants $(M,\Upsilon)$. If $h=h^*\in A$ vanishes on all tracial states of $A$ then $h$ also vanishes on the centre-valued trace of $A$. Thus $A$ has the property \descref{(P)} with the same $(M,\Upsilon)$.

But \descref{(P)} may hold much more generally: if every quotient of $A$ has a bounded trace and $A$ has either
finite nuclear dimension or finite radius of comparison by traces then $A$ has \descref{(P)} for some
$(M,\Upsilon)$ (\cite[Theorem 5.6]{NRS} for the former case, Theorem \ref{rcuzm} in the latter).

In the following results, it will occasionally be convenient to write $a\approx_{\epsilon}b$ to mean $\|a-b\|<\epsilon$.

\begin{thm} \label{theorem ranks one}
Let $A$ be a unital $C^*$-algebra with decomposition rank one and stable rank one.
Then $A$ has \emph{\descref{(P)}} with constants $(M,\Upsilon)$ for $M=15$ and every $\Upsilon\in(11/15,1)$.
In particular, if $A$ also has the Dixmier property, then it has the uniform singleton Dixmier property with these constants.
\end{thm}
\begin{proof}
Let $\epsilon>0$ be given.
Let us factorise the diagonal embedding $\iota\colon A \to A_\infty$ as $\sum_{i=0}^1 \phi_i\circ\psi_i$,
where $\psi_0,\psi_1$
are unital homomorphisms, $\phi_0,\phi_1$ are c.p.c.\ order zero,
and where $N_0,N_1$ have the form $\prod_\lambda F_\lambda/\bigoplus_\lambda F_\lambda$, for finite dimensional algebras $F_\lambda$.
(The existence of such a factorisation follows from the proof of \cite[Proposition 2.2]{Robert:Comparison}, using \cite[Proposition 5.1]{KirchbergWinter} in place of \cite[Proposition 3.2]{WinterZacharias}.)

Suppose that, for every $n\in \NN$, $M_n$ has the uniform Dixmier property with constants $(m,\gamma)$ (we shall describe suitable values $m>1$ and $\gamma$ towards the end of the proof). Then any product of finite-dimensional $C^*$-algebras has the uniform Dixmier property with constants $(m,\gamma)$ (Theorem~\ref{quotprod}(ii)), hence the Dixmier property and hence the singleton Dixmier property by Proposition~\ref{SDPandCVT} (note that the product of the centre-valued traces is a centre-valued trace on the product). Hence $N_0$ and $N_1$ have the uniform Dixmier property with constants $(m,\gamma)$ (Theorem~\ref{quotprod}) and the singleton Dixmier property (see the remark after Proposition~\ref{SDPandCVT}). By Lemma~\ref{lemma UDP to USDP}, $N_0$ and $N_1$ have the uniform singleton Dixmier property with constants $(m,2\gamma)$ (and for this we require $2\gamma<1$). Hence, as noted above, $N_0$ and $N_1$ have property \descref{(P)} with constants $M'=m>1$ and $\Upsilon'=2\gamma$.

Let $h\in A$ be a self-adjoint contraction that is zero on every tracial state.
Then the same is true
for $\psi_i(h)$ for   $i=0,1$.  So for $i=0,1$ there exist unitaries
$u_{i,1},\ldots,u_{i,M'-1}\in N_i$ such that
\begin{equation}\label{NmUpsilon}
\Big\|\frac{\psi_i(h)+\sum_{j=1}^{M'-1} u_{i,j}\psi_i(h)u_{i,j}^*}{M'}\Big\|\leq \Upsilon'.
\end{equation}
Set $h_i:=\phi_i(\psi_i(h))$ for $i=0,1$, so that $h=h_0+h_1$.
We set $x_{i,j}=\phi_i^{\frac 1 n}(u_{i,j})\in A_\infty$, so that $x_{i,j}$ depends on $n$, and
\[
|x_{i,j}|h_i \to h_i\quad\text{and}\quad x_{i,j}h_ix_{i,j}^* \to \phi_i(u_{i,j}\psi_i(h)u_{i,j}^*) \quad \text{as $n \to \infty$.}
\]
Since $A$ has stable rank one, every element in $A_\infty$ has a unitary polar decomposition, with the unitary element belonging to $A_\infty$. (Indeed, since $A$ has stable rank one every element of $A_\infty$ lifts to  an element $(a_k)\in \prod_{k=1}^\infty A$ such that $a_k$ is invertible for all $k$; further, such an $(a_k)$ has  polar decomposition in $\prod_{k=1}^\infty A$.)
So  $x_{i,j}=U_{i,j}|x_{i,j}|$ for some unitary $U_{i,j}\in A_\infty$.  Then, remembering that $x_{i,j}$ depends on $n$,
\[
U_{i,j}h_iU_{i,j}^*\to \phi_i(u_{i,j}\psi_i(h)u_{i,j}^*) \quad \text{and}\quad U_{i,j}\phi_i(1)U_{i,j}^* \to \phi_i(1) \quad \text{as $n\to\infty$}.
\]
From \eqref{NmUpsilon}, for $n$ sufficiently large we get
\[
\Big\|\frac{h_i+\sum_{j=1}^{M'-1} U_{i,j}h_iU_{i,j}^*}{M'}\Big\|\leq \Upsilon'+\epsilon,
\]
for $i=0,1$.
We choose $n$ so that in addition,
\begin{equation}
\label{eq:sr1Commuting}
 \|U_{i,j}\phi_i(1)U_{i,j}^* - \phi_i(1)\|<\epsilon.
\end{equation}
Consider
\[
\tilde h:=\frac{1}{2M'-1}\Big(h+\sum_{j=1}^{M'-1} U_{0,j}hU_{0,j}^*+\sum_{j=1}^{M'-1} U_{1,j}hU_{1,j}^*\Big);
\]
we will estimate its norm.
We manipulate the sum on the right side:

\begin{multline}\label{long}
h+\sum_{i=0,1}\sum_{j=1}^{M'-1} U_{i,j}hU_{i,j}^*
=(h_0+\sum_{j=1}^{M'-1} U_{0,j}h_0U_{0,j}^*)
+(h_1+\sum_{j=1}^{M'-1} U_{1,j}h_1U_{1,j}^*)\\
+\sum_{j=1}^{M'-1} (U_{1,j}h_0U_{1,j}^*+ U_{0,j}h_1U_{0,j}^*).
\end{multline}

Let $e_0:=\phi_0(1)$ and $e_1:=\phi_1(1)$. Then $h_0\leq e_0$, $h_1\leq e_1$, and $e_0+e_1=1$.
Next from \eqref{eq:sr1Commuting} and the fact that $e_0+e_1=1$,  it follows that $\|[U_{i,j},e_{1-i}]\|<\epsilon$ for $i=0,1$ and $j=1,\dots,M'-1$.
Hence,
\[
U_{1,j}h_0U_{1,j}^*+ U_{0,j}h_1U_{0,j}^* \leq U_{1,j}e_0U_{1,j}^*+ U_{0,j}e_1U_{0,j}^*\approx_{2\epsilon}\,e_0+e_1=1,
\]
which implies that $U_{1,j}h_0U_{1,j}^*+ U_{0,j}h_1U_{0,j}^* \leq (1+2\epsilon)1$.
A similar argument shows that $U_{1,j}h_0U_{1,j}^*+ U_{0,j}h_1U_{0,j}^* \geq -(1+2\epsilon)1$.
The norm of the right side of \eqref{long} is at most
$
2M'(\Upsilon'+\epsilon)+(M'-1)(1+2\epsilon)
$ from which we obtain that
\[
\|\tilde h\|\leq \frac{2M'(\Upsilon'+\epsilon)+(M'-1)(1+2\epsilon)}{2M'-1}.
\]
Thus $A$ has property (P) with constants $M=2M'-1$ and $\Upsilon=(2M'\Upsilon'+M'-1)/(2M'-1)+2\epsilon$ (provided that this value of $\Upsilon$ is less than $1$).

It follows from Proposition~\ref{prop:M_nExplicitConstants} that, for every $n$,
the $C^*$-algebra $M_n$ has the uniform Dixmier property with constants $m=2^3=8$ and
 	$\gamma=(1/2)^3=1/8$. By the discussion above, we may take $M'=8$ and $\Upsilon'=1/4$ and hence obtain
 that  $A$ has  \descref{(P)} with constants $(M,\Upsilon)$ for $M=15$  and every $\Upsilon\in(11/15,1)$.
\end{proof}

Our next goal is to prove Theorem \ref{explicitsimple} below. But first we need a lemma and some preliminaries.
	
Given a self-adjoint element $h$,
let us say that the spectrum of  $h$ has \emph{gaps of size at most $\delta$}
if every closed subinterval of $[l(h),r(h)]$ of length $\delta$ intersects the spectrum of $h$.

\begin{lemma}\label{gaps}
Let  $A$ be simple, unital, and non-elementary.
Let $h\in A$ be a self-adjoint element and $\epsilon>0$.  The following are true:
\begin{enumerate}
\item
There exists a unitary $u\in A$ such that the spectrum of
\[
\frac 1 3 h +\frac 2 3 uhu^*
\]
has gaps of size at most $\omega(h)/3+\epsilon$.

\item
If the spectrum of $h$ has gaps of size at most $\delta>0$ then there exist a self-adjoint $\tilde h\in A$ and
$x\in A$ such that  $\|x\|^2=\delta/2$, $x^2=0$,
\[
\|h-(\tilde h+[x^*,x])\|<\epsilon,\]
and the spectrum of $\tilde h$  is the interval $[l(h)-\delta/2,r(h)+\delta/2]$.
\end{enumerate}
\end{lemma}
\begin{proof}
(i): The result is trivial  if $l(h)=r(h)$. So assume that $l(h)<r(h)$.
If the result has been proven for a given $h$ (and arbitrary $\epsilon>0$) then it at once follows for any
$\alpha h+\beta 1$, with $\alpha,\beta\in \R$. Thus, we may assume that $0\leq h\leq 1$
and that $[l(h),r(h)]=[0,1]$.
Let us perturb $h$ slightly
using functional calculus with a continuous function that is close to the identity function but takes the constant value $0$ in a neighbourhood of $0$ and the constant value $1$ in a neighbourhood of $1$,  so that for the new element $k$ (still a positive contraction) we have that $\|h-k\|<\epsilon/2$ and that $ke=e$ and $kf=0$ for some non-zero $e,f\in A_+$.
Since $A$ is simple (whence prime),  there exists a positive element $a\in A$ such that  $eaf\neq 0$.  Let $x=eaf$. Then  $x^*x\in \overline{fAf}$  and $xx^*\in \overline{eAe}$. Since $x^2=0$, $x$ is in the closure of the invertible elements of $A$.
By \cite[Theorem 5]{Pedersen},  for each  $t>0$ there exists a unitary $u\in A$ such that $u(x^*x-t)_+u^*=(xx^*-t)_+$. 
Choose one such $u$ for some $t<\|x\|$. Set $\tilde e:=(x^*x-t)_+$ and
$\tilde f:=(xx^*-t)_+$.  Now consider
\[
\tilde k:=\frac 1 3 k+ \frac 2 3(uku^*).
\]
Then $\tilde k \tilde e=(1/3)\tilde e$ and $\tilde k\tilde f=(2/3)\tilde f$. Since $\tilde e$ and $\tilde f$ are nonzero, $1/3$  and  $2/3$  are in the spectrum of $\tilde k$.

Let $\tilde h:=\frac 1 3 h+ \frac 2 3(uhu^*)$, a positive contraction in $A$ such that $\|\tilde h-\tilde k\|<\epsilon/2$.
Suppose, towards a contradiction, that the spectrum of $\tilde h$ does not intersect $(1/3-\epsilon/2,1/3+\epsilon/2)$.
Define $b:=\tilde h -(1/3)1$ and $c:=\tilde k-(1/3)1$, so that $b$ is self-adjoint, the spectrum of $b$ does not intersect
$(-\epsilon/2,\epsilon/2)$ and $\|b-c\|<\epsilon/2$. We have
$$\|1-b^{-1}c\|\leq \|b^{-1}\|\|b-c\| < \frac{\epsilon}{2}\|b^{-1}\|\leq 1.$$
Hence $b^{-1}c$ is invertible and so $c$ is invertible, which contradicts that $1/3$ is in the spectrum of $\tilde{k}$. A similar argument shows that
the spectrum of $\tilde h$ intersects $(2/3-\epsilon/2,2/3+\epsilon/2)$. It follows that the spectrum of $\tilde h$ has gaps of size at most $1/3 +\epsilon$.

(ii):
Choose  points  $l(h)=t_0<t_1<\ldots<t_n=r(h)$ in the spectrum of $h$  such that $t_{i+1}-t_i\leq \delta$ for all $i$.
Perturb $h$ by functional calculus using an increasing continuous function close to the identity function that takes the constant value $t_i$ in a small neighborhood
of each $t_i$, so that the new $h'$ satisfies $\|h'-h\|<\epsilon$, has spectrum contained in $[l(h),r(h)]$ and has the property that there exist pairwise orthogonal non-zero positive elements $e_0,e_1,\ldots,e_n$ such that $h'e_i=t_ie_i$. For each $i=0,1,\ldots,n$, choose  $x_i\in \overline{e_iAe_i}$ such that $x_i^2=0$ and $x_i^*x_i$ (and hence $x_ix_i^*$) has spectrum equal to $[0,1]$. This is possible since $\overline{e_iAe_i}$ is simple and non-elementary. Now let
\[
x:=\sum_{i=0}^n (\delta/2)^{\frac 1 2}x_i\hbox{\, and \,}\tilde h:=h'-[x^*,x].
\]
We claim that $\tilde h$ and $x$ are as desired.
It follows from the pairwise orthogonality of the $e_i$ that $x^2=0$, that
\[
\|x\|^2 = \frac{\delta}{2}\Big\|\sum_{i=0}^n x_i^*x_i\Big\| = \frac{\delta}{2}
\]
and that
\[
\tilde h=h'-\sum_{i=0}^n (\delta/2)(x_i^*x_i-x_ix_i^*).
\]
Let us show that $\tilde h$ has spectrum $[l(h)-\delta/2,r(h)+\delta/2]$.
Note that all of the elements $h'$, $x_i^*x_i$ and $x_ix_i^*$ ($0\leq i\leq n$) lie in a commutative $C^*$-subalgebra $C$ containing the unit $1$ of $A$.
Evaluating the right-hand side of the expression for $\tilde h$ on the points of the spectrum of $C$ where $x_i^*x_i$ is supported, all  other terms except  for $h'$ vanish,
while $h'$ takes the constant value $t_i$. Since the spectrum of $x_i^*x_i$ is $[0,1]$, we obtain the interval $[t_i-\delta/2,t_i]$  in the spectrum
of $\tilde h$. Evaluating on the points where $x_ix_i^*$ is supported, we obtain the interval
$[t_i,t_i+\delta/2]$ in the spectrum of $\tilde h$.
Doing this for all $i$, we obtain the interval $[l(h)-\delta/2,r(h)+\delta/2]$ in the spectrum of $\tilde h$. Evaluating on any other point in the spectrum of $C$, we obtain a value in the spectrum of $h'$ which is contained in $[l(h),r(h)]$.
Thus the spectrum of $\tilde h$ is as required.
\end{proof}

Let  $A$ be simple, unital, non-elementary,  with stable rank one and with strict comparison by traces.
Using Cuntz semigroup classification results, one can prove the existence of  a nuclear $C^*$-subalgebra $B\subseteq A$ with rather special properties.  \cite[Theorem 4.1]{NgRobert2} spells out the properties of $B$ that we need:
\begin{enumerate}[(i)]
\item
$B\cong C\otimes W$, where $C$ is a simple AF $C^*$-algebra and $W$ is the Jacelon-Razak algebra.
\item
Every tracial state $\tau$ on $B$ extends uniquely to a tracial state on $A$.
\item
Every non-invertible self-adjoint element $h$ in $A$ with connected spectrum is approximately unitarily
equivalent to a self-adjoint element in $B$. (Note: In the statement of \cite[Theorem 4.1]{NgRobert2}
the hypothesis that the self-adjoint $h$ must be non-invertible is missing, though this is clearly necessary since $B$ is non-unital. Moreover, this hypothesis is tacitly used in the last paragraph of the proof of \cite[Theorem 4.1]{NgRobert2}.)
\end{enumerate}

A technique in \cite{JST} and  \cite{NgRobert2} involves using these properties to reduce the proof of certain properties of self-adjoint elements in $A$ to the case of self-adjoint elements in $B$.
We we will use the same technique to prove the following theorem. In this theorem, the initial hypotheses on $A$ ensure that $T(A)$ is non-empty
(surely, if we had $T(A)=\varnothing$ then the strict
comparison-by-traces property and the simplicity of $A$ would imply that $A$ is purely infinite, contradicting that $A$ has stable rank one.)
 Thus the later additional assumption that $A$ has the Dixmier property is equivalent to assuming that $A$ has a unique tracial state \cite{HZ}.

\begin{thm}\label{explicitsimple}
	Let $A$ be a simple, unital, non-elementary $C^*$-algebra with stable rank one and strict comparison by traces.
	Then $A$ has  \emph{\descref{(P)}} with constants  $M=3\cdot 7^3$ and $\Upsilon=0.86$.
	Suppose, in addition, that $A$ has the Dixmier property. Then $A$ has the uniform singleton Dixmier property with these constants.
\end{thm}

 \begin{proof}
 	Suppose that the unitisation $B+\C1$ of $B$ has \descref{(P)} with constants $(M',\Upsilon')$.
We show how to obtain constants for $A$.
Suppose that $h\in A$ is a self-adjoint element that is zero on every trace and such that $\|h\|=1$. Let $\epsilon>0$.
From Lemma \ref{gaps} (i), we obtain $h_1=(1/3)h+(2/3)uhu^*$ whose spectrum has gaps at most $2/3+\epsilon$.
Applying Lemma \ref{gaps} (ii) to $h_1$ with $\delta =2/3+\epsilon$, we obtain $\tilde h_1, x\in A$, as described, such that
\begin{equation}\label{decomp}
h_1\approx_\epsilon \tilde h_1+[x^*,x].
\end{equation}
Notice, for later use, that since the positive elements $x^*x$ and $xx^*$ are orthogonal
\[
\|[x^*,x]\| = \max\{\|x^*x\|, \|xx^*\|\} = \frac \delta 2= \frac 1 3+\frac \epsilon 2.
\]
Also, by our choice of $\delta$, the spectrum of $\tilde h_1$ is exactly the interval
\[
[l(h_1)-\frac 1 3-\frac \epsilon 2,r(h_1)+\frac 1 3+\frac \epsilon 2].\]
From this and
 $\|h_1\|\leq1$, we see that  $\|\tilde h_1\|\leq 4/3+\epsilon/2$.
Moreover, $\tilde h_1$ is non-invertible. Indeed, its spectrum contains the closed interval $[l(h_1),r(h_1)]$ and the latter contains $0$ since $h_1$ is zero on all traces
and $T(A)\neq \varnothing$ (as argued before this theorem).
By property (iii) of the $C^*$-subalgebra $B$ above, there is a self-adjoint element $b\in B$ which is approximately unitarily equivalent to $\tilde h_1$. Notice, from \eqref{decomp}, that
$\sup\{|\tau(\tilde h_1)|:\tau\in T(A)\}<\epsilon$.
 Hence $ \sup\{|\tau(b)|:\tau\in T(A)\}<\epsilon$. But
 \[
 \sup\{|\tau(b)|:\tau\in T(A)\}= \sup\{|\tau(b)|:\tau\in T(B)\},
 \]
since tracial states of $B$ extend to tracial states of $A$ (property (ii) of $B$ above).
Hence $\sup\{|\tau(b)|:\tau\in T(B)\}<\epsilon$.  It follows that there exists a self-adjoint $b'\in B$ such that $\tau(b')=0$ for all $\tau\in T(B)$ and $\|b-b'\| < \epsilon$ (by \cite[Theorem 2.9]{Cuntz-Ped} and
\cite[Proof of Lemma 3.1]{Thomsen95}).
Hence there is a unitary conjugate of $\tilde h_1$ which has distance from $b'$ less than $\epsilon$.
 Thus, since  $B$ has  \descref{(P)} with constants $(M',\Upsilon')$, there exists  an  average of $M'$ unitary conjugates of $\tilde h_1$
of norm at most
\[
\Upsilon'(\|\tilde h_1\|+\epsilon)+\epsilon\leq \Upsilon '(\frac 4 3+\frac{3\epsilon}{2})+\epsilon.
\]
Applying this average on both sides of \eqref{decomp}, we find an average of $3M'$ unitary conjugates of the original element
$h$  with norm at most
\[
\Upsilon '(\frac 4 3+\frac{3\epsilon}{2})+\epsilon+(\frac 1 3+\frac{\epsilon}{2}) +\epsilon.
\]
Since $\epsilon$ can be chosen arbitrarily small, we find that, provided $\frac 4 3\Upsilon'+\frac 1 3<1$, $A$ has \descref{(P)} with constants
\begin{equation}\label{Aconstants}
M=3M' \hbox{ and every } \Upsilon\in(\frac 4 3\Upsilon'+\frac 1 3,1).
\end{equation}

Finally, let us find suitable constants for the unitisation $B+\C1$ of $B$. Since $B$ has decomposition rank  $1$ and stable rank one, $B+\C1$ has the same properties and so has \descref{(P)} with constants $M'=15$ and arbitrary $\Upsilon'\in(11/15,1)$ by Theorem~\ref{theorem ranks one}. Therefore, it also has \descref{(P)} with constants $M'=15^3$ and arbitrary $\Upsilon'\in((11/15)^3,1)$.
Putting the latter constants into the formula \eqref{Aconstants}, we get that $A$ has \descref{(P)} with constants
$M=3\cdot 15^3$ and $\Upsilon=0.86$.
\end{proof}

\bigskip

\section{The distance between Dixmier sets}
\label{sec:Distance}

In this section, we derive results about the distance between Dixmier sets $D_A(a)$ and $D_A(b)$. Along the way,
we obtain a description of $D_A(a)\cap Z(A)$ for $C^*$-algebras with the Dixmier property and we point out some cases in which the distance between $D_A(a)$ and $D_A(b)$ is attained.
Here by \emph{the distance between two subsets $D_1,D_2$} of a $C^*$-algebra, we mean
\[ d(D_1,D_2):= \inf \{\|d_1-d_2\|: d_1 \in D_1, d_2 \in D_2\}. \]

\begin{lemma}\label{AAbidual}
Let $A$ be a unital $C^*$-algebra and let $a,b\in A$. The distance between the sets $D_A(a)$ and $D_A(b)$ is equal to the distance between the sets $D_{A^{**}}(a)$	and $D_{A^{**}}(b)$.
\end{lemma}	
\begin{proof}
Let $r:=d(D_{A^{**}}(a),D_{A^{**}}(b))$. 	It is clear that $r\leq d(D_A(a),D_A(b))$. Let us prove the opposite inequality.
Let $\epsilon>0$ be given. Let $a'\in D_{A^{**}}(a)$ and $b'\in D_{A^{**}}(b)$ be such that $\|a'-b'\|<r+\epsilon$.
Approximating $a'$ and $b'$ by averages of unitary conjugates there exists some $N$ and unitaries $u_1,\dots,u_N,v_1,\dots,v_N \in \mathcal U(A^{**})$ such that
\begin{equation}
\label{DixDistDDual}
\Big\|\frac1N\sum_{i=1}^N u_iau_i^* - \frac1N \sum_{i=1}^N v_ibv_i^*\Big\| < r+\epsilon.
\end{equation}
By  the version of Kaplansky's density theorem for unitaries \cite[Theorem 4.11]{Takesaki-I} (due to Glimm and Kadison, see \cite[Theorem 2]{Glimm-Kadison}), there exist commonly indexed nets of unitaries $(u_{i,\lambda})_{\lambda\in\Lambda}, (v_{i,\lambda})_{\lambda\in\Lambda} \in \mathcal U(A)$ such that $u_{i,\lambda} \to u_i$ and $v_{i,\lambda} \to v_i$ in the ultrastrong$^*$-topology for $i=1,\dots,N$.
Now consider
\[ S:= \mathrm{co} \Big\{\frac1N\sum_{i=1}^N u_{i,\lambda}au_{i,\lambda}^* - \frac1N \sum_{i=1}^N v_{i,\lambda}bv_{i,\lambda}^*: \lambda \in \Lambda\Big\}. \]
The weak$^*$-closure of this convex set (in $A^{**}$) contains an element of norm less than $r+\epsilon$ (namely, the element appearing in \eqref{DixDistDDual}), so by the Hahn--Banach theorem, $S$ must also contain an element of norm less than $r+\epsilon$.
(Otherwise, the Hahn--Banach theorem ensures the existence of a  functional $\lambda \in A^*$ such that $\mathrm{Re}(\lambda(x))<r+\epsilon$ for all $\|x\|<r+\epsilon$ and  $\mathrm{Re}(\lambda(s))\geq r+\epsilon$ for all $s \in S$; but then $\|\lambda\|\leq 1$ and $\mathrm{Re}(\lambda(s))\geq r+\epsilon$ for all $s$ in the weak$^*$-closure of $S$, which is a contradiction.)
However, note that $S \subseteq D_A(a)-D_A(b)$, so this shows that $d(D_A(a),D_A(b)) \leq r+\epsilon$.
Since $\epsilon$ is arbitrary, we are done.
\end{proof}

Given a unital $C^*$-algebra $A$ and $a\in A$, let $W_A(a):=\{\rho(a): \rho\in S(A)\}$, the \emph{algebraic numerical range of $a$}. Since the state space $S(A)$ is weak$^*$-compact and convex, $W_A(a)$ is a compact convex subset of $\C$.

\begin{lemma}\label{easyhalf}
Let $A$ be a unital $C^*$-algebra and let $a,b\in A$. The following are true:
\begin{enumerate}[label=\emph{(\roman*)}]
\item
$|\tau(a)-\tau(b)|\leq d(D_A(a),D_A(b))$ for all $\tau\in T(A)$.
\item
$d(W_{A/I}(q_I(a)),W_{A/I}(q_I(b)))\leq d(D_A(a),D_A(b))$ for all closed ideals $I$ of $A$.
\end{enumerate}
\end{lemma}
\begin{proof}
(i): This is clear from the fact that traces are constant on Dixmier sets.

(ii): Since
\[
d(D_{A/I}(q_I(a)),D_{A/I}(q_I(b)))\leq d(D_A(a),D_A(a))
\]
(because $q_I(D_A(a))\subseteq D_{A/I}(q_I(a))$ and similarly for $b$),
it suffices to consider the case when $I=0$.  We have
\begin{align*}
\inf \{|\rho_1(a)-\rho_2(b)|: \rho_1,\rho_2\in S(A)\} &\leq \sup \{|\rho(a)-\rho(b)|: \rho\in S(A)\}\\
&\leq \|a-b\|.
\end{align*}
Thus, $d(W_A(a),W_A(b))\leq \|a-b\|$. But if $\alpha,\beta\in \mathrm{Av}(A,\mathcal U(A))$ are averaging
operators then $W_A(\alpha(a))\subseteq W_A(a)$ and $W_A(\beta(b))\subseteq W_A(b)$. So
\[
d(W_A(a),W_A(b))\leq \|\alpha(a)-\beta(b)\|.
\]
Passing  to the infimum over all $\alpha,\beta\in \mathrm{Av}(A,\mathcal  U(A))$ we get
that
\[
d(W_A(a),W_A(b))\leq d(D_A(a),D_A(b)),\]
as desired.
\end{proof}

Lemma \ref{AAbidual} allows us to reduce the calculation of the distance between Dixmier sets to the case that the ambient $C^*$-algebra is a von Neumann algebra. To deal with the von Neumann algebra case we rely on the following   theorem of Halpern and Str\u{a}til\u{a}--Zsid\'{o}:

\begin{thm}[\cite{Halpern}*{Theorem 4.12}, \cite{StratilaZsido}*{Proposition 7.3}]\label{halpernthm}
Let $M$ be a properly infinite von Neumann algebra with center $Z$ and strong radical  $J$ (i.e., $J$ is the intersection of all maximal ideals of $M$). Let $a\in M$.  The following are equivalent:
\begin{enumerate}[label=\emph{(\roman*)}]
\item
$0\in D_M(a)$.
\item
There exists a  $Z$-linear, positive,  unital map $\phi : M\to Z$ such that
$\phi(J)=0$ and $\phi(a)=0$.
\item
$0\in W_{M/I}(q_I(a))$ for every maximal ideal $I$ of $M$.
\end{enumerate}
\end{thm}
\begin{proof}
(i)$\Rightarrow$(iii): If $0\in D_M(a)$ then $0\in D_{M/I}(q_I(a))$ for every maximal ideal $I$ of $M$. By Lemma \ref{easyhalf},
$0\in W_{M/I}(q_I(a))$, as desired.

(iii)$\Rightarrow$(ii):
This is \cite[Proposition 7.3]{StratilaZsido}.

(ii)$\Rightarrow$(i): This follows from Halpern's \cite[Theorem 4.12]{Halpern}.
\end{proof}

The following theorem extends Theorem \ref{NRS theorem} (\cite[Theorem 4.7]{NRS}) to non-self-adjoint elements.

\begin{thm}\label{0inDa}
Let $A$ be a unital $C^*$-algebra and let $a\in A$. Then $0\in D_A(a)$
if and only if
\begin{enumerate}[(a)]
\item
$\tau(a)=0$ for all $\tau\in T(A)$, and
\item
in no nonzero quotient of $A$ can the image of $\mathrm{Re}(wa)$, with $w\in \C$, be invertible and negative.
\end{enumerate}
\end{thm}

Condition (b) need only be checked on all the quotients by maximal ideals of $A$.  A reformulation of (b) is
\begin{enumerate}
\item[(b')]
on every nonzero quotient there exists a state that vanishes on $a$; i.e., $0\in W_{A/I}(q_I(a))$ for all closed ideals $I$ of $A$.
\end{enumerate}
To see this, suppose that $\rho(a)=0$ for some $\rho\in S(A)$. For all $w\in \C$, $\rho(wa)=0$ and so $\rho(\mathrm{Re}(wa))=0$. Hence $\mathrm{Re}(wa)$ is not invertible and negative.
Conversely, suppose that $0\notin W_A(a)$. Then by convexity of $W_A(a)$, for a suitable $w\in \C$ and $\epsilon>0$, $\mathrm{Re}(\rho(wa))\leq -\epsilon<0$ for all states
$\rho$, i.e., $\rho(\mathrm{Re}(wa))\leq -\epsilon<0$ for all states $\rho$. But this implies that $\mathrm{Re}(wa)$ is negative and invertible. This equivalence holds similarly in every nonzero quotient.
 Notice that if every nonzero quotient of $A$ has a tracial state then (b') follows from (a).

Another reformulation of (b)  is the following:
\begin{enumerate}
\item[(b'')]
$(\mathrm{Re}(wa)+t)_+$  is a full element (i.e., generates $A$ as a closed two-sided ideal) for all $t>0$ and all $w\in \C$.
\end{enumerate}
To see that this is equivalent to Theorem \ref{0inDa} (b), notice first that $\mathrm{Re}(w\overline a)\leq -t1$  in the quotient
by the closed two-sided ideal generated by $(\mathrm{Re}(wa)+t)_+$, where $a\mapsto \overline a$ is the quotient map for this ideal. So, assuming  (b), this quotient must be
$\{0\}$, i.e., $(\mathrm{Re}(wa)+t)_+$ is full for all $t>0$.
On the other hand, if $\mathrm{Re}(w\overline a)\leq -t1$ in the  quotient
by some ideal $I$, then clearly $(\mathrm{Re}(wa)+t1)_+\in I$.  So, if $t>0$, and assuming (b''),   we get that $I=A$.

\begin{proof}[Proof of Theorem \ref{0inDa}]
 Since traces are constant on Dixmier sets, if $0\in D_A(a)$ then $\tau(a)=0$ for all $\tau\in T(A)$, i.e., (a) holds. Also, if $0\in D_A(a)$ then $0\in D_A(wa)$ for any
  $w\in \C$ (indeed, any central element) and this prevents $\mathrm{Re}(w a)$ from being  invertible and   negative. The same holds for quotients
since $q_I(D_A(a))\subseteq D_{A/I}(q_I(a))$. Thus, (b) holds as well.

Suppose now that $a\in A$ is such that (a) and (b) hold. If $A\subseteq B$ (where $B$ is a $C^*$-algebra with the same unit as $A$) then (a) and (b)  also hold in $B$. This is clear for condition (a), since traces of $B$ restrict to traces of $A$. This is also clear for condition (b''), which is equivalent to (b).
It follows that $a$ satisfies (a) and (b) in the von Neumann algebra $A^{**}$.   Let $A_{\mathrm f}^{**}\oplus A_{\mathrm{pi}}^{**}$
be the decomposition of $A^{**}$ into a finite and a properly infinite von Neumann algebra and let $a=a_{\mathrm f}+a_{\mathrm{pi}}$
be the corresponding decomposition of $a$. From condition (a) we get that  $R(a_{\mathrm f})=0$, where $R$ denotes the center-valued trace, which in turn implies  that $0\in D_{A^{**}_{\rm f}}(a_{\mathrm f})$. On the other hand, from condition (b)
we get that  for every maximal ideal $I$ of $A_{\mathrm{pi}}^{**}$  there exists a state on $A_{\mathrm{pi}}^{**}/I$ that vanishes on $q_I(a_{\rm pi})$. By Theorem \ref{halpernthm}, $0\in D_{A^{**}_{\rm pi}}(a_{\mathrm{pi}})$.
Since we may extend unitary elements in $A_{\rm f}^{**}$ (respectively $A_{\rm pi}^{**}$) by adding the unit of $A_{\rm pi}^{**}$ (respectively $A_{\rm f}^{**}$), we conclude that $0\in D_{A^{**}}(a)$. By Lemma \ref{AAbidual},  $0\in D_A(a)$.
\end{proof}

The next result extends the discussion in Section~\ref{DPandSDP} (after Theorem \ref{DP-characterisation}) concerning the form of the sets $D_A(a)\cap Z(A)$ in a unital $C^*$-algebra $A$ which has the Dixmier property. Here, the element $a\in A$ is not required to be self-adjoint.

\begin{cor} \label{cor Zset}
	Let $A$ be a unital $C^*$-algebra with the Dixmier property and let $a\in A$.
	Let $Y\subseteq \mathrm{Max}(A)$ be the closed set of maximal ideals $M$ such that $A$ has a (unique) tracial state $\tau_M$ that vanishes on $M$.
	Then  $D_A(a)\cap Z(A)$ is mapped, via the Gelfand transform, onto the set of $f\in C(\mathrm{Max}(A))$ such that
	\begin{align*}
	f(M)&=\tau_M(a)\hbox{ if $M\in Y$},\\
	f(M)&\in W_{A/M}(q_M(a))\hbox{ otherwise}.
	\end{align*}
\end{cor}	
\begin{proof}
Let $z\in D_A(a)\cap Z(A)$ and let $f\in C(\mathrm{Max}(A))$ be its Gelfand transform (that is, $f=\theta(z)$ where $\theta: Z(A)\to C(\mathrm{Max}(A))$ is the canonical $^*$-isomorphism discussed prior to Corollary~\ref{cor vNalg}).
Let $M\in\mathrm{Max}(A)$. Since $0\in D_{A}(a-z)$, we have by Lemma \ref{easyhalf} (ii) that
\[
0\in W_{A/M}(q_M(a-z))=W_{A/M}(q_M(a))-f(M),
\]
i.e., $f(M)\in W_{A/M}(q_M(a))$. Also, $f(M)=\tau_M(z)=\tau_M(a)$ for all $M\in Y$. Thus, $f$ is as required.

Conversely, let $f\in C(\mathrm{Max}(A))$ be as in the statement of the corollary. Let $z\in Z(A)$ be the central element whose
Gelfand transform is $f$. Then
\[
0\in W_{A/M}(q_M(a))-f(M)=W_{A/M}(q_M(a-z))
\]
for all $M\in \mathrm{Max}(A)$. Also, $\tau_M(a-z)=0$ for all $M\in Y$, and since $\partial_e T(A)=\{\tau_M: M\in Y\}$ (Theorem \ref{DP-characterisation}),
$\tau(a-z)=0$ for all $\tau\in T(A)$ by the Krein--Milman theorem. By Theorem \ref{0inDa}, $0\in D_A(a-z)$, i.e., $z\in D_A(a)$, as desired.
\end{proof}

Our next goal is to  extend Theorem \ref{0inDa} to a distance formula between the sets $D_A(a)$ and $D_A(b)$ (Theorem \ref{DaDb} below).
Note that one cannot reduce the calculation of this distance to the case that one element is $0$  by looking at the distance between $D_A(b-a)$ and $0$, since $d(D_A(a),D_A(b))$ is in general not the same as $d(D_A(b-a),0)$.
For an example of this, let $a$ be a non-invertible positive element of norm 1 in a simple unital infinite $C^*$-algebra $A$ and define $b:=1+a$.
Then $D_A(a) \cap Z(A) = [0,1]$ and $D_A(b) \cap Z(A) = [1,2]$ (as sets of scalar elements of $A$) (see Corollary~\ref{genHZ1} or \cite{HZ}), so that $d(D_A(a),D_A(b))=0$.
However, $b-a=1$ so that $d(D_A(b-a),0)=1$.

\begin{lemma}[\cite{Dix}*{Proposition 3.4.2 (i)}]
Let $(I_\lambda)_\lambda$ be a collection of closed ideals of a $C^*$-algebra $A$ and let $I$ be a closed ideal of $A$
such that $\bigcap_\lambda I_\lambda\subseteq I$. Then every state of $A$ which vanishes on $I$ is
a weak$^*$-limit of convex combinations of states vanishing on the $I_\lambda$'s.
\end{lemma}

Recall that for topological spaces $X$ and $Y$, a set-valued function $\phi:X \to \{\text{subsets of $Y$}\}$ is defined to be \emph{lower semicontinuous} if for every open set $U$ of $Y$, the set
\[ \{x \in X: \phi(x) \cap U \neq \emptyset\} \]
is open in $X$.
Later in this section we will use the Michael selection theorem (\cite[Theorem 3.1']{Michael}).

The next lemma is implicit in a strategy mentioned in \cite{Magajna}.

\begin{lemma}\label{Wlsc}
Let $A$ be a unital $C^*$-algebra and let $a\in A$. The  set-valued function on $\mathrm{Max}(A)$ defined by
\[
M\mapsto W_{A/M}(q_M(a)) \hbox{ for all }M\in \mathrm{Max}(A).
\]
is lower semicontinuous.
\end{lemma}
\begin{proof}
Let $\Phi_a(M):= W_{A/M}(q_M(a))$ for all  $M\in \mathrm{Max}(A)$. 	
Let $M\in \mathrm{Max}(A)$, $w\in \Phi_a(M)$ and $\epsilon>0$. We must show that $\Phi_a(M')\cap B_\epsilon(w)$ is non-empty
for all $M'$ in a neighbourhood of $M$.   Suppose, for the sake of contradiction,  that there exists a net  $M_\lambda\to M$ such that  $\Phi_a(M_\lambda)\cap B_\epsilon(w)=\varnothing$ for all $\lambda$. For each $\lambda$ we can separate the sets
$\Phi_a(M_\lambda)$ and $B_\epsilon(w)$ by a line $\ell_\lambda$ tangent to the circle $\{z: |z-w|=\epsilon\}$. Let $c_\lambda\in \ell_\lambda$ denote the point of tangency. Let us pass to a  subnet $M_{\lambda'}\to M$ such that the
$c_{\lambda'}\to c$, and let $\ell$ be the tangent line at $c$. Since the sets $\Phi_a(M_{\lambda'})$ are uniformly bounded (they are all inside the ball $\overline{B_{\|a\|}(0)}$), there exists  $\lambda_0'$
such that  the sets
$\Phi_a(M_{\lambda'})$ for $\lambda'\geq \lambda_0'$ are all separated  from the ball $B_{\epsilon/2}(w)$ by a single line $\ell_0$ parallel to $\ell$ (and tangent to the circle $\{z: |z-w|=\epsilon/2\}$). But, since $\bigcap_{\lambda'\geq \lambda_0'} M_{\lambda'}\subseteq M$, we have by the previous lemma that any state of $A$ which vanishes on $M$ is a weak$^*$-limit of  convex combinations of states vanishing on the $M_{\lambda'}$'s.  In particular, $w$ ($=\rho(a)$ for some state $\rho$ of $A$ which vanishes on $M$) is a limit of convex combinations of elements in $\bigcup_{\lambda'\ge \lambda_0'} \Phi_a(M_{\lambda'})$. This contradicts that we can separate $\bigcup_{\lambda'\ge \lambda_0'} \Phi_a(M_{\lambda'})$ from  $B_{\epsilon/2}(w)$  by the  line $\ell_0$.

Let us describe more specifically how to obtain $\lambda_0'$.
The line $\ell_0$ is parallel to $\ell$ but closer to $w$.
We may therefore choose $\lambda_0'$  such that all the points $\{c_{\lambda'}: \lambda'\geq \lambda_0'\}$ lie on the same side of $\ell_0$ and such that the lines  $\ell_{\lambda'}$ and $\ell_0$ intersect outside of the ball $\overline{B_{\|a\|}(0)}$ for all $\lambda'\geq \lambda_0'$. Then $\lambda_0'$ is as desired.
\end{proof}

Given a subset $S$ of a metric space and $r>0$ we denote by $S^r$ the set $\{y: d(y,S)<r\}$.
\begin{lemma}\label{lsccap1}
Let $f,g$ be lower semicontinuous  set-valued functions on a topological set $X$ taking values in the   subsets of a metric space $Y$.
Let $r>\sup \{d(f(x),g(x)): x\in X\}$. Then the set-valued functions $x\mapsto f(x)\cap (g(x))^r$
and $x\mapsto \overline{f(x)\cap (g(x))^r}$ are lower semicontinuous.
\end{lemma}	
\begin{proof}
Let us show that  $h(x):=f(x)\cap (g(x))^r$ is lower semicontinuous.
Let $x\in X$, $z\in h(x)$ and $\epsilon>0$.
We must show that there exists a neighbourhood $U$ of $x$ such that
$h(y)\cap B_\epsilon(z)\neq \varnothing$ for all $y\in U$. Let $w\in g(x)$ be such that
$r':=d(z,w)<r$.  Let $\delta:=\min((r-r')/2,\epsilon)$. By the lower semicontinuity of $f$
and $g$ we can find a neighbourhood $U$ of $x$ such that $f(y)\cap B_\delta(z)$
and $g(y)\cap B_\delta(w)$ are nonempty for all $y\in U$.
Let $y\in U$, so that there exist  $z'\in f(y) \cap B_\delta(z)$ and
$w'\in g(y) \cap B_\delta(w)$. Then, using the triangle inequality, $d(z',w') < r$,
so that $z'\in h(y)$. Also by the choice of $\delta$, $z'\in B_\epsilon(z)$, so
$h(y) \cap B_\epsilon(z)$ is nonempty, as required.

Let us show that $x\mapsto \overline{h(x)}$ is also lower semicontinuous. Let $V\subseteq Y$ be an open set.
Suppose that $\overline{h(x)}\cap V\neq \varnothing$ for some $x\in X$. Then
$h(x)\cap V\neq \varnothing$, and by the lower semicontinuity of $h$ we find a neighbourhood
$U$ of $x$ such that $h(y)\cap V\neq \varnothing$ for all $y\in U$. Then,
$\overline{h(y)}\cap V\neq \varnothing$ for all $y\in U$, as required.
 \end{proof}	

\begin{lemma}\label{lsccap2}
	Let $r>0$.
	Let $f$ be a lower semicontinuous set-valued function on a topological space $X$ taking values in the convex subsets of $\C$
	and such that $f(x)\cap \overline{B_r(0)}\neq \varnothing$ for all $x$.
	Then
	\[
	h(x):=f(x)\cap \overline{B_r(0)}
	\]
	is lower semicontinuous.
\end{lemma}
\begin{proof}
	Let $x\in X$ and $z\in h(x)$. Let $\epsilon>0$ and,  without loss of generality, assume   $\epsilon< r$. We must show that $h(y)\cap B_{\epsilon}(z)$ is nonempty
	for all $y$ in a neighbourhood of $x$.  Suppose first that $|z|<r$. Let $\delta:=\min(\epsilon,r-|z|)$. Then $B_\delta(z)\subseteq B_r(0)$, by the triangle inequality.
	Since $f$ is lower semicontinuous, $f(y)\cap B_\delta(z)\neq \varnothing$ for all $y$ in a neighbourhood of $x$,
	and so  $h(y)\cap B_{\epsilon}(z)\neq \varnothing$ for all such $y$.

	Assume now that $|z|=r$.
	Let $\delta := \epsilon^2/2r$, as shown in the diagram, so that the circle of center $z$ and radius $\delta$
	is tangent to the segment $[A,B]$.
	Since $f$ is lower semicontinuous, $f(y)\cap B_\delta(z)\neq \varnothing$ for all $y$ in a neighbourhood
	$U$ of $x$. Let $y\in U$. Say $z_1\in f(y)\cap B_\delta(z)$. Recall also that, by assumption, there exists
	$z_2\in f(y)$ such that $|z_2|\leq r$.
	Since the segment $[z_1,z_2]$ is contained in $f(y)$, it suffices to show that
	$[z_1,z_2]$ intersects $B_{\epsilon}(z) \cap \overline{B_r(0)}$.
		\begin{figure}[h!]
		\includegraphics[scale=.9]{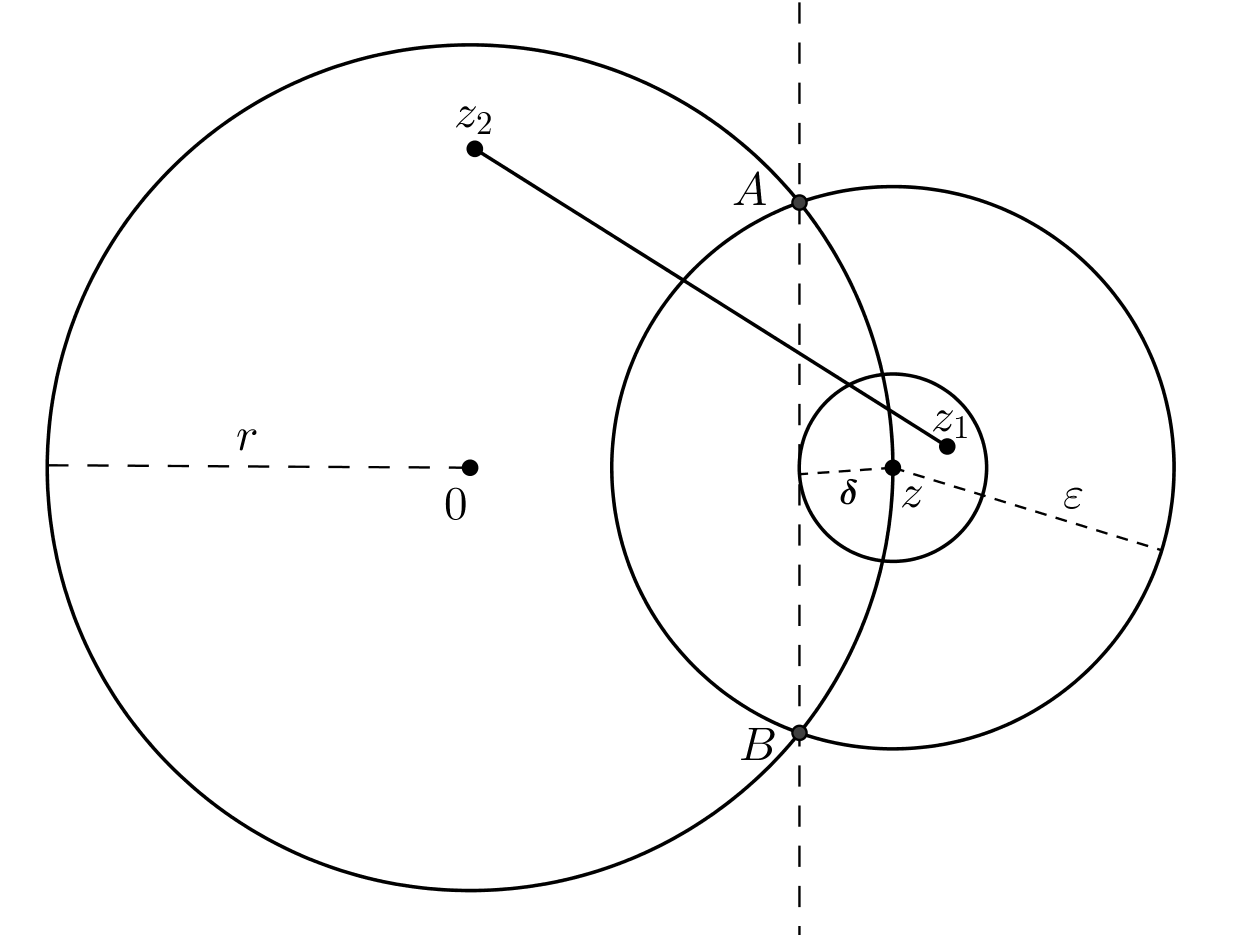}
		\end{figure}
	
		If the points $z_1$ and $z_2$
	are on the same side of the line $AB$ then $z_2\in B_\epsilon(z)$. If the points $z_1$ and $z_2$ are on different sides
	of this line $AB$ (as in the figure) then the segment $[z_1,z_2]$ intersects the segment $[A,B]$.
	 (Note for this that the tangents at $A$ and $B$ to the circle centred at 0 are also tangential to the circle centred at $z$
	with radius $\delta$.)
	\end{proof}

Let $A$ be a  unital $C^*$-algebra $A$ with the Dixmier property.
Let $Y\subseteq \mathrm{Max}(A)$ be the set of maximal ideals $M$ such that $A$
has a (unique) tracial state $\tau_M$ that vanishes on $M$. Recall that $Y$ is closed
and $M\mapsto \tau_M(a)$ is continuous on $Y$ for all $a\in A$ (Theorem \ref{DP-characterisation}). Let $a\in A$. Define a set-valued
function $F_a$ on $\mathrm{Max}(A)$ as follows:
\[
F_a(M):=
\begin{cases}
\{\tau_M(a)\} &\hbox{if }M\in Y,\\
W_{A/M}(q_M(a))&\hbox{otherwise.}
\end{cases}
\]
The values of $F_a$ are compact convex subsets of $\C$.
Since $M\mapsto W_{A/M}(q_M(a))$ is lower semicontinuous by Lemma \ref{Wlsc}, $Y$ is closed,  $F_a|_Y$ is continuous, and $\tau_M(a)\in W_{A/M}(q_M(a))$ for $M\in Y$, the set-valued function $F_a$ is lower semicontinuous.

The following proposition is trivial in the case of the singleton Dixmier property.

\begin{prop}\label{DaDbDP}
	Let $A$ be a unital $C^*$-algebra with the Dixmier property, and let $a,b\in A$.
	Set
	\[
	r:= \sup_{M \in \mathrm{Max}(A)} d(F_a(M),F_b(M)).
	\]
	Then the distance between $D_A(a)$ and $D_A(b)$ is equal to $r$. If either $a$ and $b$ are both self-adjoint,
	or $b=0$  then this distance is attained by elements in $D_A(a)\cap Z(A)$ and $D_A(b)\cap Z(A)$.
\end{prop}

\begin{proof}
The inequality $r\leq d(D_A(a),D_A(b))$ follows at once from Lemma \ref{easyhalf}.

Let $\epsilon>0$. By Lemma \ref{lsccap1}, the set-valued function
\begin{equation*}
F(M):=\overline{F_a(M)\cap (F_b(M))^{r+\epsilon}}\hbox{ for }M\in \mathrm{Max}(A).
\end{equation*}
is lower semicontinuous. Since its values are  closed convex sets,
by Michael's selection theorem there exists a continuous function $f: \mathrm{Max}(A)\to \C$ such that $f(M)\in F(M)$ for all $M$. Let $z_a\in Z(A)$ be the central element whose Gelfand transform is $f$.   Since  $f(M) \in F(M)\subseteq F_a(M)$ for all $M$ we have that $z_a\in D_A(a)$ by Corollary \ref{cor Zset}.
Let us define
\[
G(M):=\overline{\{f(M)\}^{r+2\epsilon}\cap F_b(M)} \hbox{ for }M\in \mathrm{Max}(A).
\]
Then again this is a lower semicontinuous function
taking closed convex set values. So there exists a continuous $g: \mathrm{Max}(A)\to \C$
such that $g(M)\in G(M)$ for all $M$. Let $z_b\in Z(A)$ be the central element whose Gelfand transform is
$g$.
As with $z_a$, we have that $z_b\in D_A(b)$. Also, since
 $|f(M)-g(M)|\leq r+2\epsilon$ for all $M$ we have that $\|z_a-z_b\|\leq r+2\epsilon$.
This ends the proof that $r=d(D_A(a),D_A(b))$.

Suppose now that $b=0$, and let us show that the distance from $D_A(a)$ to 0 is attained. Since  $r=\sup \{d(0,F_a(M)): M\in \mathrm{Max}(A)\}$,  the set $F_a(M)\cap \overline{B_r(0)}$ is nonempty
for all  $M$. Thus, by Lemma \ref{lsccap2}, the set-valued function $M\mapsto F_a(M)\cap \overline{B_r(0)}$ is lower semicontinuous. Since it takes values
on the closed convex subsets of $\C$, there exists, by Michael's selection theorem, a continuous function
$f: \mathrm{Max}(A)\to \C$ such that $f(M)\in F_a(M)\cap \overline{B_r(0)}$ for all $M$. Let $z_a$
be the central element whose Gelfand transform is $f$. Then $z_a\in D_A(a)$ and $\|z_a\|\leq r$, as desired.

Finally, suppose that $a$ and $b$ are self-adjoint.
Then
\begin{align*}
W_{A/M}(q_M(a)) &=[f_a(M),h_a(M)],\\
W_{A/M}(q_M(b)) &=[f_b(M),h_b(M)]
\end{align*}
for all $M\in \mathrm{Max}(A)$. Here $f_a(M):=\min (\mathrm{sp}(q_M(a)))$, $h_a(M):=\max(\mathrm{sp}(q_M(a)))$ and similarly for $f_b(M)$ and $h_b(M)$. As in the proof of
Theorem \ref{DP-characterisation},
$f_a,f_b: \mathrm{Max}(A)\to \R$ are upper semicontinuous functions and $h_a,h_b: \mathrm{Max}(A)\to \R$
are lower semicontinuous. For each $M\in \mathrm{Max}(A)$ define
\[
G(M)=\begin{cases}
\{\tau_M(a)\} &\hbox{ if }M\in Y,\\
[f_a(M),h_a(M)]\cap [f_b(M)-r,h_b(M)+r] & \hbox{otherwise.}
\end{cases}
\]
Observe that $G(M)$ is a nonempty closed interval for all $M$. Moreover, the assignment $M\mapsto G(M)$ is a lower semicontinuous set-valued function. Hence, it has a continuous selection $g(M)\in G(M)$, $g\in C(\mathrm{Max}(A))$. (Alternatively, we can derive the existence of  $g$ from the Kat\v{e}tov--Tong theorem as in the proof of Theorem \ref{DP-characterisation}.)  Let $z_a$ denote the central element whose Gelfand transform is $g$.
Then $z_a\in D_A(a)$. Now consider the assignment
\[
M\mapsto [g(M)-r,g(M)+r]\cap F_b(M).
\]
It is again lower semicontinuous and takes values in the closed intervals of $\R$. Hence, it has a continuous selection giving rise to a central element $z_b\in D_A(b)$ such that $\|z_a-z_b\|\leq r$.
\end{proof}	

\begin{example}
For general elements $a$ and $b$ in a $C^*$-algebra with the Dixmier property,  the distance from $D_A(a)$ to $D_A(b)$ need not be attained.
Let $A=C([-1,1],\mathcal O_2)$. Then $[-1,1]$ is homeomorphic to ${\rm Max}(A)$ via the assignment
$$ s\to M_s: =C_0([-1,1]\backslash \{s\},O_2) \text{ for } s\in [-1,1].$$
Since $A$ is weakly central and has no tracial states, it has the Dixmier property by Theorem~\ref{DP-characterisation} (this can also be seen from the fact that $A$ is $^*$-isomorphic to the tensor product of $\mathcal O_2$ with an abelian $C^*$-algebra).

Fix a non-invertible positive element $h\in  \mathcal O_2$ of norm
$1$ and define a continuous function $G: [-1,1] \times [0,1]\to \C$, by
\[
G(s,t) :=
\begin{cases}
 (1+si)t & \hbox{if }s\in [-1,0],\\
si+(1-si)t&  \hbox{if }s\in[0,1].
\end{cases}
\]
Now define the set-valued function
\[
F(s) := \{G(s,t): t\in [0,1]\}, \hbox{ for }s\in [-1,1].
\]
Observe that the values of $F$ are closed intervals in $\C$ (for $s\in [-1,0]$
the set $F(s)$ is an interval swinging like a door with the hinges at 0, while for $s\in [0,1]$ the interval $F(s)$
also swings but with the hinges at $1$.)

Now define $a,b\in A$ by $a(s):= G(s,h)$ (functional calculus), and $b(s):=h$ for all $s\in [-1,1]$.
One can see then that $F_a(M_s) = F(s)$ and
$F_b(M_s) = [0,1]$ for all $s$. It follows by the previous proposition that the distance between $D_A(a)$ and $D_A(b)$ is 0. However, $D_A(a)$ and $D_A(b)$ have no elements it common. For if they did, then $D_A(a)\cap D_A(b)\cap Z(A)$ would be nonempty. By Corollary~\ref{cor Zset}, elements of
$D_A(a) \cap D_A(b)\cap Z(A)$ correspond to continuous selections of
$s\mapsto F_a(M_s)\cap F_b(M_s)$. However, there are no such continuous selections, because
\[
F_a(M_s) \cap F_b(M_s)=\begin{cases}
\{0\} & \hbox{for }s\in [-1,0),\\
[0,1] & \hbox{for }s=0,\\
\{1\} &\hbox{for }s\in(0,1].
\end{cases}
\]
\end{example}

We now extend the distance formula from Proposition \ref{DaDbDP} to arbitrary $C^*$-algebras.
The following result gives a formula for the distance between the Dixmier sets of two elements of an arbitrary unital
$C^*$-algebra.
A similar result is \cite[Theorem 4.3]{NRS}, which gives a formula for the distance between one self-adjoint element and the Dixmier set of another; these results say the same thing in the case that both  elements
are self-adjoint and one is central.

\begin{thm}\label{DaDb}
	Let $A$ be a unital $C^*$-algebra and let $a,b \in A$.
	Then the following numbers are equal:
	\begin{enumerate}[label=\emph{(\roman*)}]
		\item The distance between $D_A(a)$ and $D_A(b)$.
		\item The minimum number $r\geq 0$ satisfying
		\begin{enumerate}
			\item $|\tau(a-b)|\leq r$ for all $\tau \in T(A)$, and
			\item $d(W_{A/M}(q_M(a)),W_{A/M}(q_M(b))\leq r$ for all $M \in \mathrm{Max}(A)$.
		\end{enumerate}
	\end{enumerate}
\end{thm}
\begin{proof}
The inequality  $r\leq d(D_A(a),D_A(b))$ has already been proven in Lemma \ref{easyhalf}.

We check that (ii)(a) and (ii)(b) with $A^{**}$ in place of $A$ still hold (without changing $r$).
For (ii)(a), this follows since every tracial state on $A^{**}$ restricts to a tracial state on $A$.
Similarly, for any ideal $I$ of $A^{**}$, since $A/(I\cap A) \subseteq A^{**}/I$,
\[ W_{A/I\cap A}(q_{I\cap A}(a)) =W_{A^{**}/I}(q_I(a)). \]
From this we see that (ii)(b) holds  for $A^{**}$.
But $A^{**}$, being a von Neumann algebra, has the Dixmier property. Hence $r\geq d(D_{A^{**}}(a),D_{A^{**}}(b))$
by Proposition \ref{DaDbDP}. The theorem now follows from Lemma \ref{AAbidual}.
\end{proof}


\begin{thebibliography}{99999}

\bibitem{AkJo} C.A. Akemann and B.E. Johnson, Derivations of nonseparable $C^*$-algebras, J. Funct. Anal., 33 (1979) 311--331.

\bibitem{APT} C.A. Akemann, G.K. Pedersen and J. Tomiyama, Multipliers of $C^*$-algebras, J. Funct. Anal., 13 (1973) 277--301.

\bibitem{Arch-thesis} R.J. Archbold, Certain properties of operator algebras, Ph.D. Thesis, The University of Newcastle-upon-Tyne, 1972.

\bibitem{Arch-centre-tensor} R.J. Archbold, On the centre of a tensor product of $C^*$-algebras, J. London Math. Soc. (2), 10 (1975) 257--262.

\bibitem{Arch-PLMS} R.J. Archbold, An averaging process for $C^*$-algebras related to weighted shifts, Proc. London Math. Soc. (3), 35 (1977) 541--554.


\bibitem{Arch-Camb} R.J. Archbold, On the Dixmier property of certain algebras, Math. Proc. Camb. Phil. Soc., 86 (1979) 251--259.

\bibitem{Arch-JLMS} R.J. Archbold, On the simple $C^*$-algebras of J. Cuntz, J. London Math. Soc. (2), 21 (1980) 517--526.






\bibitem{RJA} R.J. Archbold, On the norm of an inner derivation of a $C^*$-algebra,
 Math. Proc. Camb. Phil. Soc., 84 (1978) 273--291.

\bibitem{AKS} R.J. Archbold, E. Kaniuth, and D.W.B. Somerset, Norms of inner derivations for multiplier algebras of $C^*$-algebras and group $C^*$-algebras, II, Adv. Math., 280 (2015) 225--255.

\bibitem{AS} R.J. Archbold and D.W.B. Somerset, Measuring noncommutativity in $C^*$-algebras, J. Funct. Anal., 242 (2007) 247--271.


\bibitem{black-hand} B. Blackadar and D. Handelman,  Dimension functions and traces on $C\sp{*} $-algebras. J. Funct. Anal., 45 (1982), 297--340.


\bibitem{BedosJOT} E. B{\' e}dos, Operator algebras associated with free products of groups with amalgamation, Math. Ann., 266 (1984) 279--286.

\bibitem{BedosCAMB} E. B{\' e}dos, Discrete groups and simple $C^*$-algebras, Math. Proc. Camb. Phil. Soc., 109 (1991) 531--537.

\bibitem{BRTTW} B. Blackadar, L. Robert, A. Tikuisis, A.S. Toms, and W. Winter, An algebraic approach to the radius of comparison, Trans. Amer. Math. Soc., 364 (2012), 3657--3674.

\bibitem{BKKO} E. Breuillard, M. Kalantar, M. Kennedy, and N. Ozawa, $C^*$-simplicity and the unique trace property for discrete groups, arXiv:1410.2518 (2014) 20 pp.

\bibitem{Choi} M.D. Choi, A simple $C^*$-algebra generated by two finite order unitaries, Canad. J. Math., 31 (1979), 867--880.

\bibitem{Cuntz-Ped} J. Cuntz and G.K. Pedersen, Equivalence and traces on $C^*$-algebras, J. Funct. Anal., 33 (1979) 135--164.




\bibitem{Conway} J.B. Conway, The numerical range and a certain convex set in an infinite factor, J. Funct. Anal., 5 (1970) 428--435.

\bibitem{delaHarpe} P. de la Harpe, On simplicity of reduced $C^*$-algebras of groups, Bull. London Math. Soc., 39 (2007) 1--26.

\bibitem{delaHarpeSkandalis} P. de la Harpe and G. Skandalis, Les r\'eseaux dans les groupes semi-simples ne sont pas int\'erieurement moyennables, Enseign. Math. (2), 40 (1994) 291--311.

\bibitem{Delaroche} C. Delaroche, Sur les centres des $C^*$-alg{\` e}bres, Bull. Sci. Math., 91 (1967) 105--112.





\bibitem{Dix} J. Dixmier, $C^*$-algebras, North-Holland, Amsterdam,
1977.

\bibitem{Dix1949} J. Dixmier, Les anneaux d'op{\' e}rateurs de classe finie, Ann. Sci. {\' E}cole Norm. Sup. (3), 66 (1949) 209--261.

\bibitem{DixvN} J. Dixmier, Les alg{\` e}bres d'op{\' e}rateurs dans l'espace Hilbertien (alg{\` e}bres de von Neumann), 2nd edn., Cahiers Scientifiques, Fasc. XXV, Gauthier-Villars {\' E}diteur, Paris, 1969.

\bibitem{Elliott2} G.A. Elliott, Derivations of matroid $C^*$-algebras, II, Ann. of Math, 100 (1974) 407--422.

\bibitem{Elliott} G.A. Elliott, On derivations of AW$^*$-algebras, T{\^ o}hoku
Math. J., 30 (1978) 263--276.

\bibitem{ERS} G.A. Elliott, L. Robert, and L. Santiago, The cone of lower semicontinuous traces on a $C^*$-algebra, Amer. J. Math., 133 (2011) 969--1005.

\bibitem{FHLRTVW} I. Farah, B. Hart, M. Lupini, L. Robert, A. Tikuisis, A. Vignati, and W. Winter. Model theory of $C^*$-algebras, arXiv:1602.08072v3 (2016) 139 pp.

\bibitem{Gaj} P. Gajendragadkar, Norm of a derivation of a von Neumann algebra, Trans, Amer. Math. Soc., 170 (1972) 165--170.


\bibitem{Glimm} J.G. Glimm, On a certain class of operator algebras, Trans. Amer. Math. Soc., 95 (1960) 318--340.

\bibitem{GlimmSW} J.G. Glimm, A Stone--Weierstrass theorem for $C^*$-algebras,
Ann. of Math., 72 (1960) 216--244.

\bibitem{Glimm-Kadison} J. G. Glimm and R. V. Kadison, Unitary operators in $C^*$-algebras, Pacific J. Math., 10 (1960) 547--556.

\bibitem{Goldman} M. Goldman, Structure of AW$^*$-algebras, I, Duke Math. J., 23 (1956) 23--34.

\bibitem{Green} W.L. Green, Topological dynamics and $C^*$-algebras, Trans. Amer. Math. Soc., 210 (1975) 107--121.








\bibitem{Haagerup:qtraces} U. Haagerup,  Quasitraces on exact $C^*$-algebras are traces, C. R. Math. Acad. Sci. Soc. R. Can., 36 (2014), 67--92. Circulated in manuscript form in 1991.

\bibitem{Haagerup} U. Haagerup, A new look at $C^*$-simplicity and the unique trace property of a group, arXiv:1509:05880v1 (2015) 8 pp.

\bibitem{HZ} U. Haagerup and L. Zsid{\' o}, Sur la propri{\'e}t{\' e} de Dixmier pour les $C^*$-alg{\` e}bres, C. R. Acad. Sci. Paris S{\' e}r. I Math., 298 (1984) 173--176.

\bibitem{HiaiN} F. Hiai and Y. Nakamura, Closed convex hulls of unitary orbits in von Neumann algebras, Trans. Amer. Math. Soc., 323 (1991) 1--38.

\bibitem{Halpern1970} H. Halpern, Commutators modulo the center in a properly infinite von Neumann algebra, Trans. Amer. Math. Soc., 150 (1970) 55--68.

\bibitem{Halpern} H. Halpern, Essential central spectrum and range for elements of a von Neumann algebra, Pacific J. Math., 43 (1972) 349--380.

\bibitem{Halpern77} H. Halpern, Essential central range and selfadjoint commutators in properly infinite von Neumann algebras, Trans. Amer. Math. Soc., 228 (1977) 117--146.

\bibitem{Halpern86} H. Halpern, V. Kaftal, and G. Weiss, The relative Dixmier property in discrete crossed products, J. Funct. Anal., 69 (1986) 121--140.



\bibitem{JST} B. Jacelon, K. R. Strung, A. S. Toms,  Unitary orbits of self-adjoint operators in simple
$\mathcal {Z}$-stable $\rm C^*$-algebras, J. Funct. Anal. 269 (2015), 3304--3315.

 \bibitem{BEJ} B.E. Johnson, Characterization and norms of derivations on von Neumann algebras, Alg\`ebres d'op\'erateurs (S\'em., Les Plans-sur-Bex, 1978), pp.228--236, Lecture Notes in Math., 725, Springer, Berlin, 1979.



\bibitem{BEJ+JRR} B.E. Johnson and J.R. Ringrose, Derivations of operator algebras and discrete group algebras, Bull. London Math. Soc., 1 (1969) 70--74.


\bibitem{KLR} R.V. Kadison, E.C. Lance, and J.R. Ringrose, Derivations and automorphisms of operator algebras, II, J. Funct. Anal., 1 (1967)
204--221.

\bibitem{KadRing1967} R.V. Kadison and J.R. Ringrose, Derivations and automorphisms of operator algebras, Comm. Math. Phys., 4 (1967) 32--63.

\bibitem{KadRing} R.V. Kadison and J.R. Ringrose, Fundamentals of the theory of operator algebras, Vol. 2 (Advanced Theory), Academic Press, London, 1986.

\bibitem{KalKen} M. Kalantar and M. Kennedy, Boundaries of reduced $C^*$-algebras of discrete groups, J. Reine Angew. Math., to appear. arXiv:1405.4359 (2014), 26 pp.

\bibitem{Katetov} M. Kat\v{e}tov, On real-valued functions in topological spaces, Fund. Math., 38 (1951) 85--91.

\bibitem{Kennedy} M. Kennedy, Characterizations of $C^*$-simplicity, arXiv:1509.01870v3 (2015) 16 pp.

\bibitem{Kirch} E. Kirchberg, On subalgebras of the CAR-algebra, J. Funct. Anal., 129 (1995) 35--63.

\bibitem{Kirch:Abel} E. Kirchberg, Central sequences in $C^*$-algebras and strongly purely infinite algebras, Operator {A}lgebras: {T}he {A}bel {S}ymposium 2004. pp.175--231, Abel Symp., 1, Springer, Berlin, 2006.

\bibitem{KirRor} E. Kirchberg and M. R\o rdam, Central sequence $C^*$-algebras and tensorial absorption of the Jiang--Su algebra, J. Reine Angew. Math., 695 (2014) 175--214.

\bibitem{KirchbergWinter} E. Kirchberg and W. Winter, Covering dimension and quasidiagonality, Internat. J. Math., 15 (2004) 63--85.

\bibitem{LeBoudec} A. Le Boudec, $C^*$-simplicity and the amenable radical, arXiv:1507.03452 (2015), 13 pp.




\bibitem{lor-shul}
T. A. Loring and T.  Shulman,  Noncommutative semialgebraic sets and associated lifting problems, Trans. Amer. Math. Soc. 364 (2012), no. 2, 721--744.	

\bibitem{Magajna} B. Magajna, On weakly central $C^*$-algebras, J. Math. Anal. Appl., 342 (2008) 1481--1484.

\bibitem{Michael} E. Michael, Continuous selections I, Ann. of Math., 63 (1956) 361--382.

\bibitem{Misonou1} Y. Misonou and M. Nakamura, Centering of an operator algebra, T{\^ o}hoku Math. J. (2), 3 (1951) 243--248.

\bibitem{Misonou2} Y. Misonou, On a weakly central operator algebra, T{\^ o}hoku Math. J. (2), 4 (1952) 194--202.




\bibitem{Murphy} G.J. Murphy, Uniqueness of the trace and simplicity, Proc. Amer. Math. Soc., 128 (2000) 3563--3570.

\bibitem{NgRobert} P.W. Ng and L. Robert, Sums of commutators in pure $C^*$-algebras, M\"unster J. Math., 9 (2016) 121--154.

\bibitem{NgRobert2} P. W. Ng, L. Robert, The kernel of the determinant map on pure $C^*$-algebras, Houston J. Math., 43 (2017) 139--168.

\bibitem{NRS} P.W. Ng, L. Robert, and P. Skoufranis, Majorization in $C^*$-algebras, Trans. Amer. Math. Soc., to appear, arXiv:1608.04350v1 (2016) 32pp.

\bibitem{NS} P.W. Ng and P. Skoufranis, Closed convex hulls of unitary orbits in certain simple real rank zero $C^*$-algebra, Canad. J. Math., to appear, arXiv:1603.07059v1 (2016) 35 pp.

\bibitem{Ozawa} N. Ozawa, Dixmier approximation and symmetric amenability for $C^*$-algebras, J. Math. Soc. Univ. Tokyo, 20 (2013) 349--374.

\bibitem{Ped} G.K. Pedersen, $C^*$-algebras and their automorphism
groups, Academic Press, London, 1979.

\bibitem{Pedersen}G. K. Pedersen, Unitary extensions and polar decompositions in a $C^\ast$-algebra, J. Operator Theory, 17 (1987) 357--364.

\bibitem{Popa_vN} S. Popa, The relative Dixmier property for inclusions of von Neumann algebras of finite index, Ann. Sci. Ec. Norm. Sup., 32 (1999) 743--767.

\bibitem{Popa_C*} S. Popa, On the relative Dixmier property for inclusions of $C^*$-algebras, J. Funct. Anal., 171 (2000) 139--154.

\bibitem{Powers} R.T. Powers, Simplicity of the $C^*$-algebra associated with the free group on two generators, Duke Math. J., 42 (1975) 151--156.

\bibitem{Raum} S. Raum, $C^*$-simplicity of locally compact Powers groups, J. Reine Angew. Math., to appear, arXiv:1505.07793v2 (2016) 32 pp.

\bibitem{Riedel1} N. Riedel, On the Dixmier property of simple $C^*$-algebras, Math. Proc. Camb. Phil. Soc., 91 (1982) 75--78.

\bibitem{Riedel2} N. Riedel, The weak Dixmier property implies the Dixmier property, Dilation theory, Toeplitz operators and other topics, 11 (1983) 299--301.

\bibitem{RingPLMS78} J.R. Ringrose, Derivations of quotients of von Neumann algebras, Proc. London Math. Soc. (3), 36 (1978) 1--26.

\bibitem{RingPLMS84} J.R. Ringrose, On the Dixmier approximation theorem, Proc. London Math. Soc. (3), 49 (1984) 37--57.

\bibitem{Robert:Comparison} L. Robert, Nuclear dimension and $n$-comparison, M\"unster J. Math., 4 (2011) 65--71.

\bibitem{Robert} L. Robert, Nuclear dimension and sums of commutators, Indiana Univ. Math. J., 64 (2015) 559--576.

\bibitem{Rordam:JFA2} M. R\o rdam, On the structure of simple $C^*$-algebras tensored with a UHF-algebra. II, J. Funct. Anal., 107 (1992) 255--269.

\bibitem{Rordam:Book} M. R\o rdam,  Classification of nuclear, simple $C\sp *$-algebras. Classification of nuclear $C\sp *$-algebras. Entropy in operator algebras, 1--145, Encyclopaedia Math. Sci., 126, Springer, Berlin, 2002.

\bibitem{Rordam:Acta} M. R\o rdam, A simple $C^*$-algebra with a finite and an infinite projection, Acta Math., 191 (2003) 109--142.

\bibitem{Rordam:Z} M. R\o rdam, The stable and the real rank of {$\mathcal Z$}-absorbing $C^*$-algebras, Internat. J. Math., 15 (2004) 1065--1084.

\bibitem{Scarparo} E. Scarparo, Supramenable groups and partial actions, Ergodic Theory Dynam. Systems, to appear, arXiv:1504.08026 (2015) 17 pp.

\bibitem{Schwartz} J.T. Schwartz, $W^*$-algebras, Nelson, 1967.

\bibitem{SS} A.M. Sinclair and R.R. Smith, Finite von Neumann algebras and masas, London Math. Soc. Lecture Note Series 351, Cambridge University Press, Cambridge, 2008.


\bibitem{Skoufranis} P. Skoufranis, Closed convex hulls of unitary orbits in $C^*$-algebras of real rank zero, J. Funct. Anal., 270 (2016) 1319--1360.

\bibitem{Som:JOT} D.W.B. Somerset, The inner derivations and the primitive ideal space of a $C^*$-algebra, J. Operator Theory, 29 (1993) 307--321.

\bibitem{Som} D.W.B. Somerset, Inner derivations and primal ideals of $C^*$-algebras, J. London Math. Soc. (2), 50 (1994) 568--580.

\bibitem{DSprox} D.W.B. Somerset, The proximinality of the centre of a $C^*$-algebra, J. Approx. Theory, 89 (1997) 114--117.

\bibitem{Stampfli} J.G. Stampfli, The norm of a derivation, Pacific J. Math., 33 (1970) 737--748.

\bibitem{StratilaZsido} S. Str\u{a}til\u{a} and L. Zsid{\' o},  An algebraic reduction theory for
$W\sp *$-algebras. II, Rev. Roumaine Math. Pures Appl., 18 (1973), 407--460.


\bibitem{Takesaki-I} M. Takesaki, Theory of operator algebras. I, Springer-Verlag, New York-Heidelberg, 1979.

\bibitem{Thomsen} K. Thomsen, Homomorphisms between finite direct sums of circle algebras, Linear and Multilinear Algebra, 32 (1992) 33--50.

\bibitem{Thomsen95} K. Thomsen. Traces, unitary characters and crossed products by ${\bf Z}$. Publ. Res. Inst. Math. Sci. 31 (1995), no. 6, 1011--1029.


\bibitem{TWW} A. Tikuisis, S. White, and W. Winter, Quasidiagonality of nuclear $C^*$-algebras, Ann. of Math. (2), 185 (2016), 229--284.

\bibitem{Toms:Flat} A.S. Toms, Flat dimension growth for $C^*$-algebras.  J. Funct. Anal., 238 (2006), 678--708.

\bibitem{Toms:CMP} A.S. Toms, Comparison theory and smooth minimal $C^*$-dynamics. Comm. Math. Phys., 289 (2009), 401--433.

\bibitem{Tong} H. Tong, Some characterizations of normal and perfectly normal spaces.  Duke Math. J., 19 (1952) 289--292.

\bibitem{Vest} J. Vesterstr\o m, On the homomorphic image of the centre of a $C^*$-algebra, Math. Scand., 29 (1971) 134--136.



\bibitem{Villadsen1} J. Villadsen, Simple $C^*$-algebras with perforation, J. Funct. Anal., 154 (1998), 110--116.

\bibitem{Villadsen} J. Villadsen, On the stable rank of simple $C\sp *$-algebras, J. Amer. Math. Soc., 12 (1999), 1091--1102.

\bibitem{Winter:dr} W. Winter, Decomposition rank and $\mathcal Z$-stability, Invent. Math., 179 (2010) 229--301.

\bibitem{WinterZacharias} W. Winter and J. Zacharias, The nuclear dimension of $C^*$-algebras, Adv. Math., 224 (2010) 461--498.

\bibitem{Zsido} L. Zsido, Note on Dixmier's trace type sets in properly infinite $W^*$-algebras, Rev. Roumaine Math. Pures Appl., 19 (1974) 269--274.

\bibitem{Zsido-norm} L. Zsido, The norm of a derivation of a W$^*$-algebra, Proc. Amer.
Math. Soc., 38 (1973) 147--150.





\end{thebibliography}
\end{document}